\documentclass[12pt,a4paper,reqno]{amsart}
\usepackage[english]{babel}
\usepackage[applemac]{inputenc}
\usepackage[T1]{fontenc}
\usepackage{palatino}
\usepackage{amsmath}
\usepackage{amssymb}
\usepackage{amsthm}
\usepackage{amsfonts}
\usepackage{graphicx}
\usepackage[colorlinks = true, citecolor = black]{hyperref}
\title{Intrinsic Lipschitz graphs and vertical $\beta$-numbers in the Heisenberg group}
\address{University of Connecticut, Department of Mathematics}
\address{University of Jyv\"askyl\"a, Department of Mathematics and Statistics \newline (Current affiliation: University of Fribourg, Department of Mathematics)}
\address{University of Helsinki, Dept. of Mathematics and Statistics}
\author{Vasileios Chousionis, Katrin F\"assler and Tuomas Orponen}
\thanks{V.C. is supported by  the Simons Foundation via the Collaboration grant \emph{Analysis and dynamics in Carnot groups},  grant no.\  521845. K.F. is supported by the Academy of Finland through the grant \emph{Sub-Riemannian manifolds from a quasiconformal viewpoint}, grant number 285159. T.O. is supported by the Academy of Finland through the grant \emph{Restricted families of projections, and applications to Kakeya type problems}, grant number 274512. T.O. is also a member of the Finnish CoE in Analysis and Dynamics Research. }
\email{vasileios.chousionis@uconn.edu}
\email{katrin.faessler@unifr.ch}
\email{tuomas.orponen@helsinki.fi}

\newcommand{\R}{\mathbb{R}}
\newcommand{\Rn}{\mathbb{R}^n}
\newcommand{\V}{\mathbb{V}}
\newcommand{\N}{\mathbb{N}}

\newcommand{\C}{\mathbb{C}}
\newcommand{\Z}{\mathbb{Z}}

\newcommand{\calT}{\mathcal{T}}
\newcommand{\calL}{\mathcal{L}}

\newcommand{\calH}{\mathcal{H}}

\newcommand{\calG}{\mathcal{G}}

\newcommand{\calB}{\mathcal{B}}
\newcommand{\calF}{\mathcal{F}}

\newcommand{\calR}{\mathcal{R}}

\newcommand{\Hd}{\dim_{\mathrm{H}}}

\newcommand{\E}{\mathbb{E}}
\newcommand{\G}{\Gamma}

\newcommand{\diam}{\operatorname{diam}}

\newcommand{\dist}{\operatorname{dist}}

\newcommand{\W}{\mathbb{W}}
\newcommand{\He}{\mathbb{H}}
\newcommand{\CG}{\textup{CG}}
\newcommand{\Adm}{\operatorname{Adm}}

\newcommand{\ra}{\rightarrow}

\numberwithin{equation}{section}

\newtheorem{thm}{Theorem}[section]

\newtheorem{lemma}[thm]{Lemma}
\newtheorem{cor}[thm]{Corollary}
\theoremstyle{definition}

\newtheorem{proposition}[thm]{Proposition}
\theoremstyle{definition}
\newtheorem{claim}[thm]{Claim}

\theoremstyle{definition}
\newtheorem{definition}[thm]{Definition}
\theoremstyle{definition}
\newtheorem{rem}[thm]{Remark}
\newtheorem{remark}[thm]{Remark}

\begin{document}

\begin{abstract} The purpose of this paper is to introduce and study some basic concepts of quantitative rectifiability in the first Heisenberg group $\He$. In particular, we aim to demonstrate that new phenomena arise compared to the Euclidean theory, founded by G. David and S. Semmes in the 90's. The theory in $\He$ has an apparent connection to certain nonlinear PDEs, which do not play a role with similar questions in $\R^{3}$.

Our main object of study are the \emph{intrinsic Lipschitz graphs} in $\He$, introduced by B. Franchi, R. Serapioni and F. Serra Cassano in 2006. We claim that these $3$-dimensional sets in $\He$, if any, deserve to be called quantitatively $3$-rectifiable.
Our main result is that the intrinsic Lipschitz graphs satisfy a \emph{weak geometric lemma} with respect to \emph{vertical $\beta$-numbers}. Conversely, extending a result of David and Semmes from $\R^{n}$, we prove that a
$3$-Ahlfors-David regular subset in $\He$, which satisfies the weak geometric lemma and has \emph{big vertical projections}, necessarily has \emph{big pieces of intrinsic Lipschitz graphs}.
\end{abstract}

\maketitle

\tableofcontents

\section{Introduction}

\emph{Rectifiability} is a fundamental concept in geometric measure
theory. Rectifiable sets extend the class of surfaces considered in classical differential
geometry; while admitting a few edges and sharp corners, they are still smooth enough to support a rich theory of local analysis. However, for certain questions of more global nature -- the boundedness of singular integrals being the main example -- the notion of rectifiability is too qualitative.

In a series of influential papers around the year 1990, \cite{Da2,DS,DS2,DS3}, G. David and S. Semmes developed an extensive theory of \emph{quantitative rectifiability} in Euclidean spaces. One of their main objectives was to find geometric criteria to characterize the $m$-dimensional subsets of $\Rn$, $0 < m < n$, on which "nice" singular integral operators (SIO) are $L^{2}$-bounded. Here, "nice" refers to SIOs with smooth, odd Calder\'on-Zygmund kernels, the archetype of which is the Riesz kernel $x/|x|^{m+1}$, $x \in \Rn$. Notice that for $n = 2$ and $m=1$, the Riesz kernel essentially coincides with the Cauchy kernel $1/z, z \in \C,$ in the complex plane.

A motivation for the efforts of David and Semmes was the significance of SIOs for the problem of finding a geometric characterization of \emph{removable sets} for bounded analytic functions and Lipschitz harmonic functions. Due to seminal papers by David \cite{Da3}, David and P. Mattila \cite{DM}, and F. Nazarov, X. Tolsa and A. Volberg \cite{NToV1}, \cite{NToV2}, it is now known that these removable sets coincide with the purely $(n-1)$-unrectifiable sets in $\Rn$, i.e. the sets which intersect every  $\mathcal{C}^1$ hypersurface in a set of vanishing $(n-1)$-dimensional Hausdorff measure. The geometric characterization of removability, and its connections to geometric measure theory and harmonic analysis, has a very interesting history; we refer to the excellent survey by Volberg and V. Eiderman \cite{EV}, and to the recent book of Tolsa \cite{To1}.

The problem of characterizing removable sets for Lipschitz harmonic functions has a natural analogue outside the Euclidean setting in certain non-commutative Lie groups, of which the Heisenberg group is the simplest example. In this group, the role of the Euclidean Laplace operator is played by the {\it sub-Laplacian} and harmonic functions are, by definition, the solutions to the sub-Laplacian equation. The question of characterizing removability for Lipschitz harmonic functions was considered in \cite{CM} for Heisenberg groups $\He^n$ endowed with a sub-Riemannian metric. It was shown that sets with Hausdorff dimension lower than $2n+1$ are removable, while those with dimension higher than $2n+1$ are not. Moreover, there exist both removable and non-removable sets with Hausdorff dimension equal to $2n+1$. Hence, as in the Euclidean case, the dimension threshold for removable sets is $\Hd \He^{n} - 1 = 2n + 1$, where $\Hd \He^{n}$ denotes the Hausdorff dimension of $\He^n$. The results from \cite{CM} were extended in \cite{CMT} from Heisenberg groups to a larger class of Lie groups, the \emph{Carnot groups}. There exists a well developed theory of sub-Laplacians in this setting, see for instance the book \cite{BLU} by A. Bonfiglioli, E. Lanconelli and F. Uguzzoni.

In order to characterize removable sets for Lipschitz harmonic functions in $\R^{n}$, one has to characterize the sets on which the SIO associated with the Riesz kernel $x/|x|^n$ is bounded in $L^2$. In $\He^{1}$, one would need to face a SIO with kernel
$$K(p):=\left( \frac{x|z|^2+yt}{(|z|^4+t^2)^{3/2}}, \frac{y|z|^2-xt}{(|z|^4+t^2)^{3/2}} \right)$$
for $p=(z,t), z=x+iy \in \C, t \in \R$. At this point, the knowledge about the action of this SIO on $3$-dimensional subsets of $\He^{1}$ (i.e. subsets of co-dimension $1$) is extremely limited. In the present paper, we will not address this question further, but it motivates the study of quantitative rectifiability in $\He^{1}$.

The main purpose of the present paper is to initiate this study, and to introduce some new, relevant concepts in $\He=\He^1$. Our aims
are twofold. First, we demonstrate that some parts, at least, of
the Euclidean theory of quantitative rectifiability carry over to $\He$. To us, this gives hope
that -- some day in the distant future -- questions on the
boundedness of SIOs on subsets of $\He$ may be understood as well
as they currently are in $\R^{n}$. Our second aim is somewhat more
philosophical: we want to demonstrate that building a theory of
quantitative rectifiability in $\He$ is worth the effort. In
particular, the proofs are not, merely, technically challenging
replicas of their Euclidean counterparts. New phenomena appear. In
particular, proving our main result, Theorem \ref{main}, involved
studying non-smooth solutions of the (planar) non-linear PDE
\begin{equation}\label{burgersEquation} \partial_{y}\phi + \phi \partial_{t} \phi = c, \qquad c \in \R, \end{equation}
known as the (or rather "a") \emph{Burgers equation}. In proving the Euclidean counterpart of Theorem \ref{main}, such considerations are not required. At least to us, any connection between the innocent-looking statement of Theorem \ref{main}, and the PDE \eqref{burgersEquation}, was quite a surprise at first sight.

In the terminology of David and Semmes, Theorem \ref{main} is the \emph{weak geometric lemma} for certain subsets of $\He$, called \emph{intrinsic Lipschitz graphs}.
For now, we just briefly explain the meaning of these concepts;
precise definitions are postponed to Sections \ref{s:prelim} and \ref{s:BPiLG}. We consider two kinds of subgroups of $\He$: \emph{horizontal} and \emph{vertical}.
Writing $\He = \R^{2} \times \R$, the horizontal subgroups are lines through the origin inside $\R^{2} \times \{0\}$, while the vertical subgroups are planes spanned by a horizontal subgroup and the "$t$-axis" $\{0\} \times \R$.

In the present paper, we are mainly concerned with \emph{intrinsic graphs} over a fixed (but arbitrary) vertical subgroup $\W$, which we often take to be the "$(y,t)$-plane" $\W_{y,t} := \{(x,y,t) \in \He : x = 0\}$. Let $\V_{x}$ be the horizontal subgroup $\V_{x} = \{(x,0,0) : x \in \R\} \subset \He$, and consider a function $\phi \colon \W_{y,t} \to \V_{x}$. The intrinsic graph of $\phi$ (over $\W_{y,t}$) is the set
\begin{displaymath} \Gamma^{\phi} := \{w \cdot \phi(w) : w \in \W_{y,t}\} \subset \He, \end{displaymath}
where "$\cdot$" refers to the Heisenberg product (see Section \ref{s:prelim}).
Note that $\Gamma^{\phi}$ is, in general different, from the "Euclidean graph" $\{(\phi(y,t),y,t) : (y,t) \in \W_{y,t}\}$, and in fact $\Gamma^{\phi}$ need not be representable as the Euclidean graph of any function (see Example 2.5 in \cite{FSS}).

Recall that a function $f \colon \R^{2} \to \R$ is (Euclidean)
Lipschitz, if and only if there exists a cone, which, when centered
at any point $x$ on the graph of $f$, only intersects the graph at
$x$. The notion of \emph{intrinsic Lipschitz function} in $\He$ is
\textbf{defined} with this characterization in mind, with "graph"
replaced by "intrinsic graph", and "cone" replaced by a natural
$\He$-analogue, see \eqref{intrinsicCone}. Intrinsic Lipschitz
functions were introduced by B.\ Franchi, R.\ Serapioni and F.\
Serra Cassano in \cite{FSS}, and they turned out to be very influential in the evolution of geometric analysis in Heisenberg groups, see for instance \cite{ASV,ASCV,AS,BCSC, CMPS, FMS, FSS11, MSS}. Curiously, the definition
does not guarantee that an intrinsic Lipschitz function is
(metrically) Lipschitz between the spaces $\W_{y,t}$ and $\V_{x}$.

If the reader is familiar with the theory of rectifiability in metric spaces, but not with that in $\He$, she may wonder why such "intrinsic" notions are necessary in the first place; why cannot one study (metrically) Lipschitz images $\R^{k} \to \He$? The reason is simple: a Lipschitz image $f(\R^{k}) \subset \He$ has vanishing $k$-dimensional measure for $k \in \{2,3,4\}$. This is a result of L. Ambrosio and B. Kirchheim \cite{AK}. The work of Mattila, Serapioni and Serra Cassano \cite{MSS} and Franchi, Serapioni and Serra Cassano \cite{FSS01, FSS11} suggests that intrinsic Lipschitz graphs, instead, are the correct class of sets to consider in connection with $3$-rectifiability in the Heisenberg group. We believe that this is true also in the quantitative setting.

Lipschitz graphs in $\R^{n}$ are, arguably, the most fundamental examples of quantitatively rectifiable sets in the sense of David and Semmes. In the present paper, we propose that intrinsic Lipschitz graphs play the same role in $\He$. In $\R^{n}$, the term "quantitatively rectifiable" has many meanings; the fundamental results of David and Semmes show that a certain class of sets enjoys -- and can be characterized -- by a wide variety of properties, both geometric and analytic, each of which could be taken as the definition of "quantitatively rectifiable". In $\He$, no such results are available, yet, so we have to specify our viewpoint to "quantitative rectifiability". It will be that of "quantitative affine approximation". Theorem \ref{main}, informally stated, says that intrinsic Lipschitz graphs admit good affine approximations "at most places and scales".

The traditional way to quantify such a statement is via the notion of $\beta$-numbers, introduced by P. Jones in \cite{Joesc} in order to control the Cauchy singular integral on $1$-dimensional Lipschitz graphs. They were later used by Jones \cite{Jo1} and David and Semmes \cite{DS, DS2} in order to characterize quantitative rectifiability. The same approach works in $\He$, if the $\beta$-numbers are defined correctly. In Definition \ref{verticalBetas} below, we introduce the \emph{vertical $\beta$-numbers}. These nearly coincide with the usual (Euclidean) $\beta$-numbers, defined with respect to the metric in $\He$ of course: the single, crucial, difference is that approximating affine planes are restricted to sets of the form $z \cdot \W'$, where $z \in \He$ and $\W'$ is a vertical subgroup. Viewing $\He$ as $\mathbb{R}^{3}$ for a moment, these sets are simply (Euclidean) translates of the sets $\W'$. So, they are quite literally vertical planes.

Here is, finally, the main result:
\begin{thm}\label{main} An intrinsic Lipschitz graph in $\He$ satisfies the weak geometric lemma for the vertical $\beta$-numbers. \end{thm}
For a more precise restatement see Theorem \ref{wglintlg}. In brief, the weak geometric lemma states that, for any fixed $\varepsilon$, the vertical $\beta$-numbers of the graph have size at most $\varepsilon$ in all balls, except perhaps a family satisfying a Carleson packing condition (with constants depending on $\varepsilon$). This manner of quantifying the "smallness" of  an exceptional family of balls is ubiquitous in the theory of David and Semmes.

Theorem \ref{main} does not explain our need to define the vertical $\beta$-numbers; since the vertical $\beta$-numbers are, evidently, at least as large as the "usual" ones (with no restrictions on the approximating affine planes), the statement of Theorem \ref{main} merely becomes a little weaker, if the word "vertical" is removed. However, it turns out that the weak geometric lemma for the vertical $\beta$-numbers, combined with a condition on \emph{vertical projections}, essentially \textbf{characterizes} intrinsic Lipschitz graphs. This is the content of our second main result, a counterpart of a theorem of David and Semmes \cite{DS} from 1990:
\begin{thm}\label{main2} Assume that a $3$-regular set $E \subset \He$ satisfies the weak geometric lemma for the vertical $\beta$-numbers, and has big vertical projections. Then $E$ has big pieces of intrinsic Lipschitz graphs.\end{thm}

As before, we postpone explaining the notions of \emph{big vertical projections} (Definition \ref{BVP}) and \emph{big pieces of intrinsic Lipschitz graphs} (Definition \ref{BGiLG}). The latter condition does not guarantee that $E$ is an intrinsic Lipschitz graph (such a statement would be false, rather obviously). Instead, $E \cap B$ contains a measure-theoretically big piece of an intrinsic Lipschitz for every ball $B$ centered on $E$.

Finally, we mention that Theorem \ref{main2} admits a converse, which follows from Theorem \ref{main} by standard considerations, outlined at the end of the paper:
\begin{thm}\label{mainCor} Assume that a $3$-regular set $E \subset \He$ has big pieces of intrinsic Lipschitz graphs. Then $E \subset \He$ satisfies the weak geometric lemma for the vertical $\beta$-numbers, and has big vertical projections. \end{thm}
So, the property of having big pieces of intrinsic Lipschitz graphs is characterized by the combination of the weak geometric lemma for the vertical $\beta$-numbers, and the big vertical projections condition.

We remark that for $1$-dimensional sets (in contrast to $3$-dimensional sets in the present paper), quantitative rectifiability in $\He$ has been studied earlier. F. Ferrari, B. Franchi, and H. Pajot \cite{ffp},  N. Juillet \cite{jui} and recently S. Li and R. Schul \cite{sean1}, \cite{sean2}, considered the validity of the traveling salesman theorem of P. Jones \cite{Jo1}. Already in $\R^{n}$, there is a considerable difference in the techniques required to treat higher (than one) dimensional quantitative rectifiability. In $\He$, that difference is even more pronounced because of the result of Ambrosio and Kirchheim \cite{AK} mentioned earlier: $1$-dimensional rectifiable sets are, essentially, metric Lipschitz images of $\R$, whereas for $3$-dimensional sets one needs another approach. Finding such an approach is well-motivated: the critical dimension for the removability problem in $\He$ is $3$, and the development of quantitative rectifiability in this dimension is essential for making progress in that direction. 

The paper is organized as follows. In Section \ref{s:prelim} we
lay down the necessary background in the Heisenberg group and we also discuss intrinsic Lipschitz graphs and some of their main properties. In Section \ref{s:BPiLG}, we give sufficient conditions for a $3$-dimensional set in the first Heisenberg group
to have big pieces of intrinsic Lipschitz graphs. Section \ref{s:WGL} is devoted to the proof of Theorem \ref{main} and is definitely the most technical part of the paper. To facilitate the reader's navigation through the somewhat lengthy Section \ref{s:WGL} a second introductory part appears in Section \ref{outline}.

\textbf{Acknowledgements.} Part of the research for
this paper was completed while various subsets of the authors visited
the Universities of Bern, Helsinki and Connecticut. The
hospitality of these institutions is acknowledged.

\section{Preliminaries}\label{s:prelim}

The \emph{Heisenberg group} $\He$ is $\mathbb{R}^3$
endowed with the group law
\begin{equation}\label{eq:Heis}
(x,y,t)\cdot (x',y',t')=(x+x',y+y',t+t'+(xy'-yx')/2)
\end{equation}
for $(x,y,t),(x',y',t') \in \mathbb{R}^3$. We will sometimes
identify $\mathbb{R}^3$ with $\mathbb{C}\times \mathbb{R}$ and
denote points in the Heisenberg group by $(z,t)$ for $z= x+
\mathrm{i}y\in \mathbb{C}$ and $t\in \mathbb{R}$.

We use the following metric on $\He$:
\begin{displaymath}
d_{\mathbb{H}}:\He\times \He\to [0,\infty),\quad
d_{\mathbb{H}}(p,q):= \|q^{-1} \cdot p\|,
\end{displaymath}
where
\begin{displaymath}
\|(z,t)\|:= \max \left\{ |z|,|t|^{1/2}\right\}.
\end{displaymath}
The closed balls in $(\He, d_{\He})$ will be denoted by $B(x,r)$. We will also denote by $\mathcal{H}^s$ the $s$-dimensional Hausdorff measure in $(\He, d_{\He})$. The reader who is not familiar with the notion of Hausdorff measures should have a look at \cite{M}. For more information on the Heisenberg group, see for instance the book \cite{CDPT} by Capogna, Danielli, Pauls and Tyson. Here we just mention that  $\Hd \He = 4$, and the usual Lebesgue measure on $\R^{3}$ coincides (up to a constant) with $\calH^{4}$ on $\He$.

The distance $d_{\mathbb{H}}$ is invariant with respect to left
translations
\begin{displaymath}
\tau_{p}:\He \to \He,\quad \tau_{p}(q) = p \cdot q,\quad
(p\in \mathbb{H}),
\end{displaymath}
and homogeneous with respect to dilations
\begin{displaymath}
\delta_r:\He \to \He,\quad \delta_r((z,t))
=(rz,r^2 t),\quad (r>0).
\end{displaymath}

Recall that a closed set $E\subset \mathbb{H}$ is \emph{
$3$-(Ahlfors-David)-regular}, if there exists a constant $1\leq C
<\infty$, the \emph{regularity constant} of $E$, such that
\begin{displaymath}
C^{-1} r^3 \leq \mathcal{H}^3(B(x,r)\cap E) \leq C r^3
\end{displaymath}
for all $x\in E$ and $0<r\leq \mathrm{diam}(E)$.

We stress once more that metric concepts in $\He$, such as "ball", "Hausdorff measure" or "Ahlfors-David regular" are always defined with respect to $d_{\He}$, unless explicitly stated otherwise.

Identify $\He$ with $\C \times \R$ for a moment. If $V \subset \C$ is a line through the origin, then $\V := V \times \{0\}$ is called a \emph{horizontal subgroup of $\He$.} A \emph{vertical subgroup of $\He$} is a set of the
form $\W = V \times \R$, where $V \subset \C$ is a line through the origin. Note that both $\W$ and $\V$ are subgroups of $\He$, and closed under the action of $\delta_{r}$.

Under the identification of $\W$ with $\mathbb{R}^2$, the
subgroup $\W$ can be endowed with the
$2$-dimensional Lebesgue measure $\mathcal{L}^2$. This turns out to be a Haar measure on $(\W,\cdot)$, and it
agrees (up to a multiplicative constant) with $\calH^{3}$ on $(\W,d_{\mathbb{H}})$, see \cite[Proposition 2.20]{MSS}.

\begin{definition}[Complementary subgroups] Given a vertical subgroup $\W=V \times \mathbb{R}$ of
$\He$, we define the \emph{complementary horizontal subgroup}
\begin{displaymath} \V := \V_{\W} := V^{\bot} \times \{0\}, \end{displaymath}
where $V^{\bot}$ is the orthogonal
complement of $V$ in $\C$. Then every point $p \in \He$
can be written \textbf{uniquely} as $p= p_{\W}\cdot p_{\V}$ with $p_{\W} \in \W$ and $p_{\V} \in \V$. \end{definition}

One could also consider other splittings of the Heisenberg group, but in this paper we will always assume that the groups $\W$ and $\V_{\W}$ are \emph{orthogonal}; by this we mean that $V$ and $V^{\perp}$ are orthogonal as above.

\begin{definition}[Horizontal and vertical projections] Let $\W = V \times \R$ be a vertical subgroup with complementary horizontal subgroup $\V = V^{\perp} \times \{0\}$. As we observed, every point $p \in \He$ can be written uniquely as $p = p_{\W} \cdot p_{\V}$ with $p_{\W} \in \W$ and $p_{\V} \in \V$. This gives rise to the \emph{vertical projection} $\pi_{\W}$ and \emph{horizontal projection} $\pi_{\V}$, defined by
\begin{displaymath} \pi_{\W}(p) := p_{\W} \quad \text{and} \quad \pi_{\V}(p) := p_{\V}. \end{displaymath}
The mappings $\pi_{\W}$ and $\pi_{\V}$ have the following explicit formulae:
\begin{displaymath} \pi_{\V}(z,t) := (z,t)_{\V}=(\pi_{V^{\bot}}(z),0), \end{displaymath}
and
\begin{displaymath} \pi_{\W}(z,t) :=(z,t)_{\W}= (\pi_{V}(z),t - 2\omega(\pi_{V}(z),\pi_{V^{\perp}}(z))), \end{displaymath}
 for $(z,t) \in \C \times \R \cong \He$. Here $\pi_{V}$ and $\pi_{V^{\perp}}$ are the usual orthogonal projections onto the lines $V$ and $V^{\perp}$ in $\C$. We also used the abbreviating notation
 \begin{displaymath}
 \omega(z,z')= \tfrac{1}{4} \Im \left(\overline{z} z'\right)=\tfrac{1}{4}(xy'-yx')
 \end{displaymath}
 for $z = x + \mathrm{i}y$ and $z' = x' + \mathrm{i}y'$.
 \end{definition}

 The horizontal projections $\pi_{\V}$ are both (metrically) Lipschitz functions, and group homomorphisms. The vertical projections $\pi_{\W}$ are neither. However, as we will see many times in this paper, the projections $\pi_{\W}$ interact well with intrinsic Lipschitz graphs, defined below.

\begin{definition}[$\He$-cones] Let $\W$ be a vertical subgroup with complementary horizontal subgroup $\V$. An \emph{$\He$-cone perpendicular to $\W$ and with aperture $\alpha
> 0$} is the following set $C_{\W}(\alpha)$:
\begin{equation}\label{intrinsicCone} C_{\W}(\alpha) := \{p \in \He : \|p_{\W}\| \leq \alpha\|p_{\V}\|\}. \end{equation}
\end{definition}

\begin{definition}[Intrinsic Lipschitz graphs and functions]\label{d:intrLip}
 A subset $\Gamma \subset \He$ is an \emph{intrinsic $L$-Lipschitz graph over a vertical subgroup $\W$}, if
 \begin{equation}\label{coneCondition} (x \cdot C_{\W}(\alpha)) \cap \Gamma = \{x\} \quad \text{for } x \in \Gamma \text{ and } 0 < \alpha < \tfrac{1}{L}. \end{equation}
 If $A \subset \W$ is any set, and $\V$ is the complementary horizontal subgroup of $\W$, we say that a function $\phi \colon A \to \V$ is \emph{an intrinsic $L$-Lipschitz function}, if the \emph{intrinsic graph of $\phi$}, namely
 \begin{displaymath} \Gamma^{\phi} := \{w \cdot \phi(w) : w \in A\} \subset \He, \end{displaymath}
 is an intrinsic $L$-Lipschitz graph.

The \emph{intrinsic Lipschitz constant} of $\phi$ (or $\Gamma$) is defined as the infimum over all constants $L$ for which $\phi$ (or $\Gamma$) is intrinsic $L$-Lipschitz.
\end{definition}

For a nice picture of intrinsic Lipschitz graphs and the $\He$-cones, see Section 3 in \cite{FSS}.

\begin{remark}[Parametrisation of intrinsic Lipschitz graphs]\label{remk:para} An intrinsic Lipschitz graph can be uniquely parametrised by an intrinsic Lipschitz function. More precisely, given an intrinsic $L$-Lipschitz graph $\Gamma$ over a vertical subgroup $\W$, there exists a unique intrinsic $L$-Lipschitz function $\phi_{\Gamma} \colon \pi_{\W}(\Gamma) \to \V$ such that $\Gamma = \Gamma^{\phi_{\Gamma}}$.

To see this, one first checks that the cone condition \eqref{coneCondition} implies the injectivity of $\pi_{\W}|_{\Gamma}$. Indeed, assume that $\pi_{\W}(x) = \pi_{\W}(y)$ for some $x,y \in \Gamma$. Writing (uniquely) $x = x_{\W} \cdot x_{\V}$ and $y = y_{\W} \cdot y_{\V}$, this gives
 \begin{displaymath} \pi_{\W}(x^{-1} \cdot y) = \pi_{\W}(x_{\V}^{-1} \cdot y_{\V}) = 0, \end{displaymath}
 since $x_{\V}^{-1} \cdot y_{\V} \in \V$, and $\pi_{\W}$ annihilates $\V$. Hence, $x^{-1} \cdot y \in C_{\W}(\alpha)$ for any $\alpha > 0$, and so $y \in (x \cdot C_{\W}(\alpha)) \cap \Gamma$, implying $x = y$. Consequently, the following mapping $\phi_{\Gamma} \colon \pi_{\W}(\Gamma) \to \V$ is well-defined:
 \begin{displaymath} \phi_{\Gamma}(\pi_{\W}(x)) := \pi_{\V}(x). \end{displaymath}
The mapping $\phi_{\Gamma}$ clearly satisfies $\Gamma = \Gamma^{\phi_{\Gamma}}$, and thus $\phi_{\Gamma}$ is intrinsic $L$-Lipschitz by definition. The uniqueness of $\phi_{\Gamma}$ follows from the uniqueness of the representation $x = x_{\W} \cdot x_{\V}$, $x \in \He$. We refer to $\phi_{\Gamma}$ as the \emph{parametrisation of $\Gamma$}.
\end{remark}

A key property of intrinsic Lipschitz graphs is that they are invariant under left translations and dilations in $\He$; if $\Gamma$ is an intrinsic $L$-Lipschitz graph, then $\delta_{r}(\tau_{p}(\Gamma))$ is also an intrinsic $L$-Lipschitz graph for any $p \in \He$ and $r > 0$. This is why these sets are called "intrinsic"!

\begin{remark}\label{rotationalInvariance} We will often "without loss of generality" assume that the intrinsic Lipschitz functions and graphs under consideration are defined over the specific vertical plane $\W_{y,t} := \{(0,y,t) : y,t \in \R\}$. This is legitimate, because the notions are invariant under the rotations $R_{\theta}$ around the $t$-axis, given by $R_{\theta}(z,t) := (e^{i\theta}z,t)$. The rotations are both group homomorphisms and isometries with respect to the metric $d_{\He}$. The homomorphism property implies that $\pi_{R_{\theta}\W} \circ R_{\theta} = R_{\theta} \circ \pi_{\W}$, and then the isometry property gives the following: if $\Gamma$ is an intrinsic $L$-Lipschitz graph over $\W$, then $R_{\theta}\Gamma$ is an intrinsic $L$-Lipschitz graph over $R_{\theta}\W$.
\end{remark}

The class of intrinsic Lipschitz functions has (in greater generality)
been introduced and studied by Franchi, Serapioni and Serra
Cassano. The following facts are special cases of the results in \cite{FSS11} and \cite{FS}:
\begin{itemize}
\item For $A \subset \W$, an intrinsic $L$-Lipschitz function $\phi: A \to
\V$ can be extended to an intrinsic $L'$-Lipschitz function $\W
\to \V$, where $L'$ depends only on $L$.
\item An intrinsic
Lipschitz function $\phi:\W \to \V$ is intrinsically
differentiable $\mathcal{L}^2$ almost everywhere on $\W$.
\item An intrinsic $L$-Lipschitz graph over $\W$ is $3$-regular with regularity constant depending only on $L$.
\end{itemize}
We will write more about intrinsic differentiability and the related notion of \emph{intrinsic gradient} in Section \ref{ss:intr_grad}.

If $A,B > 0$, we will use the notation $A \lesssim_{p} B$ to signify that there exists a constant $C \geq 1$ depending only on $p$ such that $A \leq CB$. If the constant $C$ is absolute, we write $A \lesssim B$. The two-sided inequality $A \lesssim_{p} B \lesssim_{p} A$ is abbreviated to $A \sim_{p} B$.

\section{A sufficient condition for big pieces of intrinsic Lipschitz graphs}\label{s:BPiLG}

In this section, we start proving our two main results. To warm up, we begin with the significantly easier Theorem \ref{main2}: if a $3$-regular set $E \subset \He$ satisfies the weak geometric lemma for vertical $\beta$-numbers, and has big vertical projections (see definitions below), then $E$ has BPiLG. The argument is very similar to the Euclidean counterpart, due to David and Semmes \cite{DS}. In fact, the greatest surprise here is probably the similarity of the arguments itself: considering that the vertical projections $\pi_{\W} \colon \He \to \W$ are not Lipschitz, one might expect a rockier ride ahead.

We start with a few central definitions and auxiliary results.
\begin{definition}[BVP]\label{BVP}
 A $3$-regular set $E\subseteq \He$ is said to have \emph{big vertical projections} (BVP in short) if there exists a constant $\delta>0$ with the following property: for all $x\in E$ and for all $0 < R \leq \diam_{\He}(E)$ there exists a vertical subgroup $\W$ such that
 \begin{displaymath}
 \calL^{2}(\pi_{\W}(E\cap B(x,R))) \geq \delta R^3.
 \end{displaymath}
\end{definition}

\begin{definition}[BPiLG]\label{BGiLG}
 A $3$-regular set $E\subseteq \He$ has \emph{big pieces of intrinsic Lipschitz graphs} (BPiLG in short)
 if there exist constants $L \geq 1$ and $\theta > 0$ with the following property. For all $x\in E$ and $0 < R \leq \diam_{\He}(E)$, there exists an intrinsic $L$-Lipschitz graph
 $\Gamma$ over some vertical subgroup such that
 \begin{displaymath}
  \calH^{3}(E\cap \Gamma \cap B(x,R))\geq \theta R^3.
 \end{displaymath}
\end{definition}

\begin{definition}[Vertical $\beta$-numbers]
\label{verticalBetas}
 Let $E\subset \He$ be a set, let $B \subset \He$ be a ball with radius $r(B) > 0$ centered on $E$, let $\W$ be a vertical subgroup, and let $z \in \He$. We write
 \begin{displaymath} \beta_{E}(B;z \cdot \W) := \sup_{y \in B \cap E} \frac{\dist_{\He}(y,z \cdot \W)}{r(B)}, \end{displaymath}
and then we define the \emph{vertical $\beta$-number} as
 \begin{displaymath}
  \beta(B) := \beta_{E}(B) := \inf_{\W,z} \beta_{E}(B;z \cdot \W).
   \end{displaymath}
The infimum is taken over all vertical subgroups $\W$, and all points $z \in \He$.
\end{definition}

\begin{remark}\label{changeOfMidpoint} The following observation is useful, and not quite as trivial as its Euclidean counterpart.
Assume $B \subset \He$ is a ball centered on $E$, $\W$ is any
vertical subgroup, and $z \in \He$. Then
\begin{displaymath} \sup_{x,y \in B \cap E}\frac{\dist_{\He}(x,y \cdot \W)}{r(B)} \leq 2\beta_{E}(B;z \cdot \W). \end{displaymath}
To prove this, observe that $p \cdot \W = \W \cdot p$ for any point $p \in \He$. In particular, if $y \in B \cap E$, we have
\begin{displaymath} y \cdot \W = z \cdot \W \cdot z^{-1} \cdot y. \end{displaymath}
Hence, if further $x \in B \cap E$ and $w,w' \in \W$, we have
\begin{displaymath} \dist_{\He}(x,y \cdot \W) \leq d_{\He}(x,z\cdot w'\cdot w^{-1}\cdot z^{-1}\cdot y) \leq d_{\He}(x,z \cdot w') + d_{\He}(y,z \cdot w). \end{displaymath}
Since this holds for all $w,w' \in \W$, we have
\begin{displaymath} \frac{\dist_{\He}(x,y \cdot \W)}{r(B)} \leq 2\beta_{E}(B;z \cdot \W), \end{displaymath}
as claimed.
\end{remark}

\begin{definition}[WGL]\label{WGLDef}
 We say that a $3$-regular set $E\subseteq \He$ satisfies the \emph{weak geometric lemma for vertical $\beta$-numbers} (WGL in short), if
  \begin{equation*}
  \int_0^R \int_{E \cap B(x,R)}\chi_{\{(y,s)\in E \times \mathbb{R}_+:\; \beta(B(y,s))>\varepsilon\}}(y,s)  \,{d}\calH^3(y)\frac{{d}s}{s} \lesssim_{\varepsilon} R^3
 \end{equation*}
for all $\varepsilon > 0$, $x\in E$ and $R>0$.
\end{definition}

The following lemma shows that even if the vertical projections $\pi_{\W}$ are not Lipschitz, they still cannot increase $\calH^{3}$-measure (too much). This is rather surprising, as it is easy to find less than three-dimensional sets $E \subset \He$ such that $\Hd \pi_{\W}(E) > \Hd E$, see Example 4.1 in \cite{BDFMT}. 

\begin{lemma}\label{lipschitzLemma} Let $\W$ be a vertical subgroup in $\He$.
Then there exists a constant $0<C<\infty$ such that for all $A
\subseteq \He$, one has
\begin{equation}\label{eq:Lipmeas}\calL^{2}(\pi_{\W}(A)) \leq C \mathcal{H}^3(A).
\end{equation}
\end{lemma}

\begin{proof}
The lemma follows from \cite[Lemma
2.20]{FS}, which, when specialized to the Heisenberg group, states
that there is a constant $C >0$ such that
$\mathcal{L}^2(\pi_{\mathbb{W}}(B(p,r))=C r^3$ for all $p\in \He$
and $r>0$. See also \cite[Lemma 3.14]{FSS}.

Given a set $A\subset \He$ of positive $\mathcal{H}^3$-measure, choose a covering of $A$ by closed balls $B_i=B(p_i,r_i)$, $i \in \N$, such that
\begin{displaymath}
\sum_{i\in \N} r_i^3 \leq 2\mathcal{H}^3(A).
\end{displaymath}
Then, we find that
\begin{align*}
\calL^{2}(\pi_{\mathbb{W}}(A)) \leq \sum_{i\in \N}
\calL^{2}(\pi_{\mathbb{W}}(B_i)) = C \sum_{i\in \N} r_i^3 \leq 2C
\mathcal{H}^3(A).
\end{align*}
This completes the proof. \end{proof}

\subsubsection{David cubes}\label{davidCubes} We recall the construction of \emph{David cubes}, first introduced by David in \cite{Da2}, which can be defined on any regular set in a geometrically doubling metric space. Let $E \subset \He$ be a $3$-regular set. Then, there exists a family of partitions $\Delta_{j}$ of $E$, $j \in \Z$, with the following properties:
\begin{itemize}
\item[(i)] If $j \leq k$, $Q \in \Delta_{j}$ and $Q' \in
\Delta_{k}$, then either $Q \cap Q' = \emptyset$, or $Q \subset
Q'$. \item[(ii)] If $Q \in \Delta_{j}$, then $\diam_{\He} Q \leq 2^{j}$.
\item[(iii)] Every set $Q \in \Delta_{j}$ contains a set of the form
$B(z_{Q},c2^{j}) \cap E$ for some $z_{Q} \in Q$, and some constant $c >
0$.
\end{itemize}
The sets in $\Delta := \cup \Delta_{j}$ are called David cubes, or
just cubes, of $E$. For $Q \in \Delta_{j}$, we define $\ell(Q) :=
2^{j}$. Thus, by (ii), we have $\diam_{\He}(Q) \leq \ell(Q)$ for $Q \in \Delta$. Given a fixed cube $Q_{0} \in \Delta$, we write
\begin{displaymath} \Delta(Q_{0}) := \{Q \in \Delta : Q \subset Q_{0}\}. \end{displaymath}
It follows from (ii), (iii), and the $3$-regularity of $E$ that
$\calH^{3}(Q) \sim \ell(Q)^{3}$ for $Q \in \Delta_{j}$. It is an
easy fact, needed a bit later, that the following holds: if $x,y
\in E$ are distinct points, there exists an index $j \in \Z$, and
disjoint cubes $Q_{x},Q_{y} \in \Delta_{j}$, containing $x$ and
$y$, respectively, with the properties that $2^{j} \sim
d_{\He}(x,y)$,
\begin{displaymath} Q_{x} \subset B(z_{Q_{y}},4\ell(Q_{y})) \quad \text{and} \quad Q_{y} \subset B(z_{Q_{x}},4\ell(Q_{x})). \end{displaymath}
Indeed, let $j \in \Z$ be the largest integer such that $2^{j}
\leq d_{\He}(x,y)$, and let $Q_{x},Q_{y} \in \Delta_{j}$ be the unique cubes containing $x$ and $y$. Then $2^{j} \sim d_{\He}(x,y)$, and since
$d_{\He}(x,y) \leq 2\ell(Q_{x})$, we have $Q_{y} \subset
B(z_{Q_{x}},4\ell(Q_{x}))$. The same holds with the roles of
$x$ and $y$ reversed.

In the sequel, we write
\begin{displaymath} B_{Q} := B(z_{Q},4\ell(Q)), \qquad Q \in \Delta. \end{displaymath}

The weak geometric lemma (Definition \ref{WGLDef}) implies the
following reformulation in terms of David cubes. Write
\begin{displaymath} \beta(Q) := \beta(B_{Q}) \qquad Q \in \Delta. \end{displaymath}
Then
\begin{equation}\label{cubeWGL} \sum_{\{Q \in \Delta(Q_{0}) : \beta(Q) \geq \varepsilon\}} \calH^{3}(Q) \lesssim_{\varepsilon} \calH^{3}(Q_{0}) \end{equation}
for any $\varepsilon > 0$ and $Q_{0} \in \Delta$. Deriving this
property from Definition \ref{WGLDef} is an easy exercise using
the properties of David cubes, and we omit the details.

\subsection{Proof of Theorem \ref{main2}} Here is the statement of Theorem \ref{main2} once more:

\begin{thm}\label{t:main1_BPiLG} A $3$-regular set in $\He$ with BVP and satisfying the WGL has BPiLG.
\end{thm}

Let $\Delta$ be a system of David cubes on $E$, let $c,\varepsilon > 0$ be constants, and let $\W$ be a vertical
subgroup. Throughout this section, a cube $Q \in \Delta$ will be called \emph{good} (more precisely
$(c,\varepsilon,\W)$-good), if
\begin{equation}\label{good1} \calL^{2}(\pi_{\W}(Q)) \geq c\calH^{3}(Q), \end{equation}
and
\begin{equation}\label{good2} \beta(Q) \leq \varepsilon. \end{equation}
We outline the proof of Theorem \ref{t:main1_BPiLG}. The proof is divided into two parts, a geometric one, and an abstract one. The geometric part shows that good cubes $Q \in \Delta$ are already "almost" intrinsic Lipschitz graphs in the following sense: if $x \in Q$ and $y \in B_{Q} \cap E$ satisfy $d_{\He}(x,y) \sim \ell(Q)$, then $y \notin x \cdot C_{\W}(\alpha)$ for some small $\alpha > 0$.

The abstract part uses the WGL and BVP assumptions to infer (cutting a few corners here) that only a small fraction of $E \cap B(z,R)$, $z \in E$, meets $B_{Q}$ for some non-good cube $Q$. Hence, a large set $F \subset E \cap B(z,R)$ meets $B_{Q}$ only for good cubes $Q$. Unfortunately, this is not literally true, and additional (technical) considerations are needed. Ignoring these for now, we can complete the proof as follows. Fixing $x,y \in F$, we can use the discussion in Section \ref{davidCubes} to find a cube $Q \in \Delta$ with $x \in Q$, $y \in B_{Q}$ and $\ell(Q) \sim d_{\He}(x,y)$. Since $x,y \in F$, we know that $Q$ is a good cube, and it follows from the geometric part that $y \notin C_{\W}(\alpha)$. Consequently, $F \subset B(z,R) \cap E$ is an intrinsic Lipschitz graph.

\subsubsection{The geometric part}\label{s:DS_lem} The following lemma is our counterpart of Lemma 2.19 in David and Semmes' proof in \cite{DS}.

\begin{lemma}\label{l:DS_lem} Assume that $Q$ is a $(c,\varepsilon,\W)$-good cube,  $x \in Q$ and $y \in B_{Q} \cap E$ with $d_{\He}(x,y) \sim \ell(Q)$. Then $y \notin x \cdot C_{\W}(\alpha)$, if $\varepsilon$ is sufficiently small with respect to $c$, and $\alpha > 0$ is small enough (depending on the constants $c,\varepsilon$).
\end{lemma}
\begin{proof}

We start with a reduction to "unit scale". Assume that the statement of the lemma fails for certain parameters $c,\varepsilon,\W,\alpha$, and a certain $(c,\varepsilon,\W)$-good cube $Q \in \Delta$. By this, we mean that \eqref{good1} and \eqref{good2} hold for $Q$, yet $y \in x \cdot C_{\W}(\alpha)$ for some $x \in Q$ and $y \in B_{Q} \cap E$ with $d_{\He}(x,y) \sim \ell(Q) =: r$.

Consider the set $Q_{x,r} := \delta_{1/r}(x^{-1} \cdot Q)$, and observe that it, also, satisfies \eqref{good1}, since $\calL^{2}(\pi_{\W}(Q_{x,r})) = r^{-3}\calL^{2}(\pi_{\W}(Q))$ and $\calH^{3}(Q_{x,r}) = r^{-3}\calH^{3}(Q)$. The first equation is not altogether trivial, but it follows from the equations
\begin{displaymath}
\calL^{2}(\pi_{\W}(Q_{x,r}))=r^{-3} \calL^2(\pi_\W (x^{-1} \cdot Q))=r^{-3} \calL^2(\pi_\W (x^{-1} \cdot \pi_{\W}(Q))),
\end{displaymath}
and the fact that the mapping $P_p:\W \to \W$, $P_p(w) = \pi_{\W}(p \cdot w)$, has unit Jacobian for any fixed $p \in \He$, see the proof of Lemma 2.20 in \cite{FS}, or \eqref{jacobian} below.

Further, if $B_{Q_{x,r}} := \delta_{1/r}(x^{-1} \cdot B_{Q})$, then $\beta(Q_{x,r}) := \beta(B_{Q_{x,r}}) \leq \varepsilon$. Here, $\beta$ denotes the vertical $\beta$-number associated with $E_{x,r} := \delta_{1/r}(x^{-1} \cdot E)$. Note that $E_{x,r}$ is $3$-regular with the same constants as $E$. Finally, note that $0 \in Q_{x,r}$, and
\begin{displaymath} y_{x,r} := \delta_{1/r}(x^{-1} \cdot y) \in B_{Q_{x,r}} \cap E_{x,r} \cap C_{\W}(\alpha) \end{displaymath}
with $d_{\He}(0,y_{x,r}) \sim 1$. To sum up, if the lemma fails for $Q$, then we can construct another $3$-regular set $E_{x,r}$, and another good David cube $Q_{x,r}$ (for $E_{x,r}$) with $0 \in Q_{x,r}$ and $\ell(Q_{x,r}) = 1$, such that the lemma fails for $Q_{x,r}$. Thus, it suffices to prove the lemma for a David-cube $Q$ with the additional properties $0 \in Q$ and $\ell(Q) = 1$.

To this end, assume to the contrary that $y \in C_{\W}(\alpha)$ with $d_{\He}(0,y) \sim 1$. We will use this
to show that the entire projection $\pi_{\W}(Q)$ is contained in a small neighbourhood
of the $t$-axis $T$. This will violate \eqref{good1}. Somewhat abusing notation, we write
\begin{displaymath} \W^{\perp} := V^{\perp} \times \R. \end{displaymath}
Let $p \in \He$ and $\W'$ be such that
\begin{displaymath} \beta_{E}(B_{Q};p \cdot \W') \leq 2\beta(Q) \leq 2\varepsilon. \end{displaymath}
The first task is to show that the angle $\theta(\W',\W^{\perp})$
between $\W'=V'\times \R$ and $\W^{\perp}$ satisfies
\begin{equation}\label{eq:angle_planes}
\theta(\W',\W^{\perp}) \lesssim \alpha + \varepsilon.
\end{equation}
Write $y=(y_H,y_t)$. The plan is  to use the smallness of
$\beta_{E}(B_{Q};p \cdot \W')$ in order to find a point
$w'=:(w_H',w_t')$ on $\W'$, but close to $y$, such that
\begin{equation}\label{eq:angle}
|\pi_V(w_H')|\lesssim (\alpha+\varepsilon) |w_H'|.
\end{equation}
This proves that the angle between $V^{\perp}$ and $V'$ is
$\lesssim (\alpha+\varepsilon)$ and thus \eqref{eq:angle_planes},
as claimed. In order to show \eqref{eq:angle}, we first observe that the assumption of $y \in C_{\W}(\alpha)$ implies
\begin{equation}\label{form1} \max \left\{|\pi_{V}(y_{H})|, |y_{t} - 2\omega(\pi_{V}(y_{H}),\pi_{V^{\perp}}(y_{H}))|^{1/2}\right\} \lesssim \alpha|\pi_{V^{\perp}}(y_{H})|. \end{equation}
Recalling that $d_{\He}(0,y) \sim 1$, this is only possible if
\begin{equation}\label{form2} |y_{H}| \sim 1. \end{equation}
Indeed, we even have $|y_{H}|^{2} \geq |y_{t}|$; otherwise the left hand side of
\eqref{form1} can be bounded from below by $ |
y_t/2|^{1/2}$, while the upper bound is then $\lesssim \alpha |y_{H}| \leq \alpha |y_t|^{1/2}$. For small
 enough $\alpha$, this is impossible. Thus we may suppose
 \eqref{form2}.

Further, by Remark \ref{changeOfMidpoint} we find for all
$y' \in B_{Q} \cap E$ that
\begin{equation}\label{form33} \dist_{\He}(y',\W') \leq 8\beta_{E}(B_{Q};p \cdot \W')\ell(Q) \leq 16\varepsilon, \end{equation}
by our choice of $p$ and $\W'$. In particular, for $y'=y$,  there exists a vector $w' = (w'_{H},w'_{t})
\in \W'$ with
\begin{displaymath} |w'_{H}-y_{H}| \leq d_{\He}(y,w') \leq 16\varepsilon. \end{displaymath}
By \eqref{form2}, this gives $|w'_{H}| \sim 1$, and finally, using $y \in C_{\W}(\alpha)$,
\begin{displaymath} |\pi_{V}(w'_{H})| \leq |w'_{H} - y_{H}| + |\pi_{V}(y_{H})| \lesssim (\alpha + \varepsilon)|w_{H}'|. \end{displaymath}
This proves \eqref{eq:angle} and \eqref{eq:angle_planes}.

So, we know that
\begin{itemize}
\item[(i)] $Q$ is close to $\W'$ (by \eqref{form33}),
\item[(ii)] $\W'$ is close to $\W^{\perp}$ (by \eqref{eq:angle_planes}).
\end{itemize}
As we will next demonstrate, $\pi_{\W}(Q)$ is close to $\pi_{\W}(\W^{\perp}) = T$.

Indeed, since we do not care about the best constants here, we can finish the proof very quickly: let $\tau_{\alpha,\varepsilon} > 0$ be a number such that if $w' \in \W' \cap B(0,2)$, then $d_{\He}(w',w^{\perp}) \leq \tau_{\alpha,\varepsilon}$ for some $w^{\perp} \in \W^{\perp} \cap B(0,3)$. Recalling \eqref{eq:angle_planes}, we can pick $\tau_{\alpha,\varepsilon}$ arbitrarily small by choosing $\alpha,\varepsilon$ small enough. Now, if $y' \in Q$, then by \eqref{form33} and the triangle inequality, we have $d_{\He}(y',w^{\perp}) \leq 16\varepsilon + \tau_{\alpha,\varepsilon}$ for some $w^{\perp} \in \W^{\perp}\cap B(0,3)$. Since $\pi_{\W}$ is locally $1/2$-H\"older continuous, it follows that
\begin{displaymath} \dist_{\He}(\pi_{\W}(y'),T) \leq d_{\He}(\pi_{\W}(y'),\pi_{\W}(w^{\perp})) \lesssim d_{\He}(y',w^{\perp})^{1/2} \leq (16\varepsilon + \tau_{\alpha,\varepsilon})^{1/2}. \end{displaymath}
The same holds with $T$ replaced by $\pi_{\W}(B(0,3))$. Finally, the $\calL^{2}$-measure of the $C(16\varepsilon + \tau_{\alpha,\varepsilon})^{1/2}$-neighbourhood of $T \cap \pi_\W(B(0,3))$ is bounded by a constant depending only on $\alpha,\varepsilon$, and this constant tends to zero as $\alpha,\varepsilon \to 0$. For sufficiently small values of $\alpha,\varepsilon$, this violates \eqref{good1}, and the proof is complete. \end{proof}

\subsubsection{The abstract part}\label{ss:application}

In this section, we apply Lemma \ref{l:DS_lem} to good cubes
inside a set $E\subset \He$ satisfying the weak geometric
lemma for vertical $\beta$-numbers. This is a counterpart
for Theorem 2.11 in \cite{DS}, which in turn is modelled on a result of P.\ Jones \cite{J}. The proof below is very similar to that in \cite{DS}; given Lemmas \ref{lipschitzLemma} and \ref{l:DS_lem}, the argument does not really see the difference between $\He$ and $\R^{n}$. We still record the full details.

\begin{thm}\label{t:main2_inv}
 Assume that $E\subseteq \He$ is a $3$-regular set satisfying the WGL and let $b >0$. Then there exist numbers $\alpha > 0$ and $M \in \N$, depending only on $b$ and the $3$-regularity and WGL constants of $E$, such that the following holds:

For every David cube $Q_0$ in $E$ and for all vertical projections
$\pi_{\W}$, there exist intrinsic $(1/\alpha)$-Lipschitz graphs $F_j \subset Q_{0}$, $1\leq j\leq M$,
\begin{displaymath} \calL^{2}(\pi_{\W}(Q_0 \setminus \cup F_j)) \leq b \calH^3(Q_0). \end{displaymath}
\end{thm}

\begin{proof}
Let $E$ and $b$ be as in the assumptions of Theorem
\ref{t:main2_inv}. Let further $\varepsilon>0$ be a small number
to be chosen later (based on Lemma \ref{l:DS_lem}). Fix an arbitrary cube $Q_{0} \in \Delta$, and an arbitrary vertical
subgroup $\W$.

First, we will group the cubes in $\triangle(Q_0)$ into "good"
and "bad" cubes, and control the quantity of the bad cubes via the WGL assumption. Second, Lemma \ref{l:DS_lem}, coupled with a "coding argument", will be used to partition the complement of the "bad" cubes in $Q_{0}$ into the sets $F_{j}$.

The "good" cubes $\calG$ are the familiar $(b/2,\varepsilon,\W)$-good cubes defined right above Lemma \ref{l:DS_lem}. The class
$\mathcal{B}_1$ consists of those maximal (hence disjoint) cubes in $Q_0$ that violate the
first goodness condition, i.e.,
\begin{displaymath}
\mathcal{B}_1:= \left\{Q\in \triangle(Q_0): \calL^{2}(\pi_{\W}(Q)) <
\tfrac{b}{2}\calH^3(Q)\right\}.
\end{displaymath}
Let $\mathcal{B}_2$ be the class of (all, not maximal) cubes that violate the second goodness
condition:
\begin{displaymath}
\mathcal{B}_2:= \left\{Q\in \triangle(Q_0): \beta(Q) >
\varepsilon\right\}.
\end{displaymath}

Then, clearly, $\mathcal{G} = \triangle(Q_0) \setminus
\bigcup_{j=1}^2 \mathcal{B}_j$. It is also clear  that the projections of bad cubes from the first
class have small measure: for $R_{1} := \bigcup_{Q \in \calB_{1}} Q$, we have
\begin{equation}\label{eq:est_B1}
\calL^{2}(\pi_{\W}\left(R_{1}\right))\leq
\tfrac{b}{2} \cdot \calH^3\left(R_{1}\right) \leq
\tfrac{b}{2} \cdot \calH^3(Q_0).
\end{equation}
On the other hand, for the second bad class, one can control the
measure of the cubes directly by the variant of WGL formulated in \eqref{cubeWGL}:
\begin{displaymath}
\sum_{Q \in \mathcal{B}_2} \calH^3(Q) \leq C(\varepsilon)
\calH^3(Q_0).
\end{displaymath}

Since $\calL^{2}(\pi_{\W}(A)) \leq C\calH^{3}(A)$ for all $A \subset \He$ by Lemma \ref{lipschitzLemma}, the inequality above shows that the $\pi_{\W}$-projection of $\bigcup_{Q \in \calB_{2}} Q$ has measure no larger than $CC(\varepsilon)\calH^{3}(Q_{0})$. This is a little bit too weak for our purposes; in analogy with \eqref{eq:est_B1}, we wish to replace $CC(\varepsilon)$ by a small constant. To this end, we set
\begin{displaymath}
R_2 = \left\{x\in Q_0:\; \sum_{Q\in\mathcal{B}_2}\chi_{B_{Q}}(x)
\geq N \right\}
\end{displaymath}
where $N=N_{b,\varepsilon}$ is so large that
$\calH^3(R_2)\leq \frac{b}{2C}\calH^3(Q_0)$. This is possible:
\begin{align*}
N \calH^3(R_2)\leq \int_{Q_0} \sum_{Q\in \mathcal{B}_2}\chi_{B_{Q}}(x) \;{d}\calH^3(x) \lesssim \sum_{Q\in
\mathcal{B}_2}\calH^3(Q) \leq C(\varepsilon) \cdot
\calH^3(Q_0).
\end{align*}
With this definition of $R_{2}$, Lemma \ref{lipschitzLemma} gives
\begin{displaymath}
\calL^{2}(\pi_{\W}(R_1\cup R_2)) \leq \tfrac{b}{2} \cdot \calH^3(Q_0)+ \tfrac{b}{2} \cdot
\calH^3(Q_0) \leq b \calH^3(Q_0).
\end{displaymath}

It remains to find subsets $F_1,\ldots, F_M$ such that $Q_0
\setminus (R_1 \cup R_2)=\bigcup F_j$ and for every $j=1,\dots, M$ and every pair $x,y\in F_j, x\neq y,$ it holds that
$y\notin x \cdot C_{\W}(\alpha)$ for $\alpha$ small enough (only depending on $b$ and the $3$-regularity and WGL constants of $E$). This is done via a "coding argument", which goes back to Jones, see \cite[p.866-867]{DS}. The argument is also explained briefly in David's book \cite{Da}, p. 81--82, but we present the full details.

We start with a brief informal overview. Write $F := Q_{0} \setminus (R_{1} \cup R_{2})$. Why do we need a "coding argument"? Maybe we can show, directly, that if $x,y\in F, x\neq y,$ then
$y\notin x \cdot C_{\W}(\alpha)$? Pick two distinct points $x,y \in F$, and pick two disjoint cubes $Q_{x},Q_{y} \subset Q_{0}$ of some common generation such that $x \in Q_{x}$, $y \in Q_{y}$, $Q_{y} \subset B_{Q_{x}}$ and $d_{\He}(x,y) \sim \ell(Q)$ (such cubes exist, as discussed in Section \ref{davidCubes}). Now, since $Q_{x} \not\subset R_{1}$, we know that $\calL^{2}(\pi_{\W}(Q_{x})) \geq \tfrac{b}{2} \cdot \calH^{3}(Q_{x})$. If -- and this is the "big if" -- we also knew that $\beta(Q_{x}) \leq \varepsilon$ for $\varepsilon > 0$ small enough, we could infer from Lemma \ref{l:DS_lem} that $y \notin x \cdot C_{\W}(\alpha)$ (assuming also that $\alpha$ is small enough).
Of course, we do not know that $\beta(Q_{x}) \leq \varepsilon$ for the particular cube $Q_{x}$ we are interested in: even though $x \notin R_{2}$, there can still be up to $N$ "exceptional" cubes $Q \ni x$ such that $\beta(Q) > \varepsilon$. The "coding argument" is needed to fix this issue. Essentially, if we declare that $x \in F_{1}$, say, we want to make sure that the following holds: whenever $Q \ni x$ with $\beta(Q) > \varepsilon$, then all the points $y \in B_{Q}\setminus Q$ are stored safely away in the other sets $F_{j}$, $j \geq 2$. Once that has been accomplished, the argument above works for $F_{1}$ (or any $F_{j}$) in place of $F$.

We turn to the details, which are repeated from \cite{Da} nearly verbatim. For each cube $Q \subset Q_{0}$, we will associate a certain finite sequence of $0$'s and $1$'s, denoted by $\sigma(Q)$. The length of such a sequence is denoted by $|\sigma(Q)|$. We declare $\sigma(Q_{0})$ to be the empty sequence.

Next, assume inductively that the numbers $\sigma(Q)$ have been defined for the descendants of $Q \subset Q_{0}$ down to a certain generation, say $k$. We now aspire to make the definition for cubes $Q$ of generation $k + 1$. If $Q$ is such a cube, and $Q^{\ast}$ is its parent, we initially set $\sigma(Q) := \sigma(Q^{\ast})$.

Assume that $Q \in \calB_{2}$, that is, $\beta(Q) > \varepsilon$, and assume that there exists at least one other cube $Q_{1}$ of the same generation as $Q$ such that $Q_{1} \subset B_{Q}$ (if either of these assumptions fails, we do not alter $\sigma(Q)$ now). Note that $\sigma(Q)$ and $\sigma(Q_{1})$ have both been initially defined.

There are two cases to consider.
\begin{itemize}
\item Case 1: $|\sigma(Q)| = |\sigma(Q_{1})|$. If $\sigma(Q) \neq \sigma(Q_{1})$, we do not alter $\sigma(Q)$ or $\sigma(Q_{1})$. But if $\sigma(Q) = \sigma(Q_{1})$, we re-define $\sigma(Q)$ by adding a "$0$", and we re-define $\sigma(Q_{1})$ by adding a "$1$".
\item Case 2: $|\sigma(Q)| \neq |\sigma(Q_{1})|$. If, for instance, $|\sigma(Q)| > |\sigma(Q_{1})|$, then we do not alter $\sigma(Q)$. But we re-define $\sigma(Q_{1})$ by adding either "$0$" or "$1$" to it in such a fashion that the new $\sigma(Q_{1})$ is \textbf{not} an initial segment of $\sigma(Q)$. Finally, if $|\sigma(Q)| < |\sigma(Q_{1})|$, then we repeat the same step with the roles of $Q$ and $Q_{1}$ reversed.
\end{itemize}

After this procedure is complete, we pick another cube $Q_{2} \subset B_{Q}$ of generation $k + 1$ (if it exists), and perform the previous case chase with the pair $(Q,Q_{1})$ replaced by $(Q,Q_{2})$. Once all the pairs $(Q,Q_{i})$ with $Q_{i} \subset B_{Q}$ have been processed, we move on to other pairs $(Q',Q_{i})$ with $Q' \in \calB_{2}$ and $Q_{i} \subset B_{Q'}$, and give them the same treatment as above.

The algorithm terminates eventually (because there are only finitely many cube-pairs to consider), and, at the end, every cube $Q$ of generation $k + 1$ has an associated sequence $\sigma(Q)$. If $Q \not\subset B_{Q'}$ for all cubes $Q' \in \calB_{2}$ of generation $k + 1$, then $\sigma(Q)$ retains the initial value $\sigma(Q^{\ast})$. Even if $Q \subset B_{Q'}$ for some $Q' \in \calB_{2}$, this can occur only for a bounded number, say $C'$, of alternatives $Q' \in \calB_{2}$ of generation $k + 1$. Consequently, $\sigma(Q)$ differs from $\sigma(Q^{\ast})$ by a sequence of length $\leq C'$.

By applying the procedure at all generations $k$, every sub-cube of $Q_{0}$ gets associated with a (finite) sequence $\sigma(Q)$. Next, we wish to extend the definition of these sequences from cubes to points in $F = Q_{0} \setminus (R_{1} \cup R_{2})$. Fix $x \in F$, and let $Q_{0} \supset Q_{1} \supset \ldots$ be the unique sequence of dyadic cubes converging to $x$. As discussed in the previous paragraph, $\sigma(Q_{j + 1})$ can differ from $\sigma(Q_{j})$ only in case $Q_{j + 1}$ is contained in $B_{Q'}$ for some $Q' \in \calB_{2}$ of the same generation as $Q_{j + 1}$. By definition of $x \notin R_{2}$, there can only be $< N$ such indices $j$. In particular, the sequences $\sigma(Q_{j})$ converge to some finite sequence of $0$'s and $1$'s, denoted by $\sigma(x)$. Furthermore, for those $< N$ indices $j$, where $\sigma(Q_{j + 1})$ possibly differs from $\sigma(Q_{j})$, this difference is a sequence of length at most $C'$. Consequently, the possible values of $\sigma(x)$ form a finite set $S$, whose cardinality can be bounded from above in terms of the constants $N$ and $C'$. Given an element $s \in S$, we now define
\begin{displaymath} F_{s} := \{x \in F : \sigma(x) = s\}. \end{displaymath}
It remains to check that the sets $F_{s}$ satisfy the useful property we hinted at in the informal overview. Assume that $x,y \in F$ belong to "nearby" two cubes $Q_{x},Q_{y}$ of the same generation, namely with $Q_{x} \subset B_{Q_{y}}$ and $Q_{y} \subset B_{Q_{x}}$, and assume that either $\beta(Q_{x}) > \varepsilon$ or $\beta(Q_{y}) > \varepsilon$. Then, we claim that $x$ and $y$ belong to two different sets of the form $F_{s}$. Consider the sequences $\sigma(Q_{x})$ and $\sigma(Q_{y})$ (which are initial sequences in $\sigma(x)$ and $\sigma(y)$, respectively). Assume, for instance, that $\beta(Q_{x}) > \varepsilon$. This means that $Q_{x} \in \calB_{2}$, so the pair $(Q_{x},Q_{y})$ is considered while defining the sequences $\sigma(Q_{x})$ and $\sigma(Q_{y})$. Then, inspecting Case 1 and Case 2, it is clear that neither of the sequences $\sigma(Q_{x})$ and $\sigma(Q_{y})$ can be an initial sequence of the other. This proves that $\sigma(x) \neq \sigma(y)$, as claimed.

Now, we can quickly prove that if $x,y\in F_s, x\neq y,$ then
$y\notin x \cdot C_{\W}(\alpha)$, if $\alpha > 0$ is small enough. Pick $x,y \in F_{s}$, and let $Q_{x}$ and $Q_{y}$ be sub-cubes of $Q_{0}$, containing $x$ and $y$, respectively, with same generation, satisfying $Q_{x} \subset B_{Q_{y}}$ and $Q_{y} \subset B_{Q_{x}}$, and with $\ell(Q) \sim d_{\He}(x,y)$. It follows from the claim in the previous paragraph that $\beta(Q_{x}) \leq \varepsilon$ and $\beta(Q_{y}) \leq \varepsilon$. Consequently, by Lemma \ref{l:DS_lem}, we have $y \notin x \cdot C_{\W}(\alpha)$  for small enough $\alpha$. This completes the proof of the theorem.
\end{proof}
We are now prepared to prove Theorem \ref{t:main1_BPiLG}. Again, the proof is very similar to the Euclidean argument, see Theorem 1.14 in \cite{DS}.

\begin{proof}[Proof of Theorem \ref{t:main1_BPiLG}]
Let $E$ be a $3$-regular subset of $\He$ with BVP, and satisfying the WGL. The former property ensures that for fixed $x\in E$ and $0 < R \leq \diam_{\He}(E)$, there is a vertical subgroup $\W$ such that
\begin{displaymath}
\calL^{2}(\pi_{\W}(E\cap B(x,\tfrac{R}{2}))) \geq \delta
\left(\tfrac{R}{2}\right)^3,
 \end{displaymath}
 where $\delta$ is a constant depending only on $E$.
 Then, there exists $b > 0$, depending only on $\delta$ and the $3$-regularity constant of $E$, and a David cube $Q_0 \subset E \cap B(x,R)$ such that $\calH^{3}(Q_{0}) \sim R^{3}$ and
 \begin{equation}\label{eq:proj_meas}
 \calL^{2}(\pi_{\W}(Q_0)) \geq 2 b \calH^3(Q_0).
 \end{equation}

 Now we apply Theorem \ref{t:main2_inv} to this particular cube $Q_0$ and choice of $b > 0$. It follows that there exist numbers $\alpha > 0$ and $M \in \N$ (depending only on $b$, and the $3$-regularity and WGL constants of $E$) with the following property: there exist intrinsic $(1/\alpha)$-Lipschitz graphs $F_j \subset Q_{0}$, $1\leq j\leq M$, such that
\begin{displaymath} \calL^{2}(\pi_{\W}(Q_0 \setminus \cup F_j)) \leq b \calH^3(Q_0). \end{displaymath}
This, together with \eqref{eq:proj_meas}, implies
 \begin{displaymath}
\sum_{j = 1}^M \calL^{2}(\pi_{\W}(F_j)) \geq \calL^{2}\left(\pi_{\W}\left(\bigcup_{j=1}^M F_j\right)\right) \geq
\calL^{2}(\pi_W(Q_0))-b \calH^3(Q_0) \geq b \calH^3(Q_0).
 \end{displaymath}
 Thus there must exist some $1\leq j\leq M$ such that
 \begin{displaymath}
\calH^3(F_j) \gtrsim \calL^{2}(\pi_{\W}(F_j)) \geq b \calH^3(Q_0)/M \sim bR^{3}/M.
 \end{displaymath}
The proof is complete. \end{proof}

\section{The weak geometric lemma for intrinsic Lipschitz graphs}\label{s:WGL}

\subsection{Introduction, part II}\label{outline}

In the first half of the paper, we saw that any $3$-Ahlfors-David
regular subset of $\He$ with big vertical projections (BVP), and
satisfying the weak geometric lemma (WGL), has big pieces of
intrinsic Lipschitz graphs (BPiLG). The second half of the paper is
devoted to proving the converse. The fact that intrinsic Lipschitz graphs (and thus BPiLG as well) have BVP is almost
trivial, see Remark \ref{remark0}. Hence in order to prove Theorem \ref{mainCor} it suffices to prove the WGL for sets which have BPiLG.

\begin{thm}\label{wglbp} If $E \subset \He$ has big pieces of intrinsic Lipschitz graphs with constants $L\geq 1, \theta>0,$ then it satisfies the weak geometric lemma; i.e.
\begin{equation*}
\label{wgleqbp}
\int_{0}^{R} \int_{E \cap B(x,R)} \chi_{\{(y,s)\in E \times
\mathbb{R}_+:\; \beta(B(y,s))\geq \varepsilon\}}(y,s)
d\calH^{3}(y) \, \frac{ds}{s} \lesssim_{\varepsilon, L, \theta} R^{3}
\end{equation*}
for any $\varepsilon > 0$, $x \in E$ and $R > 0$.
\end{thm}

Theorem \ref{wglbp} will follow using standard arguments, recalled  at the end of the paper, once we have at our disposal the WGL for intrinsic Lipschitz graphs. Therefore the rest of the paper will be devoted to the proof of Theorem \ref{main}, which is precisely stated below.

\begin{thm}\label{wglintlg} Let $\Gamma$ be an intrinsic
$L$-Lipschitz graph.  Then
\begin{displaymath}
\int_{0}^{R} \int_{\Gamma \cap B(x,R)} \chi_{\{(y,s)\in \Gamma
\times \mathbb{R}_+:\; \beta(B(y,s))\geq \varepsilon\}}(y,s)
d\calH^{3}(y) \, \frac{ds}{s} \lesssim_{\varepsilon, L} R^{3}
\end{displaymath}
for any $\varepsilon > 0$, $x \in \Gamma$ and $R > 0$.
\end{thm}

The proof of Theorem \ref{wglintlg} in the Euclidean case is relatively
straightforward, and can be carried out as follows. Assume that $f
\colon \R^{2} \to \R$ is an (entire) $L$-Lipschitz function, and
let $\Delta$ be a system of dyadic squares on $\R^{2}$. Assume
that $Q \in \Delta$ is such that $f$ is far from affine in $Q$, in
the sense that $\sup_{x \in Q} |f(x) - A(x)| \geq \varepsilon
\ell(Q)$ for all affine functions $A \colon \R^{2} \to \R$. Then,
it is fairly easy to verify that the gradient of $f$ must
fluctuate significantly near $Q$: there exists a fairly large
sub-cube $Q' \subset Q$ such that $|\E_{Q} (\nabla f) - \E_{Q'}
(\nabla f)| \geq \delta$, where $\delta$ only depends on
$\varepsilon$ and $L$, and $\E_{Q} \nabla f$ denotes the average
of $\nabla f$ over $Q$. The WGL follows from this observation,
plus the fact that $\|\nabla f\|_{L^{2}(Q_{0})} \lesssim_{L}
|Q_{0}|$ for any fixed cube $Q_{0} \in \Delta$ (for more details
on this final step, see the argument after \eqref{form15}).

It is a reasonable first thought that a similar argument should
work for intrinsic Lipschitz functions $\phi \colon \W \to \V$.
After all, for such a function, there is a concept of
an \emph{intrinsic gradient} $\nabla^{\phi} \phi$ (see Section
\ref{ss:intr_grad} below), which is known to exist at almost every
point on $\W$, and moreover $\nabla^{\phi} \phi \in
L^{\infty}(\W)$. So, if it were the case that the local
"non-affinity" of $\phi$ forces $\nabla^{\phi} \phi$ to fluctuate
noticeably, one could wrap up the argument in the fashion outlined
above. However, this is simply not true: in any bounded domain
$\Omega \subset \W$, the equation $\nabla^{\phi} \phi = 0$ admits
(smooth) non-affine solutions! For instance, the function
\begin{displaymath}
\phi:(-1,+\infty)\times \R \to \R,\quad  \phi(y,t)= \tfrac{t}{y+1}
\end{displaymath}
satisfies $\nabla^{\phi}\phi \equiv 0$ on its domain. In fact, even non-smooth continuous solutions are possible: the function $\phi \colon (-1,1)^{2} \to \R$, discussed in \cite[Remark 4.4.2]{B} and defined by
\begin{equation}\label{exFunction} \phi(y,t) := \begin{cases} \frac{t}{y + 1}, & t \geq 0,\\ \frac{t}{y - 1}, & t < 0, \end{cases} \end{equation}
satisfies $\nabla^{\phi} \phi = 0$  on $(-1,1)^{2}$, but it is not $\mathcal{C}^{1}$. The intrinsic graph of $\phi$ over the $(-1,1)^{2}$ is depicted in Figure \ref{fig1}.
\begin{figure}[h!]
\begin{center}
\includegraphics[scale = 0.4]{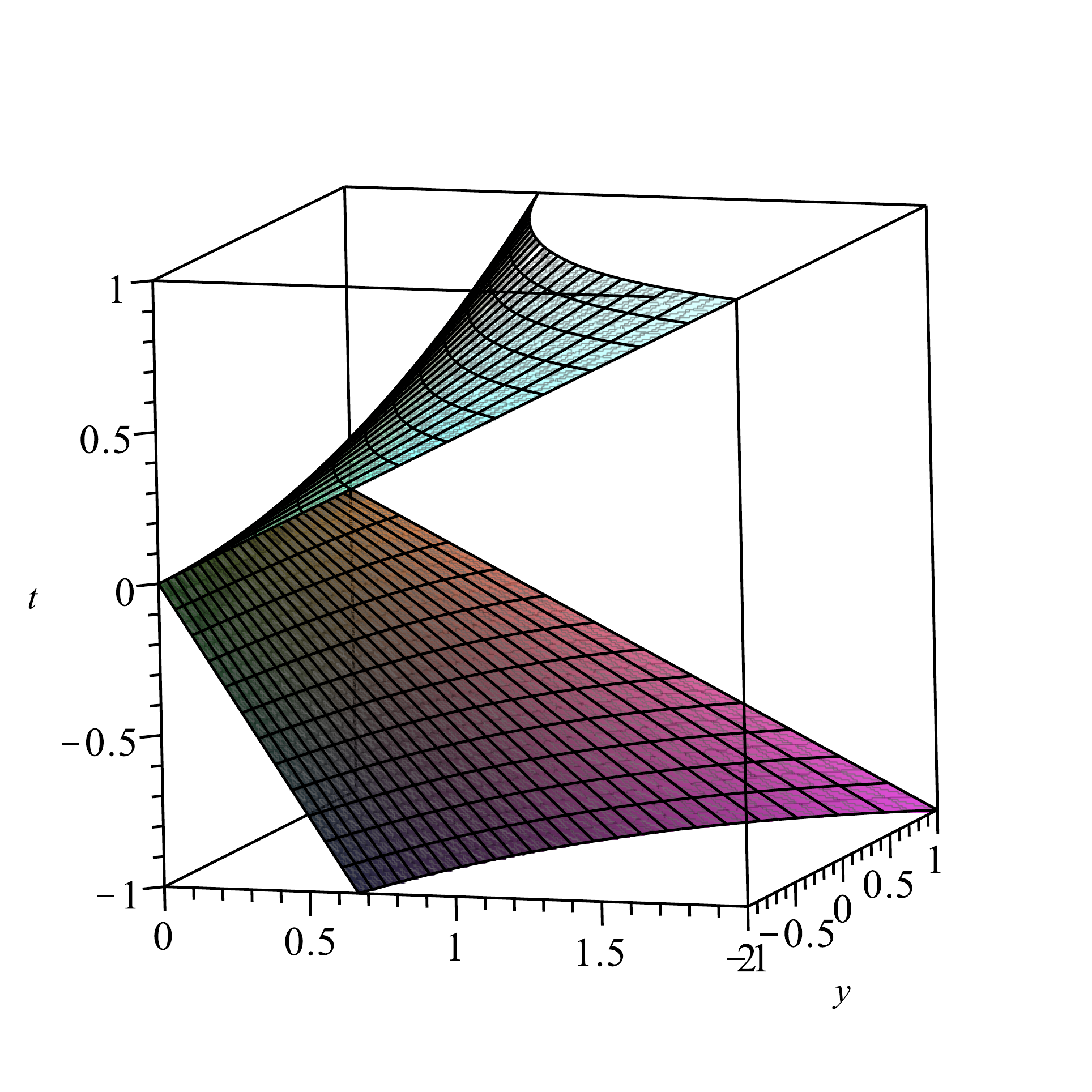}
\caption{The intrinsic graph of the function in
\eqref{exFunction} over the $(y,t)$-plane.}\label{fig1}
\end{center}
\end{figure}
Further examples of similar flavour have been discussed in \cite{ASV} in connection with
minimal surfaces in the Heisenberg group. We emphasise that the graphs of non-affine solutions to $\nabla^{\phi} \phi = 0$ have plenty of non-zero $\beta$-numbers, but this behaviour is not
registered by the fluctuation of $\nabla^{\phi} \phi$.

What can be done? We still want to use the intrinsic gradient, so
we need to invent a condition to replace "non-affinity", which
forces $\nabla^{\phi} \phi$ to fluctuate locally. It turns out
that the right notion is "being far from constant-gradient (CG)
graphs". The following (informal) statement may sound almost
tautological, but it requires a fair amount of work to verify: if
$\phi$ deviates locally from all (locally defined) CG intrinsic
Lipschitz functions, then $\nabla^{\phi} \phi$ must fluctuate
noticeably. Following the Euclidean idea, this observation (made
precise in Proposition \ref{dreamProp}) allows us to conclude that
intrinsic Lipschitz graphs satisfy a "WGL for CG $\beta$-numbers",
see Theorem \ref{t:CG-WGL}.

Up to this point, the results can be accused of being abstract
nonsense; as far as we know, locally defined CG graphs could be
quite wild, and we do not even claim to understand them very well.
What we can understand, however, are \textbf{globally defined} CG
graphs. These turn out to be affine (see Proposition
\ref{constantGradientGraphs})! Using this fact, and a compactness
argument, we can prove that even locally defined CG graphs have
the following key property: if $\Gamma$ is a CG graph "defined in
the whole ball $B(x,r)$", then $\Gamma$ is almost flat in all
sufficiently small sub-balls of $B(x,r)$. This implies almost
immediately that the usual WGL holds for CG-graphs, even if we do
not record the argument separately. Instead, we use our
"approximation by CG graphs" result to conclude directly that the
WGL holds for all intrinsic Lipschitz graphs.

We wish to mention that the proof strategy above was influenced by
X. Tolsa's proof \cite{To2} of the fact that the \emph{weak
constant density condition} implies \emph{uniform rectifiability}
in $\R^{n}$. Should the reader be familiar with that proof, she
may wish to draw the following parallels in her mind: "weak
constant density" is represented by "intrinsic Lipschitz", and
"uniform measure" is represented by "constant gradient graph".

\subsection{The intrinsic gradient} \label{ss:intr_grad}

Our proof of the WGL for intrinsic Lipschitz graphs is based on
the notion of \emph{intrinsic gradient}. The present section
serves the purpose of reviewing the definition and relevant
properties.

\subsubsection{Definitions}

According to a well known theorem by Rademacher, Lipschitz maps
between Euclidean spaces are differentiable almost everywhere. The
same result, appropriately interpreted, holds true for the
intrinsic Lipschitz functions appearing in this paper. Similarly
as in the Euclidean setting, a function $\phi:\W \to \V$ shall be
differentiable at a point $w_0\in \W$, if its graph at $p_0=w_0
\cdot \phi(w_0)$ can be well approximated by the graph of a
"linear" function $L:\W \to \V$.

Following the terminology in \cite{FSS11}, a function $L:\W \to
\V$ between complementary homogeneous subgroups is said to be
\emph{intrinsic linear}, if its intrinsic graph $\{w \cdot
L(w): \; w\in \W\}$ is a homogeneous subgroup of
$\He$. If $\W$ is a vertical subgroup and $\V$ is a
complementary orthogonal horizontal subgroup in $\He$, a
map $L: \W \to \V$ is intrinsic linear if and only if it is a
\emph{homogeneous homomorphism} (see for instance Proposition 3.26
in \cite{AS}). The latter means that $L$ is a group homomorphism
with the additional property that $L(\delta_r(w))=r L(w)$ for all
$w\in \W$ and $r>0$. To give an example, if $\W$ denotes the
$(y,t)$-plane, it is not difficult to see that all intrinsic
linear maps, or equivalently all homogeneous homomorphisms, are of
the form $L(y,t) = cy$ for a constant $c\in \R$.

We are now ready to state the definition of \emph{intrinsic
differentiability}. Again we assume that $\W$ is a vertical
subgroup with complementary horizontal subgroup $\V$. First let us
consider a function $\phi: \Omega \to \V$ defined on an open
subset of $\Omega \subset \W$ containing the origin, which we
assume to be fixed under $\phi$. We say that $\phi$ is intrinsic
differentiable at the origin, if there exists an intrinsic linear
map $L:\W \to \V$ such that, for all $w\in \W$,
\begin{displaymath}
|\phi(w) - L(w)|= o (\|w\|),\quad \text{as }\|w\|\to 0.
\end{displaymath}
The map $L$ is called the \emph{intrinsic differential} of $\phi$ at $0$ and it is denoted by $L= D_0 \phi$.

Since the definition of differentiability is supposed to be
\emph{intrinsic}, we extend it in a left invariant fashion to
arbitrary functions and points. To explain the definition, we consider a function $\phi:\Omega \to \V$ on an open set $\Omega \subset \W$  with intrinsic graph $\Gamma$.
For $w_0 \in \Omega$, we write $p_0 = w_0 \cdot \phi(w_0)$ and let $\phi_{p_0^{-1}}$ be the uniquely defined function $\W \to \V$, which parametrizes the graph $p_0^{-1} \cdot
\Gamma$, see Remark \ref{remk:para}. Note that the definition of intrinsic
graph ensures that any left translate of $\Gamma$ by a point in
$\Gamma$ is an intrinsic graph passing through the origin. This
uniquely determines the function $\phi_{p_0^{-1}}$ and it ensures
that the origin is fixed under this function. We also formulate an equivalent definition of intrinsic Lipschitz functions using the maps $\phi_p$. If $A \subset \W$, a function $\phi: A \ra \V$  with intrinsic graph $\Gamma$ is an intrinsic $L$-Lipschitz function if and only if for every $p \in \Gamma$
\begin{equation}
\label{intrili2}
\|\phi_{p^{-1}}(w)\| \leq L \|w\|
\end{equation}
for all $w$ in the domain of $\phi_{p^{-1}}$. See also \cite[Proposition 4.49]{S} for other equivalent algebraic definitions for intrinsic Lipschitz functions. An explicit formula
for $\phi_{{p_0}^{-1}}$ is given in Lemma \ref{l:transl} below.

\begin{definition} We say that a function $\phi:\Omega \to \V$, defined on an open set $\Omega \subset \W$, is
\emph{intrinsic differentiable} at a point $w_0$, if the function
$\phi_{p_0^{-1}}$ for $p_0 = w_0 \cdot \phi(w_0)$ is intrinsic differentiable at the origin.
The \emph{intrinsic differential} of $\phi$ at $w_0$ is given by
$
D_{w_0} \phi = D_0 \phi_{p_0^{-1}}.
$
\end{definition}

Intrinsic differentiability can be characterized in various
equivalent ways, see for instance the results in \cite{FSS},
\cite{FSS11}, and \cite{AS}.
One can also define intrinsic differentiability in terms of a
"graph distance", see Definition 1.4 in \cite{CMPS}, and for
intrinsic Lipschitz maps this definition is equivalent to the one
above.

The intrinsic differential is unique, and its action can be
expressed in terms of a "gradient" similarly as in the Euclidean
case. To explain this, we identify $\W \cong \R^{2}$ and $\V \cong \R$. Points in $\W$ are then denoted by $(y,t)$. Assume that a map $\phi \colon \W \to \V$ is intrinsic differentiable at $w_{0} = (y_{0},t_{0})$. Then, as pointed out earlier, the intrinsic differential $L := D_{w_{0}}\phi$ of $\phi$ at $w_{0}$ is a linear mapping of the form $L(y,t) = cy$ for some $c \in \R$.

\begin{definition}\label{d:intr_grad} The number $c$ is called the \emph{intrinsic gradient of $\phi$ at $w_{0}$}, and it will be denoted by
\begin{displaymath} \nabla^{\phi}\phi(w_{0}) := c. \end{displaymath}
\end{definition}
We next derive a few useful formulae for $\nabla^{\phi}\phi(w_{0})$. Let $p_{0} := w_{0} \cdot \phi(w_{0})$ be the point on the graph of $\phi$. By definition of intrinsic differentiability,
\begin{displaymath} |\phi_{p_{0}^{-1}}(h,0) - \nabla^{\phi}\phi(w_{0})h| = |\phi_{p_{0}^{-1}}(h,0) - L(h,0)| = o(|h|), \quad h \in \R \setminus \{0\}. \end{displaymath}
Consequently, dividing by $h$, we have
\begin{equation}\label{eq:directional} \nabla^{\phi}\phi(w_{0}) = \lim_{h \to 0} \frac{\phi_{p_{0}^{-1}}(h,0)}{h}. \end{equation}
Moreover, wherever $\phi$ is differentiable in the usual (Euclidean) sense, the formula above, and
\begin{displaymath} \lim_{h \to 0} \frac{\phi_{p_{0}^{-1}}(h,0)}{h} = \lim_{h \to 0} \frac{\phi(y_{0} + h,t_{0} + \phi(y_{0},t_{0})h) - \phi(y_{0},t_{0})}{h} \end{displaymath}
(this follows from the formula for $\phi_{p_{0}^{-1}}$ in Lemma \ref{l:transl}) yields the following representation for $\nabla^{\phi}\phi$:
\begin{equation}\label{eq:intr_grad_expl}
 \nabla^{\phi}\phi = \partial_y \phi + \phi \partial_t \phi.
 \end{equation}
This was observed in Example 5.5 in \cite{AS}. 

A large class of almost everywhere intrinsic differentiable
functions is provided by intrinsic Lipschitz functions whose
target is a  $1$-dimensional horizontal subgroup. This result was
first proved by Franchi, Serapioni and Serra Cassano in
\cite[Theorem 4.29]{FSS11} for Heisenberg groups, and later by
Franchi, Marchi and Serapioni in \cite{FMS} for certain more
general Carnot groups. We state here the result relevant for the
current paper:

\begin{thm}[Franchi, Serapioni, Serra Cassano]\label{t:intrRadem}
Let $\W$ be a vertical subgroup of $\He$ with
complementary horizontal subgroup $\V$. Assume that $\Omega$ is an
open subset of $\W$ and $\phi:\Omega \to \V$ is intrinsic
Lipschitz. Then $\phi$ is intrinsic differentiable $\mathcal{L}^2$
almost everywhere in $\W$.
\end{thm}

It follows that the intrinsic gradient $\nabla^{\phi}\phi$ of an
intrinsic Lipschitz function $\phi: \Omega \to \V$
exists almost everywhere in $\Omega$, and it is an $L^{\infty}$
function, see for instance Proposition 4.4 in \cite{CMPS}.
In our situation, one can say something more precise:

\begin{lemma}\label{l:bound_grad}
Let $\W$ be a vertical subgroup with complementary horizontal
subgroup $\V$ in $\He$, and let $\phi: \Omega
\to \V$ be an intrinsic $L$-Lipschitz function on an open set
$\Omega \subset \W$. Then
\begin{equation}\label{eq:bound_gradient_L}
\|\nabla^{\phi}\phi\|_{L^{\infty}(\Omega)} \leq L.
\end{equation}
\end{lemma}

Such a result was stated for \emph{difference quotients} in
\cite[Proposition 3.9 (i)]{FSS}. For the convenience of the
reader, we spell out the argument for the intrinsic gradient.

\begin{proof} By Theorem \ref{t:intrRadem}, $\phi$ is intrinsic differentiable
in almost every point of $\Omega$. We let $w_0$ be such a point,
and we write $p_0= w_0 \cdot \phi(w_0)$  for the corresponding point
on the graph. Recall that the function $\phi_{{p_0}^{-1}}$ in
\eqref{eq:directional} is defined as the function whose graph is
the left translate of the graph of $\phi$ by $p_0^{-1}$. We denote
the domain of $\phi_{{p_0}^{-1}}$ by $\Omega_{p_0^{-1}}$.

In order to prove \eqref{eq:bound_gradient_L}, it suffices to find a bound for the limit in
\eqref{eq:directional}. To this end, let $h$ be small enough such
that $(h,0)\in \Omega_{p_0^{-1}}$. Then by \eqref{intrili2}
\begin{displaymath}
\left|\frac{\phi_{{p_0}^{-1}}(h,0))}{h}\right|  \leq h^{-1} L \|(h,0)\| = L.
\end{displaymath}
Here we have used \eqref{intrili2}. Thus $|\nabla^{\phi}\phi(w_0)|\leq L$. Since $w_0$ was an
arbitrary point of intrinsic differentiability,
\eqref{eq:bound_gradient_L} follows by Theorem \ref{t:intrRadem}.
\end{proof}

In the converse direction, Proposition 1.8 in \cite{CMPS} provides
local upper bounds for the Lipschitz constant in terms of the
$L^{\infty}$-norm of the intrinsic gradient. While these results
connect $\nabla^{\phi}\phi$ to the geometry of intrinsic Lipschitz
graphs, the intrinsic gradient has a life of its own outside the
world of intrinsic Lipschitz maps. To see this, it is best to
express $\nabla^{\phi}\phi$ as in \eqref{eq:intr_grad_expl}. The
equation \begin{equation}
\partial_y \phi + \phi \partial_t \phi =0
\end{equation}
is well known in PDE theory as the \emph{inviscid Burgers
equation}. This will be discussed further in Section \ref{s:constantGradientGraphs} below.

In which sense are intrinsic Lipschitz functions solutions to an equation of Burgers' type? By Theorem \ref{t:intrRadem}, the intrinsic gradient of an intrinsic Lipschitz
function exists pointwise almost everywhere. In connection with PDE theory, it is useful to know
that the pointwise intrinsic gradient of an intrinsic Lipschitz
function is also a \emph{distributional} gradient. This is the content of
Proposition 4.7 in \cite{CMPS}. Precisely, if $\phi:\W \to \V$ is intrinsic Lipschitz with intrinsic gradient
$\nabla^{\phi}\phi$, defined almost everywhere by Theorem \ref{t:intrRadem}, then
\eqref{eq:intr_grad_expl} holds in a distributional sense:
\begin{equation}\label{eq:distr}
\int_\W \phi \partial_y \psi + \tfrac{1}{2} \phi^2 \partial_t\psi
\;d\mathcal{L}^2 = - \int_\W \nabla^{\phi}\phi
\psi\;d\mathcal{L}^2
\end{equation}
for all $\psi \in \mathcal{C}^1_c(\W)$.

We emphasize that the intrinsic Lipschitz functions as in Definition
\ref{d:intrLip} coincide with the intrinsic Lipschitz functions in the sense of
\cite{CMPS}, see Theorem 4.60 in \cite{S} and  Remark 3.6 in \cite{FS}. Note further
that the formula \eqref{eq:distr} looks slightly different from
\cite{CMPS} due to a different model for the Heisenberg group, see
Definition 3.1 in \cite{BCSC}.

\subsubsection{Translated and dilated graphs}

While the definition of intrinsic Lipschitz continuity is tailored so that the class of intrinsic $L$-Lipschitz graphs is preserved under dilations and translations in the Heisenberg group, the explicit formula for the parametrization of a translated graph becomes in general slightly complicated due to the non-commutativity of the group law. In the case we consider in the present paper: functions from vertical to horizontal subgroups in $\mathbb{H}$, the computations are straightforward.

For convenience, given a
point $p\in \He$, we define the map
\begin{displaymath}
P_p:\mathbb{W} \to \mathbb{W},\quad P_p(w):=
\pi_{\mathbb{W}}(p \cdot w).
\end{displaymath}
We note that $P_q$ is a diffeomorphism with Jacobian
determinant constant equal to $1$ under the obvious identification
of $\W$ with $\mathbb{R}^2$ (see \cite[Lemma 2.20]{FS}), and with
inverse map $(P_p)^{-1}=P_{p^{-1}}$. The latter claim
follows from the fact that
\begin{displaymath}
P_p( w) = p \cdot w \cdot \pi_{\V}(p)^{-1},\quad\text{for all
}w\in \W.
\end{displaymath}

\begin{lemma}\label{l:transl}
Let $\Gamma$ be the intrinsic graph of a function $\phi:
\W \to \V$ on a vertical subgroup $\W$, and let $\Omega$ be a domain in
$\He$. Then, for $q\in \He$, the set
$\tau_{q}(\Omega\cap \Gamma)$ is the intrinsic graph of the function
\begin{displaymath}
\phi_q:\pi_{\W}(\tau_{q}(\Omega\cap \Gamma))\to \mathbb{V},\quad \phi_q(w):=
\pi_\V(q) \cdot
\phi(P_{q^{-1}}(w)).
\end{displaymath}
If $\phi$ is intrinsic $L$-Lipschitz, then so is  $\phi_q$ with
\begin{equation}\label{eq:chain}
\nabla^{\phi_q} \phi_q = \nabla^{\phi} \phi \circ
P_{q^{-1}},\quad \text{almost everywhere}.
\end{equation}
Analogously, for $r>0$, the set $\delta_{r}(\Omega\cap \Gamma)$ is the
intrinsic graph of the function
\begin{displaymath}
\phi_r: \pi_{\W}( \delta_r(\Omega \cap \Gamma)) \to \V,\quad \phi_r(w)=
\delta_r(\phi(\delta_{1/r}(w))).
\end{displaymath}
If $\phi$ is intrinsic $L$-Lipschitz, then so is  $\phi_r$ with
\begin{displaymath}
\nabla^{\phi_r} \phi_r = \nabla^{\phi} \phi \circ
\delta_{r^{-1}},\quad \text{almost everywhere}.
\end{displaymath}
\end{lemma}

Directly from the definition of intrinsic differential and intrinsic gradient, it follows that
\begin{displaymath}
\nabla^{\phi_q}\phi_q (0) = \nabla^{\phi}\phi( \pi_{\W}(q^{-1})).
\end{displaymath}
In \eqref{eq:chain}, we show how $\nabla^{\phi_q}\phi_q$ and $\nabla^{\phi}\phi$ are related in a generic point of $\W$.

\begin{proof} We concentrate on proving the statement for left translations.
In \cite[Proposition 2.7]{FSS} it has been shown that if $\Gamma$ is the intrinsic
graph of a function $\phi$ over a domain in $\V$, then
$\tau_{q}(\Gamma)$ is an intrinsic graph parametrized by the function
$\phi_q$. The domain of the new function $\phi_q$ is simply the
image  of $\tau_{q}(\Gamma)$ under the projection onto $\W$. The
intrinsic Lipschitz property of $\phi_q$ -- assuming the
corresponding property for $\phi$ -- is the content of
\cite[Theorem 3.2]{FSS}.

To compute the intrinsic gradient, we may assume without loss of
generality that $\mathbb{W}$ agrees with the $(y,t)$-plane and
that $\phi$ is defined on the entire plane $\mathbb{W}$. In this
case, for $q=(x_0,y_0,t_0)$, we have that
\begin{displaymath}
\phi_q(y,t) = x_0 + \phi(P_{q^{-1}}(y,t)).
\end{displaymath}
We then use the fact, proved in \cite{CMPS}, that the intrinsic gradient of an intrinsic
Lipschitz function is also a distributional gradient; see the discussion in Section \ref{ss:intr_grad}.
Let now $\psi$ be an arbitrary test function, that is, a compactly
supported $\mathcal{C}^1$ function on $\mathbb{W}$. Since
\begin{equation}\label{jacobian}
D P_q =
\begin{pmatrix}1&0\\ x_0 & 1\end{pmatrix},
\end{equation}
we find that
\begin{displaymath}
\frac{\partial \psi\circ P_q}{\partial y} = \frac{\partial
\psi}{\partial y}\circ P_q +  x_0 \frac{\partial \psi}{\partial
t}\circ P_q\quad \text{and}\quad \frac{\partial \psi\circ
P_q}{\partial t} = \frac{\partial \psi}{\partial t}\circ
P_q.
\end{displaymath}
This, together with the facts that $\det DP_q=1$ and
$\nabla^{\phi}\phi$ is a distributional gradient,recall \eqref{eq:distr}, gives
\begin{align*}
\int \nabla^{\phi_q} \phi_q \psi \,d\mathcal{L}^2&= - \int \phi_q \tfrac{\partial \psi}{\partial y} + \tfrac{1}{2} \phi_q^2 \tfrac{\partial \psi}{\partial t}\, d\mathcal{L}^2\\
&= - \int (\phi\circ P_{q^{-1}} )\tfrac{\partial \psi}{\partial y} + x_0( \phi \circ P_{q^{-1}}) \tfrac{\partial \psi}{\partial t} + \tfrac{1}{2}(\phi^2\circ P_{q^{-1}}) \tfrac{\partial \psi}{\partial t} \,d\mathcal{L}^2\\
&= - \int \phi( \tfrac{\partial \psi}{\partial y} \circ P_q) + x_0 \phi (\tfrac{\partial \psi}{\partial t} \circ P_q) +\tfrac{1}{2} \phi^2 (\tfrac{\partial \psi}{\partial t} \circ P_q )\, d\mathcal{L}^2\\
&= - \int \phi \tfrac{\partial \psi \circ P_q}{\partial y}+ \tfrac{1}{2}\phi^2 \tfrac{\partial \psi \circ P_q}{\partial t}  \,d\mathcal{L}^2\\
&= \int \nabla^{\phi} \phi\, (\psi \circ P_q ) \,d\mathcal{L}^2\\
&= \int( \nabla^{\phi} \phi )\circ P_{q^{-1}} \,\psi
\,d\mathcal{L}^2.
\end{align*}
As this computation is valid for arbitrary test functions $\psi$,
the claim \eqref{eq:chain} follows.
\end{proof}

\subsubsection{Graphs with constant gradient}\label{s:constantGradientGraphs} In this subsection, we prove that "entire" intrinsic Lipschitz functions with almost surely constant gradient are affine. As mentioned in Section \ref{ss:intr_grad}, if $\W$ is identified with $\mathbb{R}^2$,
the differential equation $\nabla^{\phi}\phi= 0$ is known as the
inviscid Burgers equation and it is not difficult to see by
the \emph{method of characteristics} (see \cite[Proposition 5.1]{Er}) that the only global
$\mathcal{C}^1$ solutions are constant functions. If the
right-hand side of the equation is replaced by some other constant
$c$, one can show in the same vein that the only $\mathcal{C}^{1}$
solutions are affine functions of the form $\phi(y,t) = cy +d$;
see \cite[Remark 4.3]{ASV}.

Our task is  to establish the same result for functions $\phi$
that are merely assumed to be intrinsic Lipschitz with intrinsic
gradient constant almost everywhere.

\begin{proposition}\label{constantGradientGraphs}
Let $\phi:\W  \to \V$ be an intrinsic Lipschitz function. If there
exists a constant $c\in \mathbb{R}$ such that $\nabla^{\phi}\phi=
c$ almost everywhere in $\W$, then the graph of $\phi$ is the left
translate of some vertical plane $\W' = \W'_{\W,c}$.
\end{proposition}

\begin{rem}
We thank Enrico Le Donne and the anonymous referee for pointing out that Proposition \ref{constantGradientGraphs} also follows directly from existing results in the literature. Indeed, by \cite[Theorem 4.17]{FSS11}, the subgraph $\Gamma^{\phi}_{-}$  of an intrinsic Lipschitz graph $\Gamma^{\phi}$ over a vertical plane is a set with locally finite $\mathbb{H}$-perimeter. If the intrinsic gradient $\nabla^{\phi}\phi$ is constant, then the horizontal normal to the boundary of $\Gamma^{\phi}_{-}$ is constant. According to the proof of 
\cite[Claim 3]{FSS01} (see also \cite[Proposition 5.4]{AKL}), this implies that $\Gamma^{\phi}_{-}$ is the left translate of a vertical halfspace, and hence $\Gamma^{\phi}$ is the left translate of a vertical plane. 
\end{rem}

\begin{proof}[Proof of Proposition \ref{constantGradientGraphs}]
Throughout the proof, we identify $\W$ with $\mathbb{R}^2$, using
coordinates $(y,t)$. We will prove that $\phi(y,t) = cy + d$ for some $d \in \R$.
Since $\phi$ is intrinsic Lipschitz, it is continuous, see for instance \cite[Proposition 3.4]{FSS}.
 We start by observing that the almost sure
constancy of the intrinsic gradient leads to improved regularity
for $\phi$.
By \cite[Proposition 4.7]{CMPS}, the function $\phi$ is  a
distributional solution to the equation $\nabla^{\phi}\phi = g$
for $g(y,t)\equiv c$.
Since $g$ is constant, it is in particular Lipschitz continuous in the Euclidean sense and  \cite[Corollary 1.4]{BSC} implies that $\phi$ is locally Lipschitz
on $\W$ with respect to the Euclidean metric; see also
Theorem 4.2.1 and Theorem 4.4.1 in \cite{B}.   Hence, almost every
point $w$ of $\W$ is "good" in the sense that the function $\phi$
is differentiable at $w$ both in the usual Euclidean sense and in
the intrinsic sense with $\nabla^{\phi}\phi(w)=c$. We denote by
$G$ the set of such good points in $\W$, so that
$\mathcal{L}^{2}(\W \setminus G)=0$.

For every $t\in \R$, we define a curve $\gamma_t:\R \to \W$, by
setting
\begin{displaymath}
 \gamma_{t}(s):= \left(s, \tfrac{c}{2} s^2 + \phi(0,t) s + t\right).
\end{displaymath}
We will prove for almost every $t\in \R$ that
\begin{equation}\label{eq:behavior_along_char}
\phi(\gamma_t(s))= cs+ \phi(0,t)\quad\text{for all }s\in \R.
\end{equation}

Let us assume for a moment that $t\in \R$ is such that
$\gamma_t(s)\in G$ for almost every $s\in \R$. Towards a proof of
\eqref{eq:behavior_along_char}, we define the function
\begin{displaymath}
z:\R \to \R,\quad z(s):= \phi(\gamma_t(s))-(cs + \phi(0,t)).
\end{displaymath}
We note that $z$ is locally Lipschitz continuous, so $z'(s)$
exists for almost all $s$. Further, by the assumption that
$\gamma_{t}(s) \in G$ for almost every $s$, we have for such
points that
\begin{align*}
z'(s)&= \partial_y \phi(\gamma_t(s))+ (cs + \phi(0,t))\partial_t \phi(\gamma_t(s)) -c\\
&= c- \phi(\gamma_t(s))\partial_t \phi(\gamma_t(s))+(cs + \phi(0,t))\partial_t \phi(\gamma_t(s)) -c\\
&= -\left[\phi(\gamma_t(s))-(cs + \phi(0,t))\right]\partial_t \phi(\gamma_t(s))\\
&= -z(s) \partial_t \phi(\gamma_t(s)).
\end{align*}
Here we have used that $\nabla^{\phi}\phi=c=
\partial_y \phi + \phi
\partial_t\phi$ on $G$, see \eqref{eq:intr_grad_expl} and the subsequent discussion.

 Thus $z$ solves an ODE of the form
\begin{equation}\label{eq:ODE}
\left\{ \begin{array}{ll}z'(s)=a(s)z(s),&\text{almost
everywhere},\\ z(0)=0.&
\end{array}\right.
\end{equation}
Clearly, $z\equiv 0$ is a solution,
 but we have to argue that it is the \emph{only} solution. Here we are interested in \emph{Carath\'{e}odory solutions} $z:\R\to \R$, that is, in functions which are absolutely continuous on every closed interval $[\alpha,\beta]\subset \R$ and which fullfill the differential equation pointwise almost everywhere; see for instance \cite[Chapter 1]{F} for a thorough discussion of Carath\'{e}odory differential equations. Since $\phi$ is locally Lipschitz as a function on the Euclidean plane and $\gamma_t$ is a smooth curve with $\gamma_t(s)\in G$ for almost every $s\in \R$, the function
\begin{displaymath}
a:\R \to \R,\quad a(s):= \left\{\begin{array}{ll}- \partial_t
\phi(\gamma_t(s)),&\gamma_t(s)\in G,\\ 0,&\text{else},\end{array}
\right.
\end{displaymath}
is locally integrable on every interval $[\alpha,\beta]\subset
\R$. By Theorem 3 in \cite[Chapter 1]{F}, this suffices to ensure
that the ODE \eqref{eq:ODE} has a unique Carath\'{e}odory solution
on $\R$. Hence, $z\equiv 0$ and \eqref{eq:behavior_along_char}
follows for this particular choice of $t$.

Next, we would like to show that almost every $t\in \R$ has the
crucial property that $\gamma_t(s)\in G$ for almost every $s$.
This is the content of Lemma \ref{avoidance} below. The statement
would be immediate if we knew that the curves $\gamma_{t}$
foliated the plane $\W$, or even a large portion thereof, but
there is no such \emph{a priori} information available. In fact,
this foliation property is part of the statement we want to prove.

So, we have to work a bit harder, and we are essentially rescued
by the local Lipschitz regularity of $t \mapsto \phi(0,t)$. In the
proof of Lemma \ref{avoidance} we need a sharpened version of the
"easy implication" in the Besicovitch projection theorem. This
result may be known to some experts, but we did not find it in the
literature:

\begin{lemma}\label{projectionLemma} Let $K \subset \R^{2}$ be a rectifiable set with $0 < \calH^{1}(K) < \infty$. Then, there exists a set of unit vectors $G \subset S^{1}$, depending only on $K$, with the following properties:
\begin{itemize}
\item[(i)] $\calH^{1}(S^{1} \setminus G) = 0$. \item[(ii)] If $F
\subset K$ is any $\calH^{1}$-measurable subset with $\calH^{1}(F)
> 0$ and $e \in G$, then $\calH^{1}(\pi_{e}(F)) > 0$. Here
$\pi_{e}$ is the orthogonal projection $\pi_{e}(x) = x \cdot e$.
\end{itemize}
\end{lemma}

\begin{remark}
\label{sfinite} The lemma immediately extends to rectifiable sets with $\sigma$-finite $\calH^{1}$-measure.  \end{remark}

\begin{proof}[Proof of Lemma \ref{projectionLemma}] We start with a series of reductions. Without loss of generality we can assume that $K$ is bounded. Moreover it is enough to prove the lemma for $K=f([-R, R])$ for $f:\R \ra \R$ Lipschitz and $R>0$. To see this, first recall that if $K$ is a bounded rectifiable set there exist countably many Lipschitz maps $f_n:\R \ra \R$ such that $\calH^{1}(K \setminus \bigcup_n f_n([-R, R]))=0$ for some $R>0$. Applying the lemma to each of the sets $\gamma_n:=f_n([-R,R])$ we obtain sets $G_n \subset S^1$ satisfying (i) and (ii).  
Now let $G= \bigcap_{n \in \N} G_n$. Trivially $\calH^{1}(S^{1} \setminus G) = 0$ and if $F
\subset K$ is any $\calH^{1}$-measurable subset with $\calH^{1}(F)> 0$, then there exists some $\gamma_n$ such that $\calH^1(\gamma_n \cap F)>0$. Since $\gamma_n$ satisfies (ii), if $e \in G \subset G_n$ then $\calH^{1}(\pi_{e}(F)) > 0$.

Fix $\varepsilon > 0$. It is enough to find a subset
$G_{\varepsilon} \subset S^{1}$ such that $\calH^{1}(S^{1} \setminus
G_{\varepsilon}) = 0$ and the following property holds: if $F \subset K$
is measurable with $\calH^{1}(F) > \varepsilon$ and $e \in
G_{\varepsilon}$, then $\mathcal{H}^1(\pi_{e}(F)) > 0$. Then, we
can complete the proof by setting $G := \bigcap_{j} G_{1/j}$.

By \cite[Theorem 7.4]{M} there exists a compact $\mathcal{C}^{1}$-curve $\Gamma = \Gamma_{\varepsilon,K}$ 
such that $\calH^{1}(K \setminus \Gamma) \leq \varepsilon/2$.
Then, if $F \subset K$ is measurable with $\calH^{1}(F) >
\varepsilon$, we have $\calH^{1}(F \cap \Gamma) > \varepsilon/2$.
Thus, it actually suffices to construct $G_{\varepsilon}$ so that
the following holds: if $F \subset \Gamma$ is measurable with
$\calH^{1}(F) > \varepsilon/2$ and $e \in G_{\varepsilon}$, then
$\mathcal{H}^1(\pi_{e}(F)) > 0$. One final reduction: for fixed
$\delta
> 0$, we construct a set $G_{\varepsilon}^{\delta}$ with the
properties that (a) $\calH^{1}(S^{1} \setminus
G_{\varepsilon}^{\delta}) < \delta$, and (b) if $F \subset \Gamma$
is measurable with $\calH^{1}(F) > \varepsilon/2$ and $e \in
G_{\varepsilon}^{\delta}$, then $\mathcal{H}^1(\pi_{e}(F)) > 0$.
This suffices, since $G_{\varepsilon} := \bigcup_{j}
G_{\varepsilon}^{1/j}$ is then the set we are after.

To construct $G_{\varepsilon}^{\delta}$, we fix a number $m =
m_{\delta,\varepsilon} \in \N$, to be specified later, and cover
$S^{1}$ by a a collection $\mathcal{J} := \{J_{1},\ldots,J_{m}\}$
of disjoint arcs of length  between $1/m$ and $10/m$.
Next, for some $n \in \N$ depending on $m$, we partition
$\Gamma$ into short, connected sub-curves $\calF :=
\{\Gamma_{1},\ldots,\Gamma_{n}\}$ such that the following holds:
for every fixed $\Gamma_{j} \in \calF$, the restriction
$\pi_{e}|_{\Gamma_{j}}$ is bi-Lipschitz for all $e \in S^{1}$,
except possibly those $e$ in the union of four arcs in
$\mathcal{J}$ (depending only on $\Gamma_{j}$). Such a partition
$\calF$ exists, because $\Gamma$ is compact and $\mathcal{C}^{1}$.

Consider a bi-partite graph with vertex set $\calF \cup
\mathcal{J}$ and the following edge set $E$: draw an edge between
$\Gamma_{j} \in \calF$ and $J_{k} \in \mathcal{J}$, if and only if
$\pi_{e}|_{\Gamma_{j}}$ is bi-Lipschitz for \textbf{all} $e \in
J_{k}$. Thus, every vertex $\calF$ is adjacent to at least $(m -
4)$ vertices in $\mathcal{J}$. For an edge $(\Gamma_{j},J_{k}) \in
E$, define the weight
\begin{displaymath} w(\Gamma_{j},J_{k}) := \calH^{1}(\Gamma_{j}). \end{displaymath}
Thus, if $w(E)$ is the sum of all the weights of edges in $E$, we
have
\begin{displaymath} w(E) = \sum_{j} \sum_{k : (\Gamma_{j},J_{k}) \in E} w(\Gamma_{j},J_{k}) \geq \sum_{j} (m - 4)\calH^{1}(\Gamma_{j}) = (m - 4)\calH^{1}(\Gamma). \end{displaymath}
Now, write
\begin{displaymath} \tau := \min\{\delta/20,\varepsilon/(2\calH^{1}(\Gamma))\}, \end{displaymath}
and call a vertex $J_{k} \in \mathcal{J}$ \emph{light}, if the
total weight of edges emanating from $J_{k}$ is at most $(1 -
\tau)\calH^{1}(\Gamma)$. Other vertices in $\mathcal{J}$ are
\emph{heavy}. Denoting the light and heavy vertices in
$\mathcal{J}$ by $\mathcal{J}_{light}$ and $\mathcal{J}_{heavy}$,
respectively, we have
\begin{align*} (m - 4)\calH^{1}(\Gamma) & \leq w(E) \leq (1 - \tau)\calH^{1}(\Gamma)|\mathcal{J}_{light}| + \calH^{1}(\Gamma)|\mathcal{J}_{heavy}|\\
& = m\calH^{1}(\Gamma) - \tau \calH^{1}(\Gamma)
|\mathcal{J}_{light}|, \end{align*} which simplifies to
$|\mathcal{J}_{light}| \leq 4/\tau$. We now fix $m$ so large that
$m \geq 4/\tau^{2}$, which gives $|\mathcal{J}_{light}| \leq \tau
m$. Then, let
\begin{displaymath} G_{\varepsilon}^{\delta} := \bigcup_{J_{k} \in \mathcal{J}_{heavy}} J_{k}. \end{displaymath}
The set $G_{\varepsilon}^{\delta}$ satisfies the correct length
estimate:
\begin{displaymath} \calH^{1}\left(S^{1} \setminus G_{\varepsilon}^{\delta} \right) \leq \sum_{J_{k} \in \mathcal{J}_{light}} \calH^{1}(J_{k}) \leq \frac{10\tau m}{m} < \delta \end{displaymath}
by the choice of $\tau$.

Finally, we want to show that $\mathcal{H}^1(\pi_{e}(F)) > 0$,
whenever $F \subset \Gamma$ is measurable with $\calH^{1}(F) >
\varepsilon/2$, and $e \in G_{\varepsilon}^{\delta}$. So, fix $F
\subset \Gamma$ with $\mathcal{H}^{1}(F) > \varepsilon/2$, and
write $\calF_{F} := \{\Gamma_{j} \in \calF : \calH^{1}(F \cap
\Gamma_{j}) > 0\}$. Then
\begin{displaymath} \varepsilon/2 < \calH^{1}(F) \leq \sum_{\Gamma_{j} \in \calF_{F}} \calH^{1}(\Gamma_{j}), \end{displaymath}
Then, fix $e \in G_{\varepsilon}^{\delta}$, so that $e \in J_{k}$
for some $\mathcal{J}_{heavy}$. This implies that $J_{k}$ is
adjacent to at least one vertex $\Gamma_{j} \in \calF_{F}$;
otherwise, recalling that $\tau <
\varepsilon/(2\calH^{1}(\Gamma))$, we have
\begin{displaymath} \sum_{j : (\Gamma_{j},J_{k}) \in E} w(\Gamma_{j},J_{k}) \leq \sum_{j : \Gamma_{j} \notin \calF_{F}} w(\Gamma_{j},J_{k}) \leq \calH^{1}(\Gamma) - \varepsilon/2 < (1 - \tau)\calH^{1}(\Gamma) \end{displaymath}
which contradicts $J_{k} \in \mathcal{J}_{heavy}$. Now, pick
$\Gamma_{j} \in \calF_{F}$ such that $(\Gamma_{j},J_{k}) \in E$.
By definition of $E$, this means that $\pi_{e}|_{\Gamma_{j}}$ is
bi-Lipschitz, and consequently
\begin{displaymath} \mathcal{H}^1(\pi_{e}(F)) \geq \mathcal{H}^1(\pi_{e}(F \cap \Gamma_{j})) > 0. \end{displaymath}
The proof is complete. \end{proof}

We are ready to prove that the curves $\gamma_{t}$ mostly avoid
the set $\W \setminus G$:

\begin{lemma}\label{avoidance} Let $B \subset \W$ be a set with $\mathcal{L}^{2}(B) = 0$. Then, for almost every $t$, we have $\gamma_{t}(s) \in \W \setminus B$ for almost every $s$.   \end{lemma}

\begin{proof} With $B_{s} := \{t : (s,t) \in B\}$, we may re-write the claim as follows:
\begin{align*} 0 = \calL^{2}(\{(s,t) : \gamma_{t}(s) \in B\}) & = \calL^{2}(\{(s,t) : (s,\tfrac{c}{2}s^{2} + \phi(0,t)s + t) \in B\})\\
& = \calL^{2}(\{(s,t) : \tfrac{c}{2}s^{2} + \phi(0,t)s + t \in
B_{s}\}). \end{align*} So, it suffices to show that for almost
every $s \in \R$, we have $\calH^{1}(E_{s}) = 0$, where
\begin{displaymath} E_{s} := \{t : \tfrac{c}{2}s^{2} + \phi(0,t)s + t \in B_{s}\}. \end{displaymath}
Assume that this claim is false: there exists a positive measure
set $S$ of parameters $s$ such that $\calH^{1}(E_{s}) > 0$.
Observe that $\calH^{1}(\Gamma_{s}) \geq \calH^{1}(E_{s}) > 0$ for
$s \in S$, where
\begin{displaymath} \Gamma_{s} := \{(\phi(0,t),t) : t \in E_{s}\} \subset \{(\phi(0,t),t) : t \in \R\} =: \Gamma. \end{displaymath}
Next, write $\pi_{s}(y,t) := (y,t) \cdot (s,1)$ for $(y,t) \in
\W$; then, up to scaling, $\pi_{s}$ is the orthogonal projection
onto the line spanned by $(s,1)$ in the $(y,t)$-plane. Recalling Remark \ref{sfinite}  we can apply
Lemma \ref{projectionLemma} to the Lipschitz graph $\Gamma$ and
obtain a set of parameters $G \subset \R$ with $\calH^{1}(\R
\setminus G) = 0$, with the property that
$\calH^{1}(\pi_{s}(\Gamma_{s})) > 0$, whenever $s \in G$ and
$\calH^{1}(\Gamma_{s}) > 0$. In particular,
$\calH^{1}(\pi_{s}(\Gamma_{s})) > 0$ for almost all $s \in S$.
Observing that $\tfrac{c}{2}s^{2} + \pi_{s}(\Gamma_{s}) \subset
B_{s}$ for every $s$, this forces $\calH^{1}(B_{s}) > 0$ for
almost all $s \in S$, which contradicts $\calL^{2}(B) = 0$. The
proof of the lemma is complete. \end{proof}

We have now established that \eqref{eq:behavior_along_char} holds
for almost every $t \in \R$, and the rest of the proof of
Proposition \ref{constantGradientGraphs} is easy. First, notice
that
\begin{displaymath}
\gamma_t(s)= \gamma_{t'}(s')\quad \text{if and only if} \quad
(s=s' \text{ and } \phi(0,t)s+t=\phi(0,t')s+t').
\end{displaymath}
Now, if $\phi(0,t)\neq \phi(0,t')$ for some $t,t'$, then
$\phi(0,t)s_{0} + t = \phi(0,t') s_{0} + t'$ for some $s_{0} \in
\R$. It follows that for such $t,t'$, the curves $\gamma_t$ and
$\gamma_{t'}$ intersect at $\gamma_t(s_0)=\gamma_{t'}(s_0)$.

Recall that we aim to show that $\phi(y,t) =c y + d$ for some $d
\in \R$. We first show that $t \mapsto \phi(0,t)$ is constant.
Pick $t$ and $t'$ satisfying \eqref{eq:behavior_along_char}. If
$\phi(0,t) \neq \phi(0,t')$, then by the discussion in the
previous paragraph, $\gamma_t(s_{0}) = \gamma_{t'}(s_{0})$ for
some $s_{0} \in \R$. Consequently,
\begin{displaymath}
c s_0 + \phi(0,t) = \phi(\gamma_{t}(s_{0})) =
\phi(\gamma_{t'}(s_{0})) = c s_0 + \phi(0,t')
\end{displaymath}
by \eqref{eq:behavior_along_char}, which contradicts $\phi(0,t)
\neq \phi(0,t')$. So, $t \mapsto \phi(0,t)$ is constant, say $d$,
on the set where \eqref{eq:behavior_along_char} holds. Referring
again to \eqref{eq:behavior_along_char}, we find that, for
$\calL^{2}$ almost all $(s,t) \in \R \times \R$, we have
\begin{displaymath}\phi(\gamma_t(s))= \phi(s,\tfrac{c}{2}s^{2} + ds + t) = cs + d. \end{displaymath}
Since $\phi$ is continuous, this is in fact true for all pairs
$(s,t)$, and hence $\phi(y,t) = \phi(y,\tfrac{c}{2}y^{2} + dy + (t
- \tfrac{c}{2}y^{2} - dy)) = cy + d$ for all $(y,t) \in \R^{2}$.
The proof is complete. \end{proof}

\subsection{A weak geometric lemma for constant gradient $\beta$-numbers} In this section, we start to implement the plan outlined in Section \ref{outline}: we define a variant of $\beta$-numbers, the \emph{constant gradient $\beta$-numbers}, and prove that intrinsic Lipschitz graphs satisfy a weak geometric lemma with respect to this new definition.

If $\Gamma$ is an intrinsic Lipschitz graph over $\W$, the constant gradient $\beta$-number (with parameter $L$) of a ball $B(x,r)$ is designed to described how well $\Gamma\cap B(x,r)$ can be approximated by the graph of an intrinsic $L$-Lipschitz function whose gradient is constant almost everywhere in $\pi_{\W}(B(x,b_L r))$. Here $b_L$ is a small constant, given by the following lemma, and $B(x,r)$ denotes a closed ball with radius $r$
centred at $x$.

\begin{lemma}           \label{eq:ball_inclusion}
For every $L>0$ there exists a constant $b_L$ such that if $\phi:\W \to \V$ is an intrinsic $L$-Lipschitz function on a vertical subgroup $\W$, then
\begin{displaymath}
\pi_{\W}(B(x,b_L r ))\subseteq   \pi_{\W}(B(x, r )\cap \Gamma)\subseteq \pi_{\W}(B(x, r ))
\end{displaymath}
for all $x$ on the graph $\Gamma$ of $\phi$ and for all $r>0$.
\end{lemma}

For a proof of this lemma, see  (44) in \cite{FS}. We are now ready to state our definition of constant gradient $\beta$-number.

\begin{definition} Let $\Gamma = \{w \cdot \phi(w) : w \in \W \}$ be an intrinsic graph, where $\phi \colon \W \to
\V$ is an intrinsic Lipschitz function. Fix a point $x \in \Gamma$
and a radius $r > 0$. Then, for $L \geq 1$, define
\begin{displaymath} \beta_{\CG}(B(x,r)) := \beta_{\CG,\Gamma,L}(B(x,r)) := \inf_{\psi} \sup_{w \in \pi_{\W}(B(x,r) \cap \Gamma)} \frac{|\phi(w) - \psi(w)|}{r}. \end{displaymath}
The infimum is taken over all intrinsic $L$-Lipschitz functions
$\psi: \W \to \V$ which have
intrinsic gradient constant almost everywhere on the set $\pi_{\W}(B(x,b_L r))$. The class of such "admissible" functions $\psi$ will be denoted by
\begin{displaymath} \Adm(B(x,r)) := \Adm_{\CG,L}(B(x,r)). \end{displaymath}
If the Lipschitz constant $L$ is clear from the context, we omit the subscript $L$ for the constant $b_L$.
Note that $\|\nabla^{\phi}\phi\|_{\infty} \leq L$ for $\phi \in \Adm(B(x,r))$ by Lemma \ref{l:bound_grad}. \end{definition}
\begin{remark}
	\label{remk:bcg} Observe that
\begin{displaymath} |\phi(w)-\psi(w)| = \|\Psi(w)^{-1} \cdot \Phi(w)\| = d_{\He}(\Psi(w),\Phi(w)) \end{displaymath}
for $w \in \pi_{\W}(B(x,r) \cap \Gamma)$, where $\Psi$ and $\Phi$ are the graph mappings
\begin{displaymath} \Psi(w) := w \cdot \psi(w)  \quad \text{and} \quad \Phi(w) = w \cdot \phi(w)\in \Gamma. \end{displaymath}
Thus, if $\beta_{\CG}(B(x,r)) < \varepsilon$, there exists $\psi \in \Adm(B(x,r))$ with graph $\Gamma^{\psi}$ such that
\begin{displaymath} \sup_{y \in \Gamma \cap B(x,r)} \frac{\dist_{\He}(y,\Gamma^{\psi})}{r} \leq \varepsilon. \end{displaymath}
\end{remark}

The aim of this section is to prove the following weak geometric lemma for the constant gradient $\beta$-numbers:

\begin{thm}\label{t:CG-WGL} Let $\Gamma$ be an intrinsic
$L$-Lipschitz graph over a vertical subgroup.  Then
\begin{displaymath}
\int_{0}^{R} \int_{\Gamma \cap B(x,R)} \chi_{\{(y,s)\in \Gamma
\times \mathbb{R}_+:\; \beta_{\CG}(B(y,s))\geq \varepsilon\}}(y,s)
d\calH^{3}(y) \, \frac{ds}{s} \lesssim_{\varepsilon} R^{3}
\end{displaymath}
for any $\varepsilon > 0$, $x \in \Gamma$ and $R > 0$. Here $\beta_{\CG}(B(y,s)) := \beta_{\CG,\Gamma,L}(B(y,s))$.
\end{thm}

As explained in Section \ref{outline}, a large $\beta_{\CG}$ number implies that $\nabla^{\phi} \phi$ fluctuates locally. More precisely, Proposition \ref{dreamProp} will show that if $\beta_{\CG}(B(x,r)) \geq \varepsilon$ for some $x \in
\Gamma$ and $r
> 0$, then there exists another ball $B(y,s) \subset B(x,r)$ such $\dist_{\He}(y,\Gamma) \leq s/10$, $s \geq \delta_{\varepsilon,L} r$ and $|\E_{\pi_{\W}(B(x,r) \cap \Gamma)}\nabla^{\phi}\phi - \E_{\pi_{\W}(B(y,s) \cap \Gamma)}\nabla^{\phi}\phi| \geq \delta_{\varepsilon,L}$. In Section \ref{ss:WGL_CG}, we use this  to prove Theorem \ref{t:CG-WGL}.

\subsubsection{Auxiliary results} Before stating Proposition \ref{dreamProp}, we record a few lemmas. The first one gives an upper bound on how much $\E_{\pi_{\W}(B(x,r) \cap \Gamma)} \nabla^{\phi} \phi$ can change as a function of $r$. Here and in the following, we employ the notation
\begin{displaymath}
\E_A f = \frac{1}{\mathcal{L}^2(A)}\int_A f \; d \mathcal{L}^2
\end{displaymath}
for the average of a function $f$ over a set $A$ in the plane.

\begin{lemma}\label{annulusControl} Assume that $\Gamma$ is an intrinsic Lipschitz graph defined over $\W$, and that $f \in L^{\infty}(\W)$. Further, let $x \in \He$, and $0 < s_{1} \leq s_{2} < \infty$. Then,
\begin{displaymath} |\E_{\pi_{\W}(B(x,s_{1}) \cap \Gamma)} f - \E_{\pi_{\W}(B(x,s_{2}) \cap \Gamma)}f| \lesssim \frac{\calH^{3}(A(x,s_{1},s_{2}) \cap \Gamma)}{\calL^{2}(\pi_{\W}(B(x,s_{2}) \cap \Gamma))} \cdot \|f\|_{\infty}, \end{displaymath}
where $A(x,s_{1},s_{2})$ is the annulus $\{y \in \He : s_{1} \leq d_{\He}(x,y) \leq s_{2}\}$.
\end{lemma}
\begin{proof} Write $B_{1} := \pi_{\W}(B(x,s_{1}) \cap \Gamma)$ and $B_{2} := \pi_{\W}(B(x,s_{2}) \cap \Gamma)$. In this proof, let $|U| := \calL^{2}(U)$ for $U \subset \W$. Then,
\begin{align*} |\E_{B_{1}} f - \E_{B_{2}}f| & = \frac{1}{|B_{1}|} \left| \int_{B_{1}} f\,d\mathcal{L}^2 - \frac{|B_{1}|}{|B_{2}|} \int_{B_{2}} f\,d\mathcal{L}^2 \right|\\
& = \frac{1}{|B_{1}|} \left|\left(1 - \frac{|B_{1}|}{|B_{2}|} \right) \int_{B_{1}} f\,d\mathcal{L}^2 - \frac{|B_{1}|}{|B_{2}|} \int_{B_{2} \setminus B_{1}} f \,d\mathcal{L}^2\right|\\
& \leq \frac{|B_{2} \setminus B_{1}|}{|B_{2}|} \cdot \|f\|_{\infty} + \frac{|B_{2} \setminus B_{1}|}{|B_{2}|} \cdot \|f\|_{\infty}  \end{align*}
by the triangle inequality. Next, we observe that
\begin{displaymath} B_{2} \setminus B_{1} = \pi_{\W}([B(x,s_{2}) \setminus B(x,s_{1})] \cap \Gamma) \subset \pi_{\W}(A(x,s_{1},s_{2}) \cap \Gamma) \end{displaymath}
by the injectivity of $\pi_{\W}$ restricted to $\Gamma$. Finally, we use the fact that $\calH^{3}(\pi_{\W}(A)) \leq C\calH^{3}(A)$, recall Lemma \ref{lipschitzLemma}.
\end{proof}

It is desirable to have quantitative control on the
upper bound appearing in Lemma \ref{annulusControl}. This motivates the
following definition:

\begin{definition} Let $(X,d,\mu)$ be a metric measure space. A ball $B(x,r) \subset X$ has \emph{$A$-thin boundary} (with respect to $\mu$) if the following holds:
\begin{displaymath} \mu\left(B(x,2r) \cap A(x,(1 - \lambda)r,(1 + \lambda)r)\right) \leq A\,\lambda \,\mu(B(x,2r)), \qquad \lambda > 0. \end{displaymath}
Here $A(x,s,t) := \{y : s \leq d(x,y) \leq t\}$.
\end{definition}

Balls with thin boundary are abundant:

\begin{lemma}\label{l:A-small}
 Let $(X,d,\mu)$ be metric measure space. For any $0 < \delta < 1/4$, there exists a constant $A = A_{\delta} < \infty$ with the following property: for any ball $B(x,r)$ with $x \in X$ and $r > 0$, there exists a radius $s \in [r,(1 + \delta)r]$ such that $B(x,s)$ has $A$-thin boundary.
\end{lemma}

\begin{proof} For a fixed $x\in X$, let $\pi \colon X \to [0,\infty)$ be the mapping $\pi(y) := d(x,y)$, and consider the push-forward measure $\nu := \pi_{\sharp}[\mu|_{B(x,2r)}]$. Let $M$ be the usual centred Hardy-Littlewood maximal operator on $\R$, namely
\begin{displaymath} Mf(s) := \sup_{t > 0} \frac{1}{2t}\int_{B(s,t)} |f(y)| \, dy. \end{displaymath}
We extend the definition from $f$ to $\nu$ in the obvious way. It is well-known (see for instance \cite[Theorem 2.19]{M}) that $M$ is weakly bounded in the sense that
\begin{displaymath} \calH^{1}(\{s : M\nu(s) > A\}) \lesssim \frac{\|\nu\|}{A} = \frac{\mu(B(x,2r))}{A}. \end{displaymath}
In particular,
\begin{displaymath} \calH^{1}\left(\left\{s : M\nu(s) > A \frac{\mu(B(2x,r))}{r}\right\} \right) \lesssim \frac{r}{A}. \end{displaymath}
For $A = A_{\delta} \geq 1$ large enough, this implies that there is some $s \in [r,(1 + \delta)r]$ such that
\begin{displaymath} \sup_{\lambda > 0} \frac{\nu(B(s,\lambda s))}{2\lambda s} = M\nu(s) \leq A\frac{\mu(B(2x,r))}{r}. \end{displaymath}
Recalling the definition of $\nu$ this means precisely that
\begin{displaymath} \mu(A(x,(1 - \lambda)s,(1 + \lambda)s)) \leq \frac{As}{r} \cdot \mu(B(x,2r)) \leq 2A\mu(B(x,2r)) \end{displaymath}
for all such $\lambda$ that $A(x,(1 - \lambda)s,(1 + \lambda)s)
\subset B(x,2r)$. Since $\delta < 1/4$, this covers all $0 <
\lambda < 1/2$. For $\lambda > 1/2$, the thin boundaries condition
is trivial.  \end{proof}

\subsubsection{Fluctuation of the intrinsic gradient}\label{ss:fluctuation} We are now ready to state our main technical milestone on the way to Theorem \ref{t:CG-WGL}:

\begin{proposition}\label{dreamProp} For every $\varepsilon > 0$ and $L \geq 1$ there exist constants $A = A_{\varepsilon,L} \geq 1$ and $\delta = \delta_{\varepsilon,L} > 0$ with the following property. Assume that $\phi \colon \W \to \V$ is an intrinsic $L$-Lipschitz function with graph $\Gamma$. Assume that $x \in \Gamma$ and $r > 0$ are such that
\begin{displaymath} \beta_{\CG}(B(x,r)) := \beta_{\CG,\Gamma,L}(B(x,r)) \geq \varepsilon, \end{displaymath}
Then, there exists a ball $B(y,s) \subset B(x,r)$ with $A$-thin boundary (with respect to $\calH^{3}|_{\Gamma}$) such that $s \geq \delta r$, $\dist_{\He}(y,\Gamma) \leq s/10$, and
\begin{equation}\label{form29} |\E_{\pi_{\W}(B(y,s) \cap \Gamma)} \nabla^{\phi} \phi - \E_{\pi_{\W}(B(x,r) \cap \Gamma)} \nabla^{\phi} \phi| \geq \delta > 0. \end{equation}
In particular,
\begin{equation}\label{lastStatement} \calL^{2}(\pi_{\W}(B(y,s) \cap \Gamma)) \gtrsim_{\varepsilon,L} r^{3}. \end{equation}
\end{proposition}

\begin{remark}\label{remark0} Observe that \eqref{lastStatement} is an immediate consequence of $y$ being contained in the $s/10$-neighborhood of $\Gamma$; this implies that $B(y,s) \cap \Gamma$ contains a set of the form $B(y',s') \cap \Gamma$ with $y' \in \Gamma$ and $s' \sim s$. Further, for such balls $B(y',s')$ centred on $\Gamma$, we can apply Lemma  \ref{eq:ball_inclusion} to find
\begin{displaymath} \pi_{\W}(B(y',s') \cap \Gamma) \supset \pi_{\W}(B(y',s'')) \end{displaymath}
for some $s'' \sim_{L} s'$.
Finally,
\begin{displaymath} \calL^{2}(\pi_{\W}(B(y',s''))) = c(s'')^{3} \end{displaymath}
with $c = \calL^{2}(\pi_{\W}(B(0,1)) > 0$, recalling that the mapping $P_{p} \colon w \mapsto \pi_{\W}(p \cdot w)$, has unit Jacobian for any $p \in \He$, see \eqref{jacobian}. See also  \cite[Lemma 3.14]{FSS}. This concludes the proof of \eqref{lastStatement}

We also explain here, why a set $E \subset \He$ with BPiLG has big vertical projections. This is an immediate consequence of Lemma \ref{eq:ball_inclusion} if $E$ is an intrinsic Lipschitz graph over an (entire) vertical plane $\W$. In the general case, one can easily deduce the BVP property from the area formula for intrinsic Lipschitz functions (Theorem 1.6 in \cite{CMPS}). 
\end{remark}

In the following subsections, we proceed with proving the
remaining statements of Proposition \ref{dreamProp}. The outline
is the following:
\begin{itemize}
\item[\ref{ss:blowup}] We formulate a counter assumption to the
main claim in Proposition \ref{dreamProp}. Assuming the validity
of this assumption, we find $L\geq 1$ and a sequence of intrinsic
$L$-Lipschitz functions $(\phi_j)_j$ and associated graphs
$(\Gamma_j)_j$, such that $\Gamma_j$ has large $\beta_{\CG}$ number
in a ball $B(x_j,r_j)$ centred on $\Gamma_j$, yet
$\nabla^{\phi_j}\phi_j$ does not fluctuate much in that ball. We
use a blow-up procedure to normalize so that we may assume
$B(x_j,r_j)=B(0,1)$ for all $j$. \item[\ref{ss:limiting
procedure}] We show that a subsequence of $(\phi_j)_j$ converges
locally uniformly to an intrinsic $L$-Lipschitz function $\phi$
with graph $\Gamma$ such that, roughly speaking, \\ (i)
$\beta_{\CG}(B(0,1))$ is large, \\ (ii)
$\mathbb{E}_{\pi_{\W}(B(y,s)\cap \Gamma)}\nabla^{\phi}\phi$ is
independent of $B(y,s)\subset B(0,1)$, $y\in \Gamma$.
\item[\ref{ss:conclusion}] We show that the conditions (i)
and (ii) are incompatible, which concludes the proof of
Proposition \ref{dreamProp}.
\end{itemize}

\subsubsection{The counter assumption} \label{ss:blowup}

Denote by $E(\delta)$ the metric $\delta$-neighborhood of a set $E$.
Given an intrinsic Lipschitz graph $\Gamma$, a ball $B(x,r)$ with
$x \in \Gamma$, and $j \in \N$, define the following collection of
"good" balls $\calG_{j} = \calG_{j}(\Gamma,B(x,r))$. A ball
$B(y,sr) \subset B(x,r)$ is in $\calG_{j}$, if
\begin{itemize}
\item[(a)] $y \in \Gamma(sr/10)$, \item[(b)] $s \geq 2^{-j}$, and
\item[(c)] $B(y,sr)$ has $2^{j}$-thin boundary with respect to
$\calH^{3}|_{\Gamma}$.
\end{itemize}
Then, Proposition \ref{dreamProp} follows, if we can prove the
next statement:

\begin{claim}\label{eq:A}
For every $\varepsilon > 0$, $L \geq 1$, there exists  $j =
j_{\varepsilon,L} \in \N$ with the following property. If $\Gamma$
is any intrinsic $L$-Lipschitz graph and $B(x,r)$ is centred on
$\Gamma$ with $\beta_{\CG}(B(x,r)) \geq \varepsilon$, then there
exists a ball $B = B(y,sr) \in \calG_{j}(\Gamma,B(x,r))$ such that
\begin{displaymath} |\E_{\pi_{\W}(B \cap \Gamma)} \nabla^{\phi} \phi - \E_{\pi_{\W}(B(x,r) \cap \Gamma)} \nabla^{\phi} \phi| > \tfrac{1}{j}. \end{displaymath}
\end{claim}

If the claim fails, then it also fails with $B(x,r) = B(0,1)$. This reduction is the content of the next lemma:

\begin{lemma}\label{l:normalization}
If Claim \ref{eq:A} fails, then there there exist
$\varepsilon
> 0$ and $L \geq 1$ such that for every $j \in \N$ we can find an
intrinsic $L$-Lipschitz graph ${\Gamma}_{j}$, parametrised by
$\phi_{j} \colon \W \to \V$ such that
$$\beta_{\CG,\Gamma_{j}}(B(0,1)) \geq \varepsilon,$$
 yet
\begin{equation}
\label{form8_blowup} \sup_{B \in \calG_{j}(\Gamma_{j},B(0,1))}
|\E_{\pi_{\W}(B \cap \Gamma_j)} \nabla^{\phi_{j}} \phi_{j} -
\E_{\pi_{\W}(B(0,1)\cap \Gamma_j)} \nabla^{\phi_{j}} \phi_{j}|
\leq \tfrac{1}{j}. \end{equation}
\end{lemma}

\begin{proof}
By definition, if Claim \ref{eq:A} fails, then there exist
$\varepsilon > 0$ and $L \geq 1$ such that for every $j \in \N$ we
can find an intrinsic $L$-Lipschitz graph $\Gamma_{j}$,
parametrised by $\phi_{j} \colon \W \to \V$, and
\emph{some ball} $B(x_{j},r_{j})$ centred on $\Gamma_{j}$ such
that
$$\beta_{\CG,\Gamma_{j}}(B(x_{j},r_{j})) \geq \varepsilon,$$
 yet
\begin{equation}\label{form8} \sup_{B \in \calG_{j}(\Gamma_{j},B(x_{j},r_{j}))} |\E_{\pi_{\W}(B \cap \Gamma_j)} \nabla^{\phi_{j}} \phi_{j} - \E_{\pi_{\W}(B(x_{j},r_{j})\cap \Gamma_j)} \nabla^{\phi_{j}} \phi_{j}| \leq \tfrac{1}{j}. \end{equation}

In order to prove Lemma \ref{l:normalization}, we left-translate
$\Gamma_j$ by $x_j^{-1}$ and dilate it by $\delta_{r_j^{-1}}$. The
resulting set $\widetilde{\Gamma}_j$:
\begin{itemize}
\item[(i)] is again an intrinsic $L$-Lipschitz graph (Lemma
\ref{l:transl}),
\item[(ii)] has $\beta_{\CG}$-number at least
$\varepsilon$ on $B(0,1)$ (Lemma \ref{l:beta} below),
\item[(iii)] is
parametrised by a function $\widetilde{\phi}_j : \W \to \V$ so that \eqref{form8} holds with $\widetilde{\phi}_{j}$ in place of $\phi_{j}$.
(Lemma \ref{l:average} below).
\end{itemize} \end{proof}

We now proceed to establish the two auxiliary results, needed in the
proof of Lemma \ref{l:normalization}.
\begin{lemma}\label{l:beta}
Assume that $\phi:\W \to \V$ is an intrinsic Lipschitz function with graph $\Gamma$, $x$ is a point on
$\Gamma$, and $r>0$. Then
\begin{displaymath}
\beta_{\CG,L,\Gamma}(B(x,r))=\beta_{\CG,L,\widetilde{\Gamma}}(B(0,1))
\end{displaymath}
for $\widetilde{\Gamma}= \delta_{r^{-1}}(\tau_{x^{-1}}(\Gamma))$.
\end{lemma}

\begin{proof}

According to Lemma \ref{l:transl} and Remark \ref{remk:para} there exists a uniquely defined intrinsic Lipschitz function $\widetilde{\phi}$ that parametrizes $\widetilde{\Gamma}$. By definition,
\begin{displaymath}
\beta_{\CG,L,\widetilde{\Gamma}}(B(0,1)) = \inf_{\sigma \in \Adm_{\CG,L}(B(0,1))}  \sup_{w\in \pi_{\W}(B(0,1)\cap \widetilde{\Gamma})} |\widetilde{\phi}(w)-\sigma(w)|.
\end{displaymath}
We first aim to prove that the family $\Adm_{\CG,L}(B(0,1)) =: \Adm(B(0,1))$ is in $1$-to-$1$ correspondence with the family $\Adm_{\CG,L}(B(x,r)) =: \Adm(B(x,r))$.
Assume that $\psi \in \Adm(B(x,r))$. Define
 \begin{equation}\label{eq:formula_transdil}
\widetilde{\psi}:\W \to \V,\quad \widetilde{\psi}(w):=
\delta_{\frac{1}{r}}\left(\pi_{\mathbb{V}}\left(x\right)^{-1} \cdot \psi \left(P_x(
\delta_{r}(w))\right)\right),
 \end{equation}
and note that $\tilde{\psi} = (\psi_{x^{-1}})_{r^{-1}}$ in the notation of Lemma \ref{l:transl}. Let us prove that $\tilde{\psi} \in \Adm(B(0,1))$. To this end, we observe first that
\begin{displaymath}
\pi_{\W}\left( x \cdot \delta_r(\pi_{\W}(p))\right)= \pi_{\W} \left(x \cdot \delta_r(p)\right),\quad\text{for all }p\in \mathbb{H},
\end{displaymath}
see for instance \cite[Proposition 2.15]{MSS}.
By homogeneity and left invariance of the distance $d_{\mathbb{H}}$, it follows that
\begin{equation}\label{form28}
P_{x}\left(\delta_{r}\left(\pi_{\W}(B(0,1)\cap
\widetilde{\Gamma})\right)\right)= \pi_{\W}(B(x,r)\cap \Gamma) \end{equation}
and
\begin{displaymath}
P_{x}\left(\delta_{r}\left(\pi_{\W}(B(0,b_L))\right)\right)= \pi_{\W}(B(x,b_L r)).
\end{displaymath}
Lemma \ref{l:transl} then implies that $\tilde{\psi} \in \Adm(B(0,1))$.

Conversely, if $\sigma \in \Adm(B(0,1))$, then $(\sigma_{r})_{x} \in \Adm(B(x,r))$. Thus, by \eqref{form28},
\begin{align*}
\beta_{\CG,L,\widetilde{\Gamma}}(B(0,1)) & = \inf_{\psi \in \Adm(B(x,r))} \sup_{w\in \pi_{\W}(B(0,1) \cap \widetilde{\Gamma})}{|\widetilde{\phi}(w)-\widetilde{\psi}(w)|}\\
& = \inf_{\psi \in \Adm(B(x,r))} \sup_{w \in \pi_{\W}(B(0,1) \cap \widetilde{\Gamma})} r^{-1} |\phi(P_x(\delta_r(w)))-\psi(P_x(\delta_r(w)))|\\
& = \inf_{\psi \in \Adm(B(x,r))} \sup_{w \in \pi_{\W}(B(x,r) \cap \Gamma)} r^{-1} |\phi(w) - \psi(w)|\\
& = \beta_{\CG,L,\Gamma}(B(x,r)). \end{align*}
This completes the proof. \end{proof}

\begin{lemma}\label{l:average}
Assume that $\phi:\W \to \V$ is an intrinsic
Lipschitz function with graph $\Gamma$, $x$ is a point on
$\Gamma$, and $r>0$. Then
\begin{equation}\label{eq:good}
\delta_{r^{-1}} \tau_{x^{-1}} \mathcal{G}_j(\Gamma, B(x,r))=
\mathcal{G}_j(\widetilde{\Gamma}, B(0,1))
\end{equation}
and
\begin{equation}\label{eq:expectation}
\mathbb{E}_{\pi_{\W}(B(y,sr)\cap \Gamma)}\nabla^{\phi}\phi=
\mathbb{E}_{\pi_{\W}(B(\delta_{r^{-1}}(x^{-1}\cdot y),s)\cap
\widetilde{\Gamma})}\nabla^{\widetilde{\phi}}\widetilde{\phi}.
\end{equation}
Here $\widetilde{\Gamma} = \delta_{r^{-1}} (\tau_{x^{-1}}(\Gamma))$
is the graph parametrized by $\widetilde{\phi}$ (defined as in
\eqref{eq:formula_transdil}).
\end{lemma}

\begin{proof}

We start with the first claim. Since the Heisenberg distance is
left invariant with respect to the group law, and homogeneous with
respect to the dilations $(\delta_r)_{r>0}$, if a ball $B(y,sr)$
has $A$-thin boundary, then so does the ball $\delta_{r^{-1}} \tau_{x^{-1}}
B(y,sr)= B(\delta_{r^{-1}}(x^{-1}\cdot y),s)$. The remaining conditions
that one has to verify in order to prove \eqref{eq:good} are also
immediate.

Regarding \eqref{eq:expectation}, we first recall that Lemma
\ref{l:transl} yields that
\begin{displaymath}
\nabla^{\widetilde{\phi}} \widetilde{\phi} = \nabla^{\phi} \phi
\circ P_x \circ \delta_r.
\end{displaymath}
Since $\delta_r$ restricted to $\W$ has Jacobian determinant
constant equal to $r^3$, and $P_x$ has Jacobian determinant
equal to $1$, it follows by the usual transformation formula for
functions on $\mathbb{R}^2$ that
\begin{displaymath}
\calL^{2}(\pi_{\W}(B(y,sr)\cap \Gamma)) = r^3
\calL^{2}(\pi_{\W}(B(\delta_{r^{-1}}(x^{-1}\cdot y),s)\cap \widetilde{\Gamma}))
\end{displaymath}
and
\begin{displaymath}
\int_{\pi_{\W}(B(y,sr)\cap
\Gamma)}\nabla^{\phi}\phi\,d\mathcal{L}^2 = r^{3}
\int_{\pi_{\W}(B(\delta_{r^{-1}}(x^{-1}\cdot y),s)\cap
\widetilde{\Gamma})}\nabla^{\widetilde{\phi}}\widetilde{\phi}\,d\mathcal{L}^2.
\end{displaymath}
This establishes \eqref{eq:expectation}.
\end{proof}

\subsubsection{Limiting procedure}\label{ss:limiting procedure}

In this section, we work under the standing (counter) assumption
to Proposition \ref{dreamProp}. In particular, we may assume by
Lemma \ref{l:normalization} that there exists $\varepsilon,L>0$
and  a sequence $(\phi_j)_{j \in \N}$ of intrinsic $L$-Lipschitz functions
with graphs $(\Gamma_j)_j$ such that $\beta_{\CG,\Gamma_{j},L}(B(0,1)) \geq \varepsilon$, yet the intrinsic gradient
$\nabla^{\phi_j}\phi_j$ fluctuates only little in $B(0,1)\cap
\Gamma_j$ as quantified in \eqref{form8_blowup}.

The main goal of this section is to consider an "accumulation point" $\phi$ of the sequence $(\phi_j)_j$, and to
discuss how the properties of the maps $\phi_j$ carry over to
$\phi$.

\begin{lemma}\label{eq:limit_beta}
The sequence $(\phi_j)_j$ defined above contains a subsequence
that converges locally uniformly on $\W$ to an intrinsic
$L$-Lipschitz function $\phi:\W \to \V$ with graph $\Gamma$ such
that
\begin{equation}\label{form9} \beta_{\CG,\Gamma,L}(B(0,1)) \geq \varepsilon. \end{equation}
\end{lemma}

\begin{proof} Since each graph $\Gamma_{j}$ by construction contains the origin, we have
$\phi_{j}(0) = 0$ for every $j$. Therefore \ref{intrili2}  implies that the family
$(\phi_{j})_{j \in \N}$ is locally equibounded. Hence, by
Proposition 3.10 in \cite{FS}, there exists a subsequence which
converges locally uniformly to an intrinsic $L$-Lipschitz function
$\phi$ on $\W$ with intrinsic graph $\Gamma$. For simplicity, we also denote this subsequence
by $(\phi_j)_{j \in \N}$.

Let $\delta > 0$. In order to prove that $\beta_{\CG,\Gamma,L}(B(0,1)) \geq \varepsilon$, it suffices to fix $\psi \in \Adm_{\CG,L}(B(0,1))$ and find a point $w_{\psi} \in \pi_{\W}(\Gamma \cap B(0,1))$ with $|\psi(w_{\psi}) - \phi(w_{\psi})| \geq (1 - \delta)\varepsilon$. To this end, the assumption
\begin{displaymath}
\beta_{\CG,\Gamma_j,L}(B(0,1)) \geq \varepsilon,\quad\text{for all }j\in \N,
\end{displaymath}
implies that, for each $j \in \N$, there exists a point $w_{\psi}^{j}  \in \pi_{\W}(\Gamma_j \cap B(0,1))$ such that
\begin{displaymath}
|\psi(w_{\psi}^{j})-\phi_j(w_{\psi}^{j})|\geq (1-\delta)\varepsilon.
\end{displaymath}
Write $p_{\psi}^{j} := w_{\psi}^{j} \cdot \phi_{j}(w_{\psi}^{j})$. The sequence $(p^j_{\psi})_{j\in \N}$ has a subsequence $(p^{j_k}_{\psi})_{k\in \N}$ convergent to a point $p_{\psi} \in \Gamma \cap B(0,1)$. Since $\pi_{\W}$ is continuous, the points $w_{\psi}^{j_{k}} = \pi_{\W}(p_{\psi}^{j_{k}})$ converge to $w_{\psi} := \pi_{\W}(p_{\psi})\in \pi_{\W}(\Gamma \cap B(0,1))$. 
\begin{displaymath} |\psi(w_{\psi}) - \phi(w_{\psi})| = \lim_{k \to \infty} |\psi(w_{\psi}^{j_{k}}) - \phi_{j_{k}}(w^{j_{k}}_{\psi})| \geq (1 - \delta)\varepsilon \end{displaymath}
by the continuity of $\psi$, and the locally uniform convergence $\phi_{j} \to \phi$. The proof is complete.
\end{proof}

Without loss of generality, we assume in the following that the
whole sequence $(\phi_{j})_{j \in \N}$ converges locally uniformly
to $\phi$. Our next goal is to prove the following convergence
result for the corresponding intrinsic gradients.

\begin{lemma}\label{l:conv_grad}
Let $(\phi_j)_j:\W \to \V$ be a sequence of intrinsic
$L$-Lipschitz functions converging locally uniformly to an
$L$-Lipschitz function $\phi:\W \to \V$. Then
\begin{equation}\label{form10} \int_{\pi_{\W}(B(y,s)\cap \Gamma)} \nabla^{\phi}\phi \, d\calL^{2} = \lim_{j \to \infty} \int_{\pi_{\W}(B(y,s)\cap \Gamma_j)} \nabla^{\phi_{j}}\phi_{j} \, d\calL^{2} \end{equation}
for all balls $B(y,s) \subset B(0,1)$ with $y\in \Gamma$, and such
that $\mathcal{H}^3(\partial B(y,s)\cap \Gamma)=0$.
\end{lemma}

The subtle point here is, of course, that the domain of integration is different on both sides of the above equation.
Therefore the following auxiliary result will be useful in the proof of Lemma \ref{l:conv_grad}.

\begin{lemma}\label{annulusControl2} Assume that $\Gamma_{j}$ is a sequence of intrinsic $L$-Lipschitz graphs, which converges locally in the Hausdorff metric in $\He$ to an intrinsic Lipschitz graph $\Gamma$. Then, for any $x \in \Gamma$ and $0 < r < s < \infty$, we have
\begin{displaymath} \limsup_{j \to \infty} \calH^{3}(A(x,r,s) \cap \Gamma_{j}) \lesssim_{L} \calH^{3}(A(x,r,s) \cap \Gamma). \end{displaymath}
\end{lemma}

\begin{proof} Fix $\delta > 0$. First, pick $r_{1} < r$ and $s_{1} > s$ so that
\begin{displaymath} \calH^{3}(A(x,r_{1},s_{1}) \cap \Gamma) \leq \calH^{3}(A(x,r,s) \cap \Gamma) + \delta. \end{displaymath}
Then, let $\varepsilon > 0$ be so small that if $y \in A(x,r,s)$,
then $B(y,20\varepsilon) \subset A(x,r_{1},s_{1})$. Let
$\{y_{1},\ldots,y_{N}\}$ be an $\varepsilon$-net in $A(x,r,s)$,
and, finally, let $Y_{j}$ be the subset of points $y$ in this net
with the property that $B(y,5\varepsilon)$ contains a point in
$A(x,r,s) \cap \Gamma_{j}$. Then
\begin{displaymath} \calH^{3}(A(x,r,s) \cap \Gamma_{j}) \leq \sum_{y \in Y_{j}} \calH^{3}(B(y,5\varepsilon) \cap \Gamma_{j}) \lesssim_{L} |Y_{j}| \cdot \varepsilon^{3}, \end{displaymath}
by the $3$-regularity of intrinsic $L$-Lipschitz graphs. For $j$ large enough, every ball $B(y,10\varepsilon)$ with $y \in
Y_{j}$ also contains a point in $\Gamma$, whence $\varepsilon^{3}
\lesssim_{L} \calH^{3}(B(y,20\varepsilon) \cap \Gamma)$, $y \in
Y_{j}$. Since the balls $B(y,20\varepsilon) \subset A(x,r_1,s_1)$ have
bounded overlap (independent of $\varepsilon > 0$), we conclude
that
\begin{displaymath} \calH^{3}(A(x,r,s) \cap \Gamma_{j}) \lesssim_{L} \calH^{3}(A(x,r_{1},s_{1}) \cap \Gamma) \leq \calH^{3}(A(x,r,s) \cap \Gamma) + \delta \end{displaymath}
for all large enough $j$. This completes the proof. \end{proof}

\begin{proof}[Proof of Lemma \ref{l:conv_grad}]

We employ the fact proven in \cite[Proposition 4.7]{CMPS} that the
intrinsic gradient is also a distributional gradient for intrinsic
Lipschitz functions: if $\phi$ is intrinsic Lipschitz and defined
on $\W$, then
\begin{displaymath} \int_{\W} (\phi \partial_{y} \psi +\tfrac{1}{2} \phi^{2} \partial_{t} \psi) \, d\calL^{2} = -\int_{\W} [\nabla^{\phi} \phi] \psi \, d\calL^{2} \end{displaymath}
for all compactly supported $\mathcal{C}^{1}$-functions $\psi$ on $\W$ (here
we assume that $\W$ is the $(y,t)$-plane, as we may). Since
uniform convergence implies weak convergence, we infer that
\begin{equation}\label{form32} \int_{\W} [\nabla^{\phi} \phi] \psi \, d\calL^{2} = \lim_{j \to \infty} \int_{\W} [\nabla^{\phi_{j}}\phi_{j}] \psi \, d\calL^{2} \end{equation}
for all compactly supported $\mathcal{C}^{1}$ functions $\psi \colon \W \to
\V$. It remains to deduce from this statement the claim
\eqref{form10}.

To achieve this, we recall from Lemma \ref{l:bound_grad} that
\begin{equation}\label{eq:norm_bound}
\|\nabla^{\phi_j}\phi_j\|_{\infty}
\leq L,\quad j\in\mathbb{N},\quad \text{and} \quad
\|\nabla^{\phi}\phi\|_{\infty} \leq L.
\end{equation}
Then, given $B(y,s) \subset B(0,1)$ with $y\in \Gamma$, for
$\varepsilon>0$, choose an open set $U$ in
$\pi_{\mathbb{W}}(B(0,2))$ so that
 \begin{displaymath}
\overline{\pi_{\mathbb{W}}(B(y,s)\cap \Gamma)} \subset U,
\end{displaymath}
and $\calL^{2}(U \setminus \pi_{\W}(B(y,s) \cap \Gamma)) \leq \varepsilon/4L$. In particular,
\begin{equation}\label{eq:int_bound1}
\left|\int_{U\setminus \pi_{\mathbb{W}}(B(y,s)\cap \Gamma)}
\left(\nabla^{\phi_j}\phi_j -\nabla^{\phi}\phi \right)
\psi\;{d}\mathcal{L}^2\right| \leq \tfrac{\varepsilon}{2}
\end{equation}
for all test functions $\psi$ with $0\leq \psi \leq 1$ and all
$j\in \mathbb{N}$. Now, let $\psi$ be a smooth cut-off function
with compact support in $U$, $0\leq \psi \leq 1$ on $U$ and such
that $\psi \equiv 1$ on $\pi_{\mathbb{W}}(B(y,s)\cap \Gamma)$.
Then, by \eqref{form32}, there exists $j(\varepsilon)\in \mathbb{N}$ such that for
$j\geq j(\varepsilon)$, one has
\begin{equation}\label{eq:int_bound2}
\left| \int_{\W} \left(\nabla^{\phi_j}\phi_j
-\nabla^{\phi}\phi\right) \psi \;{d}\mathcal{L}^2 \right| \leq
\tfrac{\varepsilon}{2}.
\end{equation}
Let us denote $B:=\pi_{\mathbb{W}}(B(y,s)\cap \Gamma)$. Combining \eqref{eq:int_bound1} and \eqref{eq:int_bound2}, and
recalling that $\psi$ is supported on $U$, we find that
\begin{align*}
\left|\int_{B}\nabla^{\phi_j}\phi_j-\nabla^{\phi}\phi{d}\mathcal{L}^2\right|=
\left|\int_{U}\left(\nabla^{\phi_j}\phi_j-\nabla^{\phi}\phi\right)\psi
{d}\mathcal{L}^2- \int_{U \setminus B}\left(\nabla^{\phi_j}\phi_j-\nabla^{\phi}\phi\right)\psi
{d}\mathcal{L}^2\right| \leq \varepsilon
\end{align*}
for $j\geq j(\varepsilon)$. This yields
\begin{equation*}\int_{\pi_{\W}(B(y,s)\cap \Gamma)} \nabla^{\phi}\phi \, d\mathcal{L}^2 = \lim_{j \to \infty} \int_{\pi_{\W}(B(y,s)\cap \Gamma)} \nabla^{\phi_{j}}\phi_{j} \, d\mathcal{L}^2. \end{equation*}

To establish \eqref{form10}, it remains to show that $\Gamma$ on
the right hand side of the above equation can be replaced by
$\Gamma_j$. This follows immediately from the uniform bound for
$\nabla^{\phi}\phi$ and $\nabla^{\phi_j}\phi_j$, see
\eqref{eq:norm_bound}, provided we can show the following:
\begin{equation}\label{eq:symm_diff}
\calL^{2}(\pi_{\mathbb{W}}(B(y,s)\cap\Gamma_j)\triangle
\pi_{\mathbb{W}}(B(y,s)\cap\Gamma)) \to 0,\quad\text{as }j\to
\infty.
\end{equation}
We now prove \eqref{eq:symm_diff}. Here we need to assume that
$\mathcal{H}^3(\partial B(y,s)\cap \Gamma)=0$. Since
\begin{displaymath}
\partial B(y,s)\cap \Gamma = \bigcap_{n\in \mathbb{N}} A(y,
(1-\tfrac{1}{n})s,s) \cap \Gamma,
\end{displaymath}
there exists for every $\delta>0$ a number $n\in \mathbb{N}$ such
that
\begin{equation}\label{eq:measure_annulus}
\mathcal{H}^3\left(A(y, (1-\tfrac{1}{n})s,s) \cap \Gamma\right) \leq \delta.
\end{equation}
Let $n=n(\delta)$ be such and observe that
\begin{displaymath} \pi_{\W}(B(y,s) \cap \Gamma) \subset \pi_{\W}(B(y,(1 - \tfrac{1}{n})s) \cap \Gamma) \cup \pi_{\W}(A(y,(1 - \tfrac{1}{n})s,s) \cap \Gamma) \end{displaymath}
and
\begin{displaymath} \pi_{\W}(B(y,(1 - \tfrac{1}{n})s) \cap \Gamma) \subset \pi_{\W}(B(y,s) \cap \Gamma_{j}) \end{displaymath}
for  $j$ larger than some $j(\delta)$ given by uniform
convergence. So
\begin{displaymath} \pi_{\W}(B(y,s) \cap \Gamma) \setminus \pi_{\W}(B(y,s) \cap \Gamma_{j}) \subset \pi_{\W}(A(y,(1 - \tfrac{1}{n})s,s) \cap \Gamma) \end{displaymath}
for $j \geq j(\delta)$. Since
\begin{displaymath} \calL^{2}(\pi_{\W}(A(y,(1 - \tfrac{1}{n})s,s) \cap \Gamma)) \lesssim \calH^{3}(A(y,(1 - \tfrac{1}{n})s,s) \cap \Gamma)) \leq \delta, \end{displaymath}
by the general inequality $\calL^{2}(\pi_{\W}(B)) \lesssim \calH^{3}(B)$ (Lemma \ref{lipschitzLemma})
and by \eqref{eq:measure_annulus}, we infer that
\begin{displaymath} \calL^{2}\left(\pi_{\W}(B(y,s) \cap \Gamma) \setminus \pi_{\W}(B(y,s) \cap \Gamma_{j})\right) \to 0, \end{displaymath}
as $j \to \infty$. It remains to prove the same with the roles of
$\Gamma$ and $\Gamma_{j}$ reversed. Repeating the argument above,
it suffices to estimate the measure of the projection
$\pi_{\W}(A(y,(1 - \tfrac{1}{n})s,s) \cap \Gamma_{j})$. This is
not altogether trivial, since the assumption on $\partial B(y,s)$
concerns $\mathcal{H}^3|_{\Gamma}$, not
$\mathcal{H}^3|_{\Gamma_{j}}$. However, Lemma
\ref{annulusControl2} still implies that
\begin{displaymath} \limsup_{j \to \infty} \calH^{3}(A(y,(1 - \tfrac{1}{n})s,s) \cap \Gamma_{j}) \lesssim_{L} \calH^{3}(A(y,(1 - \tfrac{1}{n})s,s) \cap \Gamma) \leq \delta, \end{displaymath}
which proves that
\begin{displaymath} \limsup_{j \to \infty} \calL^{2}(\pi_{\W}(B(y,s) \cap \Gamma_{j}) \setminus \pi_{\W}(B(y,s) \cap \Gamma)) \lesssim \delta. \end{displaymath}
Since $\delta > 0$ is arbitrary, \eqref{eq:symm_diff}, and hence
also \eqref{form10}, follows. This concludes the proof of Lemma
\ref{l:conv_grad}.
\end{proof}

We now return to the specific sequence $(\phi_j)_{j \in \N}$ and limit
function $\phi$ (with graph $\Gamma$) given by Lemma
\ref{eq:limit_beta}. In particular, we work under the standing
assumptions that $\beta_{\CG,\Gamma_j}(B(0,1))\geq \varepsilon$
and $\nabla^{\phi_j}\phi_j$ fluctuates only little on
$\pi_{\W}(B(0,1)\cap \Gamma_j)$, as made precise in Lemma
\ref{l:normalization}.

\begin{lemma}\label{l:a.s.const}
Let $\phi$ be as in Lemma \ref{eq:limit_beta}.  Then $\E_{\pi_{\W}(B(y,s))}
\nabla^{\phi}\phi$ is a constant independent of $B(y,s) \subset
B(0,1)$, for all balls $B(y,s)$ such that
 $y\in \Gamma$ and $\calH^{3}(\partial B(y,s) \cap \Gamma) = 0$.
\end{lemma}

\begin{proof}
Pick two  balls $B(y_{1},s_{1}) \subset B(0,1)$ and
$B(y_{2},s_{2}) \subset B(0,1)$ with $y_1,y_2 \in \Gamma$
satisfying the assumption stated in the lemma, and let $\delta >
0$ be arbitrary. Our goal is to show that
\begin{displaymath}
|\E_{\pi_{\W}(B(y_{1},s_{1})\cap\Gamma)} \nabla^{\phi} \phi - \E_{\pi_{\W}(B(y_{2},s_{2})\cap \Gamma)}\nabla^{\phi} \phi| \lesssim \delta.
\end{displaymath}
We start by applying the triangle inequality:
\begin{displaymath}
|\E_{\pi_{\W}(B(y_{1},s_{1})\cap\Gamma)} \nabla^{\phi} \phi - \E_{\pi_{\W}(B(y_{2},s_{2})\cap \Gamma)}\nabla^{\phi} \phi|\leq D_1 + D_2 +D_3,
\end{displaymath}
where
\begin{align*}
D_1&:= |\E_{\pi_{\W}(B(y_{1},s_{1})\cap \Gamma)} \nabla^{\phi} \phi - \E_{\pi_{\W}(B(y_{1},s_{1})\cap \Gamma_j)} \nabla^{\phi_{j}} \phi_{j}|,\\
D_2&:=|\E_{\pi_{\W}(B(y_{2},s_{2})\cap \Gamma)} \nabla^{\phi} \phi - \E_{\pi_{\W}(B(y_{2},s_{2})\cap \Gamma_j)} \nabla^{\phi_{j}} \phi_{j}|,\\
D_3&:=|\E_{\pi_{\W}(B(y_{1},s_{1})\cap \Gamma_j)} \nabla^{\phi_{j}} \phi_{j} - \E_{\pi_{\W}(B(y_{2},s_{2})\cap \Gamma_j)} \nabla^{\phi_{j}} \phi_{j}|.
\end{align*}
To see that $D_{1} +
D_{2} \leq \delta$ for large enough $j$, writing temporarily $|U| := \calL^{2}(U)$ for $U \subset \W$, and
\begin{displaymath}
B:=\pi_{\W}(B(y_{1},s_{1})\cap \Gamma)\quad \text{and} \quad B_j:=\pi_{\W}(B(y_{1},s_{1})\cap \Gamma_j),
\end{displaymath}
we perform the following estimate:
\begin{align*}
D_1&=|\E_{B} \nabla^{\phi} \phi - \E_{B_j} \nabla^{\phi_{j}} \phi_{j}|
\\& = \frac{1}{|B|}\left|\int_{B}\nabla^{\phi}\phi\, d \mathcal{L}^2 - \int_{B_j} \nabla^{\phi_j}\phi_j \,d \mathcal{L}^2\right|
+ \frac{\left||B|-|B_j|\right|}{|B||B_j|}\left|\int_{B_j} \nabla^{\phi_j} \phi_j\,d \mathcal{L}^2\right|\\
&\leq \frac{1}{|B|}\left|\int_{B}\nabla^{\phi}\phi d \mathcal{L}^2 - \int_{B_j} \nabla^{\phi_j}\phi_j d \mathcal{L}^2\right|+ \frac{
|B_j\triangle
B|}{|B|}\|\nabla^{\phi_{j}} \phi_{j}\|_{\infty}.
\end{align*}
The same is true with "1" replaced by "2". By
\eqref{form10} and \eqref{eq:symm_diff}, and the uniform bound
$\|\nabla^{\phi_{j}} \phi_{j}\|_{\infty} \leq L$, both terms above
tend to zero as $j \to \infty$. So, it suffices to deal with
$D_{3}$.

We first make the subsequent estimate for certain
$s_{1}^{j} \geq s_{1}$ and $s_{2}^{j} \geq s_{2}$ to be fixed
soon:
\begin{displaymath}
D_3 \leq E_1 + E_2 + E_3,
\end{displaymath}
where
\begin{align*}
E_1&=|\E_{\pi_{\W}(B(y_{1},s_{1})\cap \Gamma_j)} \nabla^{\phi_{j}} \phi_{j} - \E_{\pi_{\W}(B(y_{1},s_{1}^{j}) \cap \Gamma_{j})} \nabla^{\phi_{j}}\phi_{j}|,\\
E_2&=|\E_{\pi_{\W}(B(y_{2},s_{2})\cap \Gamma_j)} \nabla^{\phi_{j}} \phi_{j} - \E_{\pi_{\W}(B(y_{2},s_{2}^{j}) \cap \Gamma_{j})} \nabla^{\phi_{j}}\phi_{j}|,\\
E_3&=|\E_{\pi_{\W}(B(y_{1},s_{1}^{j}) \cap \Gamma_{j})} \nabla^{\phi_{j}}\phi_{j} - \E_{\pi_{\W}(B(y_{2},s_{2}^{j}) \cap \Gamma_{j})} \nabla^{\phi_{j}}\phi_{j}|.
\end{align*}
It follows from the key
(counter) assumption \eqref{form8_blowup} that $E_{3} \leq 2/j$
if the balls $B(y_{1},s_{1}^{j})$ and $B(y_{2},s_{2}^{j})$ belong
to $\calG_{j}(\Gamma_{j},B(0,1))$. This requires that
\begin{itemize}
\item[(a)] $\dist(y_{1},\Gamma_{j}) \leq s_{1}/10$ and
$\dist(y_{2},\Gamma_{j}) \leq s_{2}/10$, \item[(b)] $s_{1}^{j}
\geq 2^{-j}$ and $s_{2}^{j} \geq 2^{-j}$, and \item[(c)]
$B(y_{1},s_{1}^{j})$ and $B(y_{2},s_{2}^{j})$ have $2^{j}$-thin
boundary with respect to $\calH^{3}|_{\Gamma_{j}}$.
\end{itemize}
Condition (a) is automatically satisfied for large enough $j$,
since $y_{1},y_{2} \in \Gamma$. Condition (b) is also trivially
satisfied for large enough $j$, since $s_{1}^{j} \geq s_{1}$ and
$s_{2}^{j} \geq s_{2}$. To achieve condition (c), fix an auxiliary
parameter $\tau > 0$, which will depend on
$\delta,y_{1},y_{2},s_{1},s_{2},\Gamma$ and $L$. Then, for large
enough $j$, Lemma \ref{l:A-small} guarantees the existence of
$s_{1}^{j} \in [s_{1},(1 + \tau)s_{1}]$ and $s_{2}^{j} \in
[s_{2},(1 + \tau)s_{2}]$ such that (c) is satisfied for
$B(y_{1},s_{1}^{j})$ and $B(y_{2},s_{2}^{j})$. We choose
$s_{1}^{j},s_{2}^{j}$ accordingly, and then $E_{3} \leq 2/j \leq
\delta$ (for $j \geq 2/\delta$).

It remains to estimate $E_{1}$ and $E_{2}$; by symmetry, we may
concentrate on $E_{1}$. Recalling Lemma \ref{annulusControl}, we
first have
\begin{displaymath} E_{1} \lesssim \frac{\calH^{3}(A(y_{1},s_{1},s_{1}^{j}) \cap \Gamma_{j})}{\calL^{2}(\pi_{\W}(B(y_{1},s_{1}^{j}) \cap \Gamma_{j}))} \lesssim_{L} \frac{\calH^{3}(A(y_{1},s_{1},(1 + \tau)s_{1}) \cap \Gamma_{j})}{s_{1}^{3}} \end{displaymath}
since $s_{1}^{j} \leq (1 + \tau)s_{1}$, and
$\calL^{2}(\pi_{\W}(B(y_{1},s_{1}^{j}) \cap \Gamma_{j})) \gtrsim_{L}
s_{1}^{3}$, see Remark \ref{remark0}. Lemma \ref{annulusControl2}
then tells us that
\begin{displaymath} \limsup_{j \to \infty} \calH^{3}(A(y_{1},s_{1},(1 + \tau)s_{1}) \cap \Gamma_{j}) \lesssim_{L} \calH^{3}(A(y_{1},s_{1},(1 + \tau)s_{1}) \cap \Gamma), \end{displaymath}
and now we choose $\tau =
\tau(\delta,y_{1},y_{2},s_{1},s_{2},\Gamma,L)$ so small that the
right hand side above is smaller than $c_{L}\delta$ for some small
constant $c_{L}$ depending only on $L$, and the same holds with
"$1$" replaced by "$2$". If $c_{L}$ is small
enough, this implies that $E_{1} + E_{2} \leq \delta$.

Combining the estimates for $D_{i}$ and $E_{i}$ above, we conclude
that
\begin{displaymath} |\E_{\pi_{\W}(B(y_{1},s_{1})\cap\Gamma)} \nabla^{\phi} \phi - \E_{\pi_{\W}(B(y_{2},s_{2})\cap \Gamma)}\nabla^{\phi} \phi| \leq 3\delta. \end{displaymath}
Since $\delta > 0$ was arbitrary, this proves that
$\E_{\pi_{\W}(B(y,s) \cap \Gamma)} \nabla^{\phi} \phi$ is
independent of the ball $B(y,s) \subset B(0,1)$ with $y \in
\Gamma$ and satisfying $\calH^{3}(\partial B(y,s) \cap \Gamma) =
0$.
\end{proof}

\subsubsection{Conclusion}\label{ss:conclusion}
In this section, we conclude the proof of Proposition
\ref{dreamProp}. Under the standing counter assumption derived in
Lemma \ref{l:normalization}, we will infer from Lemma \ref{l:a.s.const}
that $\nabla^{\phi} \phi$ is a.e. constant on
$\pi_{\W}(B(0,1) \cap \Gamma)\supseteq \pi_{\W}(B(0,b_L))$. This will contradict the conclusion of Lemma \ref{eq:limit_beta}, which gave  $\beta_{\CG,\Gamma,L}(B(0,1)) \geq \varepsilon$.

\begin{lemma}\label{l:constancy}
Let $\phi$ be as in Lemma \ref{eq:limit_beta}, with graph $\Gamma$. Then $\phi \in \Adm_{\CG,L}(B(0,1))$. \end{lemma}

Recall that we have shown in Lemma \ref{l:a.s.const} that the
averages of $\nabla^{\phi}\phi$ are the same for a large class of
sets in $\W$. In order to conclude that then $\nabla^{\phi}\phi$
must be constant almost everywhere on $\pi_{\W}(B(0,1) \cap
\Gamma)$, and in particular on $\pi_{\W}(B(0,b_L))$, we apply a Lebesgue differentiation theorem for the
measure space $(\W,\mathcal{L}^2)$ endowed with the \emph{graph
distance} $d_{\Gamma}$. The latter can be defined for an arbitrary
intrinsic Lipschitz function $\phi:\mathbb{W} \to \mathbb{V}$ with
graph $\Gamma$ by setting
\begin{displaymath}
d_{\Gamma}(w,w') = d_{\mathbb{H}} ( w \cdot\phi(w),w'\cdot\phi(w')),\quad w,w'\in \mathbb{W}.
\end{displaymath}

\begin{proof}[Proof of Lemma \ref{l:constancy}]

The triple $(\W,d_{\Gamma},\calL^{2})$ is a metric space of
homogeneous type, so the usual Lebesgue differentiation theorem is
valid:
\begin{equation}\label{eq:leb_diff}
\lim_{r\to
0+}\frac{1}{\calL^{2}(B_{d_{\Gamma}}(w_0,r))}\int_{B_{d_{\Gamma}}(y,s)}f(w)
\, d\calL^{2} =f(w_0),\quad\mathcal{L}^2\text{ a.e. }w_0 \in
\mathbb{W},
\end{equation}
holds for all $f\in L^1_{\mathrm{loc}}(\W,\mathcal{L}^2)$. See for instance \cite{HKST}.

We apply this theorem for $f=\nabla^{\phi}\phi$. More precisely,
we choose for every Lebesgue point $w_0$ of $f = \nabla^{\phi}
\phi$ a sequence of radii $(s_{k})_{k \in \N}$ such that $s_{k}
\to 0$ and $\calH^{3}(\partial B(y_0,s_{k}) \cap \Gamma) = 0$ for
$\pi_{\W}(y_0) = w_0$. Then we apply \eqref{eq:leb_diff} to
conclude that
\begin{equation}\label{eq:exp_lim}
\lim_{k\to \infty}
\mathbb{E}_{B_{d_{\Gamma}}(w_0,s_k)}\nabla^{\phi}\phi=
\nabla^{\phi}\phi(w_0).
\end{equation}
Without loss of generality we may assume here that $w_0$ has been
chosen so that the pointwise intrinsic gradient
$\nabla^{\phi}\phi$ exists in $w_0$.

Finally, it follows immediately from the definition of the graph
distance that
\begin{equation}\label{eq:ball_proj}
B_{d_{\Gamma}}(\pi_{\mathbb{W}}(y), s) = \pi_{\W}(B(y,s)\cap
\Gamma)
\end{equation}
for all $y$ {on the graph of $\phi$} and $s>0$, and hence the
claim of the lemma follows from \eqref{eq:exp_lim}
and Lemma \ref{l:a.s.const}.
 \end{proof}

\begin{proof}[Proof of Proposition \ref{dreamProp}] According to Remark \ref{remark0}, it suffices prove \eqref{form29}.  Assume that this is not true. Then, by Lemma \ref{eq:limit_beta} and Lemma \ref{l:constancy}, there
exists an intrinsic Lipschitz function $\phi \colon \W \to \V$ with graph $\Gamma$ such that
\begin{displaymath} \phi \in \Adm_{\CG,L}(B(0,1)) \quad \text{and} \quad \beta_{\CG,\Gamma,L}(B(0,1)) \geq \varepsilon. \end{displaymath}
This contradiction proves Proposition \ref{dreamProp}.
\end{proof}

\subsection{Proof of the weak geometric lemma for constant gradient $\beta$-numbers}\label{ss:WGL_CG}
In this section, we apply Proposition \ref{dreamProp} to prove the
weak geometric lemma for the $\beta_{\CG}$-numbers, Theorem
\ref{t:CG-WGL}. In brief, if the integral appearing in Theorem
\ref{t:CG-WGL} is large, then we can find many balls
centered on the graph $\Gamma$ with large $\beta_{\CG}$-number
(this is quantified in Lemma \ref{discretisation}). Then,
Proposition \ref{dreamProp} implies that $\nabla^{\phi}\phi$
fluctuates strongly inside (the projections of) such balls, which
can be used to find a lower bound for
$\|\nabla^{\phi}\phi\|_{L^{2}(\pi_{\W}(B(x,R) \cap \Gamma)}^{2}$
with $x\in \Gamma$ and $R>0$. Since we also have the upper bound
$\lesssim LR^{3}$ for this quantity, we will eventually find the
correct upper bound for the integral in Theorem \ref{t:CG-WGL}.

\begin{definition} A collection $\mathcal{B}$ of balls in a metric space $(X,d)$ is \emph{pre-dyadic}, if
\begin{displaymath} \mathop{\sum_{B \in \mathcal{B}}}_{r/N \leq  \diam(B) < r} \chi_{B}(x) \lesssim_{N} 1, \quad x \in X, \: r > 0, \end{displaymath}
for any $N > 1$. A collection of balls (or sets in general) in
$(X,d)$ is \emph{dyadic}, if for every pair of sets $B,B'$ in the
family, either $B \cap B' = \emptyset$, or $B \subset B'$, or $B'
\subset B$.
\end{definition}

\begin{lemma}\label{discretisation} Let $\Gamma$ be an intrinsic Lipschitz graph, and let $x \in \Gamma$, $R > 0$. Write
\begin{displaymath} \int_{0}^{R} \int_{\Gamma \cap B(x,R)} \chi_{\{(y,s)\in \Gamma \times \mathbb{R}_+:\; \beta_{\CG}(B(y,s))\geq \varepsilon\}}(y,s) d\calH^{3}(y) \, \frac{ds}{s} =: C. \end{displaymath}
Then, there exists a pre-dyadic family of balls $B(x_{i},r_{i})
\subset B(x,2R)$ with $x_{i} \in \Gamma$ such that
$\beta_{\CG}(B(x_{i},r_{i})) \geq \varepsilon/2$ for every $i$,
every ball $B(x_{i},r_{i})$ has $A_{0}$-thin boundary
with respect to $\calH^{3}|_{\Gamma}$ for some $A_{0}
\geq 1$, and
\begin{displaymath} \sum_{i} \calH^{3}(B(x_{i},r_{i}) \cap \Gamma) \gtrsim C. \end{displaymath}
\end{lemma}

\begin{proof} Write $\Gamma_{R} := \Gamma \cap B(x,R)$. For $j \geq 1$, let
\begin{displaymath} E_{j} := \{y \in \Gamma_{R} : \beta_{\CG}(B(y,s)) \geq \varepsilon \text{ for some } s \in (2^{-j}R , 2^{-j + 1}R]\}. \end{displaymath}
Let $P_{j} \subset E_{j}$ be a maximal $2^{-j + 5}R$-separated
subset, and for each $y \in P_{j}$, choose a radius $s_{y} \in
(2^{-j}R , 2^{-j + 1}R]$ such that $\beta_{\CG}(B(y,s_{y})) \geq
\varepsilon$. Then, any pair of balls
$B(y_{1},2s_{y_{1}}),B(y_{2},2s_{y_{2}})$ with distinct
$y_{1},y_{2} \in P_{j}$ is disjoint. Further, since the sets
$B(y,2^{-j + 7}R) \cap \Gamma$ with $y \in P_{j}$ cover $E_{j}$,
we have
\begin{displaymath} \calH^{3}(E_{j}) \leq \sum_{y \in P_{j}} \calH^{3}(B(y,2^{-j + 7}R) \cap \Gamma) \sim \sum_{y \in P_{j}} (2^{-j}R)^{3} \sim \sum_{y \in P_{j}} \calH^{3}(B(y,s_{y}) \cap \Gamma). \end{displaymath}
Next, observe that for any $y \in \Gamma_{R}$, we have
\begin{displaymath} \int_{2^{-j}R}^{2^{-j + 1}R} \chi_{\{(y,s) \in \Gamma \times \R_{+} : \beta_{\CG}(y,s)) \geq \varepsilon\}}(y,s) \, \frac{ds}{s}
\lesssim \chi_{E_{j}}(y), \end{displaymath} since the left hand
side is always bounded by $\lesssim 1$, and it is evidently zero
for $y \notin E_{j}$. It follows that
\begin{align*} \sum_{j = 1}^{\infty} \calH^{3}(E_{j}) & = \sum_{j = 1}^{\infty} \int_{\Gamma_{R}} \chi_{E_{j}}(y) \, d\calH^{3}(y) \\
&\gtrsim \sum_{j = 1}^{\infty} \int_{\Gamma_{R}} \int_{2^{-j}R}^{2^{-j +1}R} \chi_{\{(y,s) \in \Gamma \times \R_{+} : \beta_{\CG}(B(y,s)) \geq \varepsilon\}}(y,s) \, d\calH^{3}(y) \, \frac{ds}{s}\\
& = \int_{0}^{R} \int_{\Gamma_{R}} \chi_{\{(y,s) \in \Gamma \times
\R_{+} : \beta_{\CG}(B(y,s)) \geq \varepsilon\}}(y,s) \,
d\calH^{3}(y) \, \frac{ds}{s} = C, \end{align*} and consequently
\begin{equation}\label{form17} \sum_{j = 1}^{\infty} \sum_{y \in P_{j}} \calH^{3}(B(y,s_{y}) \cap \Gamma) \gtrsim C. \end{equation}
Next, observe that if $\beta_{\CG}(B(y,s)) \geq \varepsilon$ and $s' \in [s,\frac{9}{8}s]$, then $\beta_{\CG}(B(y,s')) \geq \varepsilon/2$. Then,
for every ball $B(y,s_{y})$, apply Lemma \ref{l:A-small} with
$\delta = 1/8$ and $\mu = \calH^{3}|_{\Gamma}$ to
find a radius $s_{y}' \in [s_{y},\frac{9}{8}s_{y}]$ such that
$B(y,s_{y}')$ has $A_{1/8}$-thin boundary with respect to
$\calH^{3}|_{\Gamma}$, and $\beta_{\CG}(B(y,s_{y}')) \geq
\varepsilon/2$. Then the $3$-regularity of $\Gamma$ implies that \eqref{form17} holds for the balls
$B(y,s_{y}')$ as well. Finally, since the balls $B(y,s_{y}')$ with
$y$ in a fixed set $P_{j}$ are disjoint, the collection
\begin{displaymath} \{B(y,s_y'): y \in P_j, j \in \N\}  \end{displaymath}
is pre-dyadic. Indeed, if $x \in \He$, $N > 1$ and $r > 0$ are
given, there are clearly at most $\lesssim_{N}1$ indices $j$ such
that $x \in B(y,s_{y}')$ and $r/N \leq \diam(B(y,s_{y}')) < r$ for
some $y \in P_{j}$. The proof is complete. \end{proof}

The next lemma "refines" a pre-dyadic family of balls into a
dyadic one.

\begin{lemma}\label{predyadic} Let $(X,d,\mu)$ be a metric measure space, and let $\delta > 0$ and $A \geq 1$. Assume that $\calB$ is a pre-dyadic family of balls with $\sup\{\diam(B) : B \in \calB\} < \infty$, and with $\sum_{B \in \calB} \mu(B) = C$. Moreover, assume that every ball $B \in \calB$ contains a (possibly empty) family of disjoint sub-balls $\calF(B)$ such that that $\diam(B') \geq \delta \diam(B)$ for $B' \in \calF(B)$. Finally, assume that all the balls $B'$ in the families $\calB$ and $\calF(B)$, $B \in \calB$, have $A$-thin boundary, and the balls $B \in \calB$ have the doubling property
\begin{displaymath} \mu(5B) \leq A\mu(B). \end{displaymath}
Then, there exists a dyadic sub-collection $\{B_{j}\} \subset \calB$ such that 
\begin{enumerate}
\item[(i)] $\sum \mu(B_{j}) \gtrsim_{A,\delta} C,$ and
\item[(ii)] if $B_j \subset B_{j'}$ for $j \neq j'$ then $B_j \cap B'= \emptyset$ for all $B' \in \calF(B_{j'})$ or $B \subset B'$ for some $B' \in \calF(B_{j'})$.
\end{enumerate}
\end{lemma}

\begin{proof} Split the balls in $\calB$ into families $\calB_{j} := \{B \in \calB : N^{j - 1} \leq \diam(B) < N^{j}\}$, where $N = N_{\delta} \geq 2/\delta$. By the assumption of the family $\calB$ being pre-dyadic, and the doubling property of the balls $B \in \calB$, we can (by applying the $5r$-covering lemma to the balls in $\calB_{j}$) choose a subfamily of $\calB_{j}$ \textbf{disjoint} balls (still denoted by $\calB_{j}$) with
\begin{displaymath} \sum_{j} \sum_{B \in \calB_{j}} \mu(B) \gtrsim_{A,\delta} C. \end{displaymath}
Next, observe that either
\begin{displaymath} \sum_{j} \sum_{B \in \calB_{2j}} \mu(B) \gtrsim_{A,\delta} C \quad \text{or} \quad \sum_{j} \sum_{B \in \calB_{2j + 1}} \mu(B) \gtrsim_{A,\delta} C. \end{displaymath}
We assume that the former option holds, and from now on we will
only consider balls $B \in \calB_{2j}$, $j \in \Z$. Note that if
$j < i$, $B_{1} \in \calB_{2j}$, $B_{2} \in \calB_{2i}$ and
$B'_{2} \in \calF(B_{2})$, then
\begin{equation}\label{form12} \frac{\diam(B_{1})}{\diam(B_{2}')} \leq \frac{N^{2j}}{\delta N^{2i - 1}} = \frac{N^{2(j - i) + 1}}{\delta} \leq \frac{N^{j - i}}{\delta} \leq 2^{j - i}, \end{equation}
recalling that $N \geq 2/\delta$.

Since $\sup\{\diam(B) : B \in \calB\} < \infty$, the collection
$\calB_{2j}$ is empty for large enough $j$. Let $j_{0}$ be the
index of the largest non-empty collection $\calB_{2j_{0}}$, and
set
\begin{displaymath} \calG_{j_{0}} := \calB_{2j_{0}} \cup \bigcup_{B \in \calB_{2j_{0}}} \calF(B). \end{displaymath}
Inductively, for $j < i \leq j_{0}$, let $\calR_{j}^{i}$ be the
collection of balls in $\calB_{2j}$, which meet the boundary of
one of the balls in $\calG_{i}$, let
\begin{displaymath} \calG_{j}^{0} := \calB_{2j} \setminus \bigcup_{j < i \leq j_{0}} \calR_{j}^{i}, \end{displaymath}
and finally
\begin{displaymath} \calG_{j} := \calG_{j}^{0} \cup \bigcup_{B \in \calG_{j}^{0}} \calF(B). \end{displaymath}
Note that if $B \in \calR_{j}^{i}$, $j < i \leq j_{0}$, then there
is some ball $B' = B(y,r) \in \calG_{i}$ such that
\begin{align*} B \subset \{x : \dist(x,\partial B') \leq \diam(B)\} & \subset \{x : \dist(x,\partial B') \leq 2^{j - i} \diam(B')\}\\
& \subset B(y,2r) \cap A(y,(1 - 2^{j - i - 1})r,(1 + 2^{j - i +
1})r).  \end{align*} (The second inclusion follows from
\eqref{form12}).  Consequently, by the disjointness of the balls
in $\calB_{2j} \supset \calR_{j}^{i}$, and the $A$-thinness of
boundaries of the balls in $\calB \supset \calG_{i}$, we infer
that
\begin{align*} \sum_{B \in \calR_{j}^{i}} \mu(B) & = \mu\left(\bigcup_{B \in \calR_{j}^{i}} B \right)\\
& \leq \mu\left(\bigcup_{B(y,r) \in \calG_{i}} B(y,2r) \cap A(y,(1 - 2^{j - i - 1})r,(1 + 2^{j - i + 1})r) \right)\\
& \leq \sum_{B(y,r) \in \calG_{i}} \mu(B(y,2r) \cap A(y,(1 - 2^{j
- i - 1})r,(1 + 2^{j - i + 1})r))\\& \lesssim A2^{j - i} \sum_{B'
\in \calG_{i}} \mu(B'). \end{align*} Finally, since every ball in
$\calB_{2j}$ either belongs to one of the collections
$\calR_{j}^{i}$, with $j < i \leq j_{0}$, or to $\calG_{j}^{0}
\subset \calG_{j}$, we have
\begin{align} C \lesssim_{A,\delta} \sum_{j \leq j_{0}} \sum_{B \in \calB_{2j}} \mu(B) & \leq \sum_{j \leq j_{0}} \left[ \sum_{j < i \leq j_{0}} \sum_{B \in \calR_{j}^{i}} \mu(B) + \sum_{B \in \calG_{j}} \mu(B) \right] \notag\\
& \lesssim A\sum_{j \leq j_{0}} \left[ \sum_{j < i \leq j_{0}} 2^{j - i} \sum_{B' \in \calG_{i}} \mu(B') + \sum_{B \in \calG_{j}} \mu(B) \right] \notag\\
&\label{form13} \sim A\sum_{j \leq j_{0}} \sum_{B \in \calG_{j}}
\mu(B). \end{align} Define
\begin{displaymath} \calG := \bigcup_{j \leq j_{0}} \calG_{j}. \end{displaymath}
Inspecting the definition of $\calG_{j}$, it is clear that $\calG$
can indeed be written as
\begin{displaymath} \calG = \{B_{j}\} \cup \bigcup_{j} \calF(B_{j}) \end{displaymath}
for a certain sub-collection $\{B_{j}\} \subset \calB$. Since
$\sum_{B' \in \calF(B)} \mu(B') \leq \mu(B)$ for all balls $B \in
\calB$, we infer from \eqref{form13} that $\sum \mu(B_{j})
\gtrsim_{A,\delta} C$, as required by condition (i). The fact that $\{B_{j}\}$ is a
dyadic family of balls follows immediately from the inductive
definition. In fact, the construction shows that even the family $\calG$ is dyadic, and this combined with \eqref{form12} implies condition (ii). The proof is complete. \end{proof}

\subsubsection{Proof of the Theorem \ref{t:CG-WGL}} During the proof, we will abbreviate $\lesssim_{\varepsilon,L}$ and $\gtrsim_{\varepsilon,L}$ to simply $\lesssim$ and $\gtrsim$. We apply Lemma \ref{discretisation} to infer that if
\begin{equation}\label{definitionOfC} \int_{0}^{R} \int_{\Gamma \cap B(x,R)} \chi_{\{(y,s)\in \Gamma \times \mathbb{R}_+:\; \beta_{\CG}(B(y,s))\geq \varepsilon\}}(y,s) d\calH^{3}(y) \, \frac{ds}{s} =: C, \end{equation}
then there exists a pre-dyadic family of balls $\{B_{j}\}$,
contained in $B(x,2R)$ and centred on $\Gamma$, with $A_0$-thin boundaries, with $\beta_{\CG}(B_{j}) \geq
\varepsilon/2$, and such that
\begin{equation}\label{form16} \sum_{j} \diam(B_{j})^{3} \gtrsim \sum_{j} \calH^{3}(B_{j} \cap \Gamma) \gtrsim C. \end{equation}
For each such ball $B_{j}$, we use Proposition \ref{dreamProp} to
find another ball $\hat{B}_{j} \subset B_{j}$ with $A_\varepsilon$-thin
boundary, such that $\diam(\hat{B}_{j}) \geq
\delta_{\varepsilon,L} \diam(B)$,
\begin{equation}\label{form27} \calL^{2}(\pi_{\W}(\hat{B}_{j} \cap \Gamma)) \sim \diam(B_{j})^{3}, \end{equation}
and
\begin{equation}\label{form15} |\E_{\pi_{\W}(\hat{B}_{j} \cap \Gamma)}\nabla^{\phi}\phi - \E_{\pi_{\W}(B_{j} \cap \Gamma)}\nabla^{\phi}\phi| \geq \delta_{\varepsilon,L}. \end{equation}

Note that every ball in the family $\{B_j, \hat{B}_j\}$ has $\max\{A_0,A_\varepsilon\}$-thin boundary. Hence, using Lemma \ref{predyadic} on the metric space
$(\He,d_{\He},\calH^{3}|_{\Gamma})$, with $\calB := \{B_{j}\}$ and
$\calF(B_{j}) := \{\hat{B}_{j}\}$, we find a dyadic sub-collection
$\{B_{j_{i}}, \hat{B}_{j_{i}}\}$, which still satisfies $\sum_{i}
\calH^{3}(B_{j_{i}} \cap \Gamma) \gtrsim C$. To avoid the double
indices, we assume that the family $\{B_{j}, \hat{B}_{j}\}$ is
dyadic itself. Since $\pi_{\W}$ is injective on $\Gamma$, it
follows that the family $\{A^{j},B^{j}\}$ with $A^{j} := \pi_{\W}(\hat{B}_{j} \cap \Gamma)$ and $B^{j} := \pi_{\W}(B_{j} \cap \Gamma)$ is also dyadic. Moreover, by (ii) of Lemma \ref{predyadic}, the family $\{A^{j},B^{j}\}$ even satisfies slightly more, namely if $j \neq j'$, then
\begin{equation}\label{betterThanDyadic} B^{j} \subset B^{j'} \quad \Longrightarrow \quad B^{j} \subset A^{j'} \text{ or } B^{j} \cap A^{j'} = \emptyset. \end{equation}

Holding that thought, we now make some remarks of more abstract
nature. Given two sets $A,B \subset \W$ with $A \subset B$, let
$V_{A,B}$ be the subspace of $L^{2}(\W)$ consisting of those
functions which are zero outside $B$, are constant on both $B
\setminus A$ and $A$, and have integral zero. Let $\pi_{A,B} :=
\pi_{V_{A,B}}$ be the orthogonal projection onto this subspace.
Finding an explicit formula for $\pi_{A,B}$ is simple, as
$V_{A,B}$ is one-dimensional, and spanned by the unit vector
\begin{displaymath} e_{A,B} := \left(\frac{|B \setminus A|}{|A||B|} \right)^{1/2} \chi_{A} - \left(\frac{|A|}{|B||B \setminus A|}\right)^{1/2} \chi_{B \setminus A}. \end{displaymath}
Here, and for the rest of the proof, we write $|U| :=
\calL^{2}(U)$ for $U \subset \W$. It follows that $\pi_{A,B}(f) =
(f \cdot e_{A,B})e_{A,B}$, and in particular
\begin{align} \|\pi_{A,B}(f)\|_{2}^{2} = |f \cdot e_{A,B}|^{2} & = \frac{|B \setminus A|}{|A||B|} \left( \int_{A} f\, d\mathcal{L}^2- \frac{|A|}{|B \setminus A|} \int_{B \setminus A} f d\mathcal{L}^2\right)^{2}\notag\\
& = \frac{|B \setminus A|}{|A||B|} \left(\frac{|B|}{|B \setminus A|} \int_{A} f\,d\mathcal{L}^2 - \frac{|A|}{|B \setminus A|} \int_{B} f\,d\mathcal{L}^2 \right)^{2}\notag\\
&\label{form14} = \frac{|A||B|}{|B \setminus A|}
\left(\frac{1}{|A|} \int_{A} f\, d\mathcal{L}^2 - \frac{1}{|B|}
\int_{B} f \,d\mathcal{L}^2\right)^{2} \end{align} 

We apply these observations to the subspaces $V_{A^{j},B^{j}}$ with the family $\{A^{j},B^{j}\}$ defined above. Since this family is dyadic and satisfies \eqref{betterThanDyadic}, it is easy to check that the subspaces $V_{A^{j},B^{j}}$ are pairwise orthogonal. Now, let 
\begin{displaymath} f := \nabla^{\phi}\phi \chi_{\pi_{\W}(B(x,2R) \cap \Gamma)} \end{displaymath}
and recall that $\|f\|_{L^{\infty}} \lesssim_{L} 1$. Since  $B^j \subset \pi_{\W}(B(x,2R)
\cap \Gamma)$ for all $j \in \N$, it follows from Bessel's inequality, \eqref{form14}, \eqref{form15}, \eqref{form27} and \eqref{form16} (in this order) that
\begin{align*} R^{3} \gtrsim \|f\|_{2}^{2} \geq \sum_{j} \|\pi_{A^{j},B^{j}}(f)\|_{2}^{2} & = \sum_{j} \frac{|A^{j}||B^{j}|}{|B^{j} \setminus A^{j}|}\left[ |\E_{A^{j}}(f) - \E_{B^{j}}(f)|^{2} \right] \\
& \gtrsim \delta_{\varepsilon,L}^{2} \sum_{j} |A^{j}| \sim
\delta_{\varepsilon,L}^{2} \sum_{j} \diam(B_{j})^{3} \gtrsim
\delta_{\varepsilon,L}^{2} C. \end{align*} Recalling the
definition of $C$ from \eqref{definitionOfC}, this completes the proof of Theorem
\ref{t:CG-WGL}.

\subsection{Proof of the weak geometric lemma for intrinsic Lipschitz graphs}
Armed with Theorem \ref{t:CG-WGL}, we are nearly ready to prove
the WGL for intrinsic Lipschitz graphs (Theorem \ref{wglintlg}).
First, we need a few short lemmas.
 The first one states that any intrinsic Lipschitz graph $\Gamma$ has a "big flat piece" inside any ball $B(x,r)$ with $x \in \Gamma$:

\begin{lemma}\label{bigFlatPiece}
For every $\varepsilon > 0$ and $L \geq 1$, there exists $\delta =
\delta_{\varepsilon,L} > 0$ with the following property. If
$\Gamma$ is an intrinsic $L$-Lipschitz graph, $x \in \Gamma$ and
$r > 0$, then there exists a ball $B(y,s) \subset B(x,r)$ with $y
\in \Gamma$ and $s \geq \delta r$ such that
$\beta_{\Gamma}(B(y,s)) \leq \varepsilon$.
\end{lemma}

\begin{proof} The proof is based on the compactness of intrinsic $L$-Lipschitz graphs,
and the existence of tangent subgroups. We make a counter
assumption that there exist $\varepsilon > 0$ and $L \geq 1$ with
the following properties. For every $j \in \N$, there exists an
intrinsic Lipschitz graph $\Gamma_{j} = \{w \cdot \phi_{j}(w) : w
\in \W\}$, and a ball $B(x_{j},r_{j})$ with $x_{j} \in \Gamma_{j}$
such that
\begin{displaymath} \beta_{\Gamma_{j}}(B(y,sr_j)) \geq \varepsilon \end{displaymath}
for all balls $B(y,sr_j) \subset B(x_{j},r_{j})$ with $y \in
\Gamma$ and $s \geq 2^{-j}$. Without loss of generality (or
recalling a similar reduction from the previous section), we may
assume that $x_{j} = \phi_{j}(0) = 0$, and $r_{j} = 1$ for all $j
\in \N$. Further, by the compactness results for intrinsic
$L$-Lipschitz functions, already employed in the previous section,
we may assume without loss of generality that the functions
$\phi_{j}$ converge locally uniformly to an intrinsic
$L$-Lipschitz function $\phi \colon \W \to \V$ with graph
$\Gamma$. One can then easily check (emulating the argument for
Lemma \ref{eq:limit_beta}) that
\begin{displaymath} \beta_{\Gamma}(B(y,s)) \geq \varepsilon \end{displaymath}
for any ball $B(y,s) \subset B(0,1)$ with $y \in \Gamma$.
But this is absurd: the function $\phi$ is intrinsically
differentiable at almost every point of $w \in \W$, which means
precisely that $\calH^{3}$-almost every point $y \in \Gamma$ has a vertical
tangent subgroup to $\Gamma$ (Theorem 4.15 in \cite{FSS11}). We record that a subgroup $\mathbb{G} \subset \He$ is a tangent group to $\Gamma$ in $y \in \G$ if
$$\lim_{h \ra \infty} \delta_{h}(y^{-1}\cdot \Gamma)=\mathbb{G}$$
with respect to the Hausdorff convergence of compact subsets of $\He$.   But this is
evidently impossible for such $y \in \Gamma$ that
$\beta_{\Gamma}(B(y,s)) \geq \varepsilon$ for all sufficiently
small $s > 0$. All the points $y \in \Gamma \cap
\operatorname{int} B(0,1)$ have this property by the previous
discussion, and we have reached a contradiction.
\end{proof}

Recall that the graph of an entire intrinsic Lipschitz function
$\phi:\W \to \V$ with constant gradient almost everywhere is a
left translate of a vertical plane (Proposition
\ref{constantGradientGraphs}). The next lemma, essentially a
corollary of this fact, gives a substantial improvement over the
previous lemma for constant gradient graphs.

\begin{lemma}\label{stability0} For every $\varepsilon > 0$ and $L \geq 1$, there exists $\delta = \delta_{\varepsilon,L} > 0$ with the following property.
Assume that $\Gamma = \{w \cdot \phi(w) : w \in \W\}$ is an
intrinsic $L$-Lipschitz graph, and $x \in \Gamma$, $r > 0$ are
such that $\phi \in \Adm_{\CG,L}(B(x,r))$. Then, if $y \in
\Gamma$, $0 < s < r$, and $B(y,s/\delta) \subset B(x,b_{L}r)$, we
have
\begin{displaymath} \beta(B(y,s)) \leq \varepsilon. \end{displaymath}
\end{lemma}

\begin{proof} We make a counter assumption: there exist $\varepsilon > 0$ and $L \geq 1$ such that the following holds for every $\delta_{j} = 1/j$, $j \in \N$. There exists an intrinsic $L$-Lipschitz graph $\Gamma_{j} = \{w \cdot \phi_{j}(w) : w \in \W\}$ and two balls $B(y_{j},js_{j}) \subset B(x_{j},b_{L}r_{j})$ with $x_{j},y_{j} \in \Gamma_{j}$, such that
\begin{displaymath} \phi_{j} \in \Adm_{\CG,L}(B(x_{j},r_{j})) \quad \text{and} \quad \beta_{\Gamma_{j}}(B(y_{j},s_{j})) \geq \varepsilon. \end{displaymath}
We may assume that $B(y_{j},s_{j}) = B(0,1)$. Thus, $0 =
\phi_{j}(0) \in \Gamma_{j}$ and $\beta_{\Gamma_{j}}(B(0,1)) \geq
\varepsilon$, and $\phi_{j} \in
\Adm_{\CG,L}(B(z_{j},r_{j}/s_{j}))$ for a certain point $z_{j} \in
\Gamma_{j}$. The assumption $B(y_{j},j s_{j}) \subset
B(x_{j},b_L\,r_{j})$ implies that the balls
$B(z_{j},b_{L}[r_{j}/s_{j}])$ eventually cover any ball $B(0,R)$,
$R > 0$. We infer that a subsequence of the functions $\phi_{j}$
converge locally uniformly to an intrinsic $L$-Lipschitz function $\phi_{R}
\colon \W \to \V$, which has a.e. constant gradient on
$\pi_{\W}(B(0,R))$, and such that $\beta_{\Gamma_{R}}(B(0,1)) \geq
\varepsilon$ (here $\Gamma_{R}$ is the graph of $\phi_{R}$).
Finally, using a diagonal procedure, a further subsequence of the
functions $\phi_{j}$ converges locally uniformly to an intrinsic
$L$-Lipschitz function $\phi$, which has a.e. constant gradient on
$\W$, and $\beta_{\Gamma}(B(0,1)) \geq \varepsilon$. But the
existence of such a $\phi$ was ruled out in Proposition
\ref{constantGradientGraphs}, and we have reached a contradiction.
The proof of the lemma is complete. \end{proof}

Let $\Gamma$ be an intrinsic Lipschitz graph, and let $\Delta$ be
a system of David cubes on $\Gamma$; recall that every cube $Q \in
\Delta$ contains a set of the form $B(z_{Q},c\ell(Q)) \cap \Gamma$
for some point $z_{Q} \in Q$ and some constant $c > 0$. In this
section, where our task is to prove the weak geometric lemma for
intrinsic $L$-Lipschitz graphs, we write
\begin{displaymath} B_{Q} := B(z_{Q},C\ell(Q)), \end{displaymath}
where $C = C_{L} \geq 2$ is a constant depending on $L$ only. The
precise requirement for $C$ is that $b_{L}C \geq 8$, where $b_{L}
> 0$ is the constant from the definition of
$\beta_{\CG,\Gamma,L}$. With this notation for $B_{Q}$, we still
write $\beta(Q) := \beta(B_{Q})$, and $\beta_{\CG,\Gamma,L}(Q) =
\beta_{\CG,\Gamma,L}(B_{Q})$. The reason for choosing $C_{L}$ so
large is the following: if $Q,R \in \Delta$ with $Q \subset R$,
and $\delta > 0$ is any constant, then
\begin{equation}\label{form31a} B(z_{Q},C\ell(Q)/\delta) \subset (b_{L}/4)B_{R} := B(z_{R},(b_{L}/4)C\ell(R)) \end{equation}
as soon as $\ell(Q)/\ell(R) \leq b_{L}\delta/8$. This is easy to
check: if $\ell(Q)/\ell(R) \leq b_{L}\delta/8$ and $y \in
B(z_{Q},C\ell(Q)/\delta)$, then
\begin{displaymath} d_{\He}(y,z_{R}) \leq d_{\He}(y,z_{Q}) + d_{\He}(z_{Q},z_{R}) \leq (C \ell(Q)/(\delta \ell(R)) + 1)\ell(R) \leq b_{L}C\ell(R)/4, \end{displaymath}
We also mention that $C \geq 2$ implies
\begin{displaymath} Q \subset R \quad \Longrightarrow \quad B_{Q} \subset B_{R}, \end{displaymath}
which is verified in the similar manner as above. In this section,
we abbreviate $\calH^{3}(A \cap \Gamma) =: |A|$. We choose to
prove the weak geometric lemma in the following form: for any
fixed cube $Q_{0} \in \Delta$,
\begin{equation}\label{form23} \sum_{Q \in \calB_{\varepsilon}(Q_{0})} |Q| \lesssim_{\varepsilon,L} |Q_{0}|, \end{equation}
where $\calB_{\varepsilon}(Q_{0}) := \{Q \in \Delta(Q_{0}) :
\beta(Q) > \varepsilon\}$; the integral formulation is an easy
corollary. To this end, we note that the following inequality is a
simple consequence (essentially a reformulation) of Theorem
\ref{t:CG-WGL}:
\begin{equation}\label{form21} \sum_{Q \in \calB_{\CG,\varepsilon}(Q_{0})} |Q| \lesssim_{\varepsilon,L} |Q_{0}|, \end{equation}
where $\calB_{\CG,\varepsilon}(Q_{0}) := \{Q \in \Delta(Q_{0}) :
\beta_{\CG}(Q) > \varepsilon\}$.

Let $\eta > 0$ be a constant, which will depend only on
$\varepsilon$ and $L$. Now, by \eqref{form21}, the weak geometric
lemma \eqref{form23} will follow, if we are able to prove that
\begin{equation}\label{form22} \sum_{Q \in \calB_{\varepsilon}(Q_{0}) \setminus \calB_{\CG,\eta}(Q_{0})} |Q| \lesssim_{\varepsilon,L} |Q_{0}|. \end{equation}
The following corollary of Lemma \ref{stability0} will be crucial:
\begin{cor}\label{stability2} For every $\varepsilon > 0$ and $L \geq 1$, there exists $\eta = \eta_{\varepsilon,L}$ with the following property. Assume that $Q,R \in \Delta$ are cubes such that $Q \subset R$, $\beta(R) \leq \eta$, and $\beta_{\CG}(P) \leq \eta$ for all $P \in \Delta$ with $Q \subset P \subset R$. Then $\beta(Q) \leq \varepsilon$. \end{cor}
\begin{proof} Write $\delta := \delta_{\varepsilon/8,L}$, where $\delta_{\varepsilon/8,L} > 0$ is the constant from Lemma \ref{stability0}.
 Write $Q^{(n)}$ for the $n^{th}$ dyadic ancestor of $Q$, with $Q^{(0)} := Q$. Let $n_{\varepsilon,L} \in \N$ be a large number depending
 on $\varepsilon$ and $L$; we will eventually pick $n_{\varepsilon,L}$ first, and then require that $\eta$ is sufficiently small depending
 on $n_{\varepsilon,L}$. If $R = Q^{(n)}$ for some $n < n_{\varepsilon,L}$, we have
\begin{displaymath} \beta(Q) \leq \beta(R) \cdot \frac{\ell(R)}{\ell(Q)} = 2^{n} \beta(R) \leq 2^{n_{\varepsilon,L}}\eta \leq \varepsilon, \end{displaymath}
assuming $\eta \leq 2^{-n_{\varepsilon,L}}\varepsilon$. This is
the first restriction we place on $\eta$.

Now, assume that $R = Q^{(n)}$ for some $n \geq
n_{\varepsilon,L}$. It follows that $P := Q^{(n_{\varepsilon,L})}$
satisfies $Q \subset P \subset R$, hence $\beta_{\CG}(P) \leq
\eta$. Recalling Remark \ref{remk:bcg} this means that there exists an intrinsic $L$-Lipschitz
function $\psi \colon \W \to \V$ with graph $\Gamma_{P}$ such that
$\psi$ has a.e.\ constant gradient on $\pi_{\W}(b_{L}B_{P}) =
\pi_{\W}(B(z_{P},b_{L}C\ell(P)))$, and
\begin{equation}
\label{bcggp}\sup_{y \in B_{P} \cap \Gamma} \frac{\dist_{\He}(y,\Gamma_{P})}{C\ell(P)} \leq 2\eta. \end{equation}
Let $x_{P},x_{Q}$ be the nearest points to $z_{P},z_{Q} \in B_{P}
\cap \Gamma$ on $\Gamma_{P}$, so that
\begin{displaymath}
d_{\He}(x_{P},z_{P}) \leq 2C\eta\ell(P)\quad \text{and}\quad
d_{\He}(x_{Q},z_{Q}) \leq 2C\eta \ell(P).
\end{displaymath}
 Write $r := C\ell(P)/2$,
and observe that
\begin{equation}\label{form30} (b_{L}/4)B_{P} \subset B(x_{P},b_{L}r) \subset b_{L}B_{P} \end{equation}
for $\eta$ small enough. Then, write $s := 2C\ell(Q)$, and observe
that \eqref{form31a}  gives the following chain of inclusions,
assuming $2^{n_{\varepsilon,L}} = \ell(P)/\ell(Q)$ to be large enough,
and $\eta_{\varepsilon,L} > 0$ to be small enough:
\begin{equation}\label{form31} B(x_{Q},s/\delta) \subset B(z_{Q},2s/\delta) \subset (b_{L}/4)B_{P} \subset B(x_{P},b_{L}r). \end{equation}
From \eqref{form30} we first conclude that $\psi$ has a.e.\
constant gradient on $\pi_{\W}(B(x_{P},b_{L}r))$, which means that
\begin{displaymath} \psi \in \Adm_{\CG,L}(B(x_{P},r)). \end{displaymath}
Then, from \eqref{form31} we infer that Lemma \ref{stability0} is
applicable to the ball $B(x_{Q},s)$. Consequently,
\begin{displaymath} \beta_{\Gamma_{P}}(B(x_{Q},s)) \leq \varepsilon/8. \end{displaymath}
Thus, there exists a set of the form $z \cdot \W'$, where $z \in
\He$ and $\W'$ is a vertical subgroup, such that $\dist_{\He}(y,z
\cdot \W') \leq C(\varepsilon/2)\ell(Q)$ for every $y \in
B(x_{Q},s) \cap \Gamma_{P}$. Finally, recalling that we wish to
prove $\beta_{\Gamma}(Q) \leq \varepsilon$, we fix
\begin{displaymath} x \in B_{Q} \cap \Gamma \subset B_{P} \cap \Gamma. \end{displaymath}
By \eqref{bcggp} we can pick $y \in \Gamma_{P}$ with $d_{\He}(x,y) \leq 2C\eta
\ell(P)$. Then, for $\eta > 0$ small enough, we have $y \in
B(x_{Q},s) \cap \Gamma_{P}$, and $\dist_{\He}(x,z \cdot \W') \leq
2C\eta \ell(P) + C(\varepsilon/2)\ell(Q) \leq \varepsilon
C\ell(Q)$. This proves that $\beta_{\Gamma}(Q) \leq \varepsilon$,
as required. \end{proof}

We also record a version of the "big flat piece lemma", Lemma
\ref{bigFlatPiece}, for cubes:
\begin{lemma}\label{bigFlatCube} For every $\eta > 0$ and $L \geq 1$, there exists $\delta = \delta_{\eta,L} > 0$ with the following property. For any cube $R \in \Delta$, there exists a cube $Q \in \Delta(R)$ such that $\ell(Q) \geq \delta \ell(R)$ and $\beta(Q) \leq \eta$.  \end{lemma}

\begin{proof} Let $c:=c_{L} > 0$ be a constant depending only on $L$, and let $\varepsilon := c_{L}\eta$. By Lemma \ref{bigFlatPiece} applied to the ball $B(z_{R},c\ell(R))$ with $B(z_{R},c\ell(Q)) \cap \Gamma \subset R$, there exists a point $y \in B(z_{R},c\ell(R)) \cap \Gamma$ and a radius $s \gtrsim_{\varepsilon,L} \ell(R)$ such that $B(y,s) \subset B(z_{R},c\ell(R))$ and $\beta_{\Gamma}(B(y,s)) \leq \varepsilon$. Now, it suffices to note that $B(y,s)$ contains a ball of the form $B_{Q}$ for some $Q \subset R$ with $\ell(Q) \sim_{L} s$. Hence $\beta(Q) \lesssim_{L} \varepsilon = c_{L}\eta$. Choosing $c_{L} > 0$ small enough, the lemma follows.  \end{proof}

We are prepared to prove the weak geometric lemma. As discussed,
it remains to verify \eqref{form22}. The proof is nearly verbatim
the same as on the last pages of \cite{To2}, but we record the
details for completeness.

\subsubsection{Proof of the weak geometric lemma for intrinsic Lipschitz graphs}

We write $\calG_{\eta} := \calG_{\CG,\eta}(Q_{0}) := \Delta(Q_{0}) \setminus \calB_{\CG,\eta}(Q_{0})$, where $\eta = \eta_{\varepsilon,L}$ is the constant from Corollary \ref{stability2}. Then, we partition the collection $\calG_{\eta}$ into \emph{trees}. A family $\calT \subset \Delta$ is called a tree, if
\begin{itemize}
\item $\calT$ has a maximal element $Q(\calT)$, called the
\emph{root} of $\calT$. \item If $Q,Q' \in \calT$ with $Q \subset
Q'$, and $Q \subset P \subset Q'$, then also $P \in \calT$. \item
If $Q \in \calT$, then either all, or none, of the children of $Q$
belong to $\calT$. \item Those cubes $Q \in \calT$ with no
children in $\calT$ are called the \emph{stopping cubes} of
$\calT$, and they are denoted by $Stop(\calT)$.
\end{itemize}
We now partition $\calG_{\eta}$ into trees by an inductive
procedure. The roots of the initial trees are simply the maximal
cubes in $\calG_{\eta}$. Then, if $\calT_{1}$ is one of these
trees, and $Q \in \calT_{1}$, we add all the children of $Q$ into
$\calT_{1}$ if all of them belong to $\calG_{\eta}$; if not, we
declare that $Q \in Stop(\calT)$ and stop building $\calT$ along
this branch.

After the initial trees have been constructed, and in case some
cubes in $\calG_{\eta}$ still remain outside them, we repeat the
previous procedure: we pick the maximal cubes in $\calG_{\eta}$,
which are not, yet, contained in a tree, and we declare these to
be the roots of new trees. These trees are, then, constructed by
the rule described above.

Iterating this algorithm produces a partition of $\calG_{\eta}$
into a countable number of trees $\{\calT_{1},\calT_{2},\ldots\}$.
If $Q(\calT_{i}) = Q_{0}$ for some (unique) $i = i_{0}$, we set
$Q(\calT_{i})' := Q(\calT_{i})$. In the opposite case,
$Q(\calT_{i}) \subsetneq Q_{0}$, it follows from the construction
that either the parent or one of the siblings of $Q(\calT_{i})$,
say $Q'$ is in $\Delta(Q_{0}) \setminus \calG_{\eta} =
\calB_{\CG,\eta}(Q_{0})$. In this case, we set $Q(\calT_{i})' :=
Q'$.

With this notation in hand, and observing that any cube $Q'$ can
 only serve as $Q(\calT_{i})'$ for boundedly many indices $i$,
we are prepared to show that the roots $Q(\calT_{i})$ satisfy a
Carleson property:
\begin{equation}\label{form24} \sum_{i} |Q(\calT_{i})| \lesssim \sum_{i} |Q(\calT_{i})'| \lesssim |Q_{0}| + \sum_{Q \in \calB_{\CG,\eta}(Q_{0})} |Q| \lesssim_{\varepsilon} |Q_{0}|. \end{equation}
The last inequality is, of course, \eqref{form21}.

Next, recalling our objective \eqref{form22}, and writing
$\calB_{\varepsilon} := \calB_{\varepsilon}(Q_{0})$, we make the
following natural splitting:
\begin{displaymath} \sum_{Q \in \calB_{\varepsilon} \cap \calG_{\eta}} |Q| = \sum_{i} \sum_{Q \in \calB_{\varepsilon} \cap \calT_{i}} |Q|.\end{displaymath}
By \eqref{form24}, it remains to prove that
\begin{equation}\label{form25} \sum_{Q \in \calB_{\varepsilon} \cap \calT_{i}} |Q| \lesssim_{\varepsilon,L} |Q(\calT_{i})| \end{equation}
for any fixed $i$. To this end, let $\calF_{i}$ be the family of
maximal cubes $Q \in \calT_{i}$ with the property that $\beta(Q)
\leq \eta$. By Corollary \ref{stability2}, if $Q' \in \calT_{i}$
is contained in a cube in $\calF_{i}$, then $\beta(Q') \leq
\varepsilon$, and hence $Q' \notin \calB_{\varepsilon}$. So, in
fact the summation in \eqref{form25} runs over at most those cubes
which are not contained in any cube in $\calF_{i}$; let those
cubes be called $\calH_{i}$.

Fix $R \in \calH_{i}$, and let $Q \in \Delta(R)$ be the maximal
cube with the property that $\beta(Q) \leq \eta$. Lemma
\ref{bigFlatCube} promises that there exists such a cube $Q$ with
$\ell(Q) \geq \delta_{\eta,L} \ell(R)$. If $Q \in \calT_{i}$, then
obviously $Q \in \calF_{i}$, and we set $f(R) := Q$. If, on the
other hand, $Q$ already lies outside $\calT_{i}$, then there is
some cube $Q' \in Stop(\calT_{i})$ with $Q \subset Q' \subset R$,
and we set $f(R) := Q'$. In either case $|f(R)| \gtrsim_{\eta,L}
|R|$. Taking into account that both families $\calF_{i}$ and
$Stop(\calT_{i})$ consist of disjoint cubes, and that each cube $Q
\in \calF_{i} \cup Stop(\calT_{i})$ can only be assigned as $f(R)$
for boundedly many cubes $R$ (with bounds depending only on
$\eta,L$), we may conclude that
\begin{displaymath} \sum_{R \in \calH_{i}} |R| \lesssim_{\eta,L} \sum_{R \in \calH_{i}} |f(R)| \lesssim_{\eta,L} \sum_{Q \in \calF_{i}} |Q| + \sum_{Q \in Stop(\calT_{i})} |Q| \lesssim |Q(\calT_{i})|. \end{displaymath}
Since $\eta = \eta_{\varepsilon,L}$ only depends on $\varepsilon$
and $L$, we have proven \eqref{form25}, and the weak geometric
lemma \eqref{form23} (Theorem \ref{wglintlg}) for intrinsic Lipschitz graphs.

\subsubsection{BPiLG implies WGL} We will  conclude the paper with a brief discussion of Theorem \ref{wglbp}. The analogous result in the Euclidean setting is due to David and Semmes and it follows by Theorem 1.8 \cite[Part IV, Chapter 1]{DS2}. Before restating it in our setting, we record a definition.
\begin{definition}
Let $C_0>0$ and $\gamma:(0,1] \ra (0,\infty)$. We define $\text{WGL}(C_0, \gamma)$ to be the collection of $3$-regular sets $E \subset \He$ with regularity constant at most $C_0$  such that
\begin{equation*}
\label{wgleqbp}
\int_{0}^{R} \int_{E \cap B(x,R)} \chi_{\{(y,s)\in E \times
\mathbb{R}_+:\; \beta(B(y,s))\geq \varepsilon\}}(y,s)\,
d\calH^{3}(y) \, \frac{ds}{s} \leq \gamma(\varepsilon) \, R^{3}
\end{equation*}
for $\varepsilon > 0$, $x \in E$ and $R > 0$.
\end{definition}
We can now provide the restatement of Theorem 1.8 \cite[Part IV, Chapter 1]{DS2} in the Heisenberg group.

\begin{thm}
\label{ds18heis}
Let $E \subset \He$ be a $3$-regular set. Suppose that there exist $\theta>0, C_0>0$ and $\gamma:(0,1] \ra (0,\infty)$ such that for each $x \in E$ and $0 < R \leq \diam_{\He}(E)$ there exists $\tilde{E} \in \text{WGL}(C_0, \gamma)$ satisfying
$$\mathcal{H}^3(E \cap \tilde{E} \cap B(x,R)) \geq \theta R^3.$$
Then $E$ satisfies the WGL.
\end{thm}
As it happens, the proof of Theorem \ref{ds18heis} follows exactly as the  proof of Theorem 1.8 \cite[Part IV, Chapter 1]{DS2}, modulo notational changes. Therefore we skip the details. Now observe that Theorem \ref{wglbp} follows  from Theorem \ref{wglintlg} and Theorem \ref{ds18heis}, since the $3$-regularity and WGL constants (for any fixed $\varepsilon > 0$) of an intrinsic $L$-Lipschitz graph are bounded above by  constants depending only on $\varepsilon$ and $L$.

\end{document}